\documentclass[12pt,reqno,a4paper]{amsart}
\usepackage{
    amsmath,  amsfonts, amssymb,  amsthm,   amscd,
    gensymb,  graphicx, comment,  etoolbox, url,
    booktabs, stackrel, mathtools,enumitem, mathdots,  microtype, lmodern,    mathrsfs, graphicx, tikz,  longtable,tabularx, float, tikz, pst-node, tikz-cd, multirow, tabularx, amscd,  bm, array, makecell, diagbox, booktabs,ragged2e, caption, subcaption}
\usepackage{makecell,slashbox}

\usepackage{comment}

\allowdisplaybreaks

\makeatletter
\def\@tocline#1#2#3#4#5#6#7{\relax
  \ifnum #1>\c@tocdepth % then omit
  \else
    \par \addpenalty\@secpenalty\addvspace{#2}%
    \begingroup \hyphenpenalty\@M
    \@ifempty{#4}{%
      \@tempdima\csname r@tocindent\number#1\endcsname\relax
    }{%
      \@tempdima#4\relax
    }%
    \parindent\z@ \leftskip#3\relax \advance\leftskip\@tempdima\relax
    \rightskip\@pnumwidth plus4em \parfillskip-\@pnumwidth
    #5\leavevmode\hskip-\@tempdima
      \ifcase #1
       \or\or \hskip 1em \or \hskip 2em \else \hskip 3em \fi%
      #6\nobreak\relax
    \hfill\hbox to\@pnumwidth{\@tocpagenum{#7}}\par% <---- \dotfill -> \hfill
    \nobreak
    \endgroup
  \fi}
\makeatother
\usepackage{xcolor}
\usepackage[utf8]{inputenc}
\usepackage{microtype, fullpage, wrapfig,textcomp,mathrsfs,csquotes,fbb}
\usepackage[colorlinks=true, linkcolor=blue, citecolor=blue, urlcolor=blue, breaklinks=true]{hyperref}
\usepackage[capitalise]{cleveref}
\usepackage{todonotes}
\usetikzlibrary{positioning}
\usetikzlibrary{shapes,arrows.meta,calc}
\usetikzlibrary{arrows}

\newtheorem{theorem}{Theorem}[section]

\newtheorem{conjecture}[theorem]{Conjecture}
\newtheorem{corollary}[theorem] {Corollary}
\newtheorem{lemma}[theorem]{Lemma}

\newtheorem{proposition}[theorem]{Proposition}

\newtheorem*{theorem*}{Theorem}

\theoremstyle{definition}
\newtheorem{definition}[theorem]{Definition}
\newtheorem{example}[theorem]{Example}
\newtheorem{remark}[theorem]{Remark}

\newtheorem*{remark*}{Remark}

\newcommand\R{\mathbb{R}}
\newcommand\Z{\mathbb{Z}}
\newcommand\C{\mathbb{C}}
\newcommand\p{\mathfrak{p}}

\newcommand{\TC}{\mathrm{TC}}

\newcommand{\ct}{\mathrm{cat}}

\newcommand{\secat}{\mathrm{secat}}

\newcolumntype{x}[1]{>{\centering\arraybackslash}p{#1}}

\begin{document}

\begin{comment}
    
\author[R. Singh]{Ramandeep Singh Arora}
\address{Department of Mathematics, Indian Institute of Science Education and Research Pune, India}
\email{ramandeepsingh.arora@students.iiserpune.ac.in}
\email{ramandsa@gmail.com}

\author[N. Daundkar]{Navnath Daundkar}
\address{Department of Mathematics, Indian Institute of Science Education and Research Pune, India.}
\email{navnath.daundkar@acads.iiserpune.ac.in}

\end{comment}

\begin{center}

{{\Large \textbf { \sc Equivariant and invariant parametrized topological complexity}}
\\

\medskip

{\sc Ramandeep Singh Arora \footnote{Corresponding Author}}\\
 %{\footnote Corresponding Author}
{\footnotesize Department of Mathematics, Indian Institute of Science Education and Research Pune, Maharashtra, India}\\

{\footnotesize e-mail: {\it ramandeepsingh.arora@students.iiserpune.ac.in}}}\\

\medskip

{\sc Navnath Daundkar}\\
{\footnotesize Department of Mathematics, Indian Institute of Technology Madras, Chennai, Tamil Nadu, India}\\

{{\footnotesize e-mail: {\it navnath@iitm.ac.in}}}\\

\end{center}

\medskip
%\thanks{}
\begin{center}
  {\sc Abstract}\\
  \end{center}
For a $G$-equivariant fibration $p \colon E\to B$, we introduce and study the  invariant analogue of Cohen, Farber and Weinberger's parametrized topological complexity, called the invariant parametrized topological complexity. 
This notion generalizes the invariant topological complexity introduced by Lubawski and Marzantowicz. 
% We establish the equivariant fibrewise homotopy invariance of this notion and derive several bounds, including a cohomological lower bound and a dimensional upper bound.  
% Additionally, we compare invariant parametrized topological complexity with other well-known invariants.  
When $G$ is a compact Lie group acting freely on $E$, we show that the invariant parametrized topological complexity of the $G$-fibration $p \colon E\to B$ coincides with the parametrized topological complexity of the induced fibration $\overline{p} \colon \overline{E} \to \overline{B}$ between the orbit spaces. 
% Finally, we compute the invariant parametrized topological complexity of equivariant Fadell-Neuwirth fibrations, which measures the complexity of motion planning in the presence of obstacles having unknown positions such that the order in which they are placed is irrelevant. 
Furthermore, we compute the invariant parametrized topological complexity of equivariant Fadell–Neuwirth fibrations, which measures the complexity of motion planning in the presence of obstacles with unknown positions, where the order of their placement is irrelevant.
% Apart from this, we establish several bounds, including a cohomological lower bound, an equivariant homotopy dimension-connectivity upper bound and various product inequalities for the equivariant sectional category. 
% Applying them, we obtain some interesting results for equivariant and invariant parametrized topological complexity of a $G$-fibration. %$p\colon E\to B$ coincides with the equivariant $\Delta(E)$-LS category of the fibre product $E\times_B E$ and establish several product inequalities.
In addition, we study the equivariant sectional category and the equivariant parametrized topological complexity, which serve as essential tools for obtaining several results in this paper.

\hrulefill



%\tableofcontents

%------------------------------------------------------------------------------------
\section{Introduction}
\label{sec:intro}

The \emph{topological complexity} of a space $X$, denoted by $\TC(X)$, is defined as the smallest positive integer $k$ such that the product space $X \times X$ can be covered by open sets $\{U_1,\dots, U_k\}$, where each $U_i$ admits a continuous section of the free path space fibration
\begin{equation}
\label{eq: free path space fibration}
    \pi \colon PX\to X\times X \quad \text{defined by} \quad \pi(\gamma)=(\gamma(0),\gamma(1)), 
\end{equation}
where $PX$ denotes the free path space of $X$ equipped with the compact-open topology. 
The concept of topological complexity was introduced by Farber in \cite{FarberTC} to analyze the computational challenges associated with motion planning algorithms for the configuration space $X$ of a mechanical system. 
% Farber further demonstrated that $\TC(X)$ is a homotopy invariant of $X$. 
Over the past two decades, this invariant has attracted significant attention and has been a subject of extensive research.

\subsection*{Parameterized motion planning problem}

Recently, a novel parametrized approach to the theory of motion planning algorithms was introduced in \cite{farber-para-tc, PTCcolfree}. 
This approach provides enhanced universality and flexibility, allowing motion planning algorithms to operate effectively in diverse scenarios by incorporating external conditions. 
These external conditions are treated as parameters and form an integral part of the algorithm's input.
A parametrized motion planning algorithm takes as input a pair of configurations subject to the same external conditions and produces a continuous motion of the system that remains consistent with these external conditions.

We now briefly define the concept of parametrized topological complexity.  
For a fibration $p \colon E \to B$, let $E\times_B E$ denote the fibre product, which is the space of all pair of points in $E$ that lie in a common fibre of $p$. 
Let $E^I_B$ denote the space of all paths in $E$ whose images are contained within a single fibre. 
Define the \emph{parametrized endpoint map} 
\begin{equation}
\label{eq: Pi}
    \Pi \colon E^I_B\to E\times_{B}E \quad \text{by} \quad \Pi(\gamma)=(\gamma(0),\gamma(1)).    
\end{equation}
In \cite{farber-para-tc}, it is shown that $\Pi$ is a fibration.
The \emph{parametrized topological complexity} of a fibration $p \colon E \to B$, denoted by $\TC[p \colon E\to B]$, is the smallest positive integer $k$ such that there is an open cover $\{U_1,\dots, U_k\}$ of $E \times_B E$, where each $U_i$ admits a continuous section of $\Pi$.
For further details and interesting computational results for parametrized topological complexity, see \cite{farber-para-tc}, \cite{PTCcolfree}, \cite{ptcspherebundles} and \cite{minowa2024parametrized}.
Additionally, the concept has been extended to fibrewise spaces by Garc\'{\i}a-Calcines in \cite{fibrewise}. On the other hand, Crabb \cite{crabb2023fibrewise} established some computational results in the fibrewise setting.

One of the key motivations for introducing this concept was to address the challenge of collision-free motion planning in environments where obstacles have unknown positions in advance.  
This can be described by the following scenario, as illustrated in \cite{farber-para-tc}: A military commander oversees a fleet of $t$ submarines navigating waters with $s$ mines. 
The positions of these mines change every 24 hours. 
Each day, the commander must determine a motion plan for each submarine, ensuring that they travel from their current locations to their designated destinations without colliding with either the mines or other submarines.
A parametrized motion planning algorithm will take as input the positions of the mines on the given day and the current and the desired positions of the submarines and will produce as output a collision-free motion of a fleet. 
Hence, the complexity of the universal motion planning algorithm in this setting can be described as the parametrized topological complexity of the Fadell-Neuwirth fibration
$$
    p \colon F(\R^d,s+t) \to F(\R^d, s), \quad (x_1,\dots,x_s,y_1,\dots,y_t) \mapsto (x_1,\dots,x_s)
$$
where $F(\R^d,s)$ is the configuration space of $s$ distinct points lying in $\R^d$, see \Cref{sec: examples}.
% where $F(\R^d, s) = \{(x_1,\dots,x_s \in (\R^d)^s \mid x_i \neq x_j \text{ for } i \neq j\}$.
 
However, in a real-life scenario, the specific order in which the mines are placed should be irrelevant.
For the two configurations of mines,  
$$
    (x_1, \dots, x_s) \quad \text{ and } \quad (x_{\sigma(1)},\dots,x_{\sigma(s)}),
$$
for any $\sigma$ in the permutation group $\Sigma_s$, the military commander should assign the same motion plan for the submarines.
This is because both configurations describe the mines being placed at the same set of positions, regardless of their labeling. 
Thus, we should consider the unordered configuration space 
$
    F(\R^d,s)/\Sigma_s
$
for the placement of mines. 
Hence, in this new perspective, the complexity of the universal motion planning algorithm should be described as the parametrized topological complexity of the fibration
$$
    \overline{p} \colon \overline{F(\R^d,s+t)} \to \overline{F(\R^d, s)},
$$
which is induced from $p$ by taking the quotient under the natural action of $\Sigma_s$ on the configuration spaces. 
% In this paper, we introduce the notion of invariant parametrized topological complexity for a $G$-fibration $p\colon E\to B$, denoted by $\TC^G[p \colon E\to B]$, to measure the complexity of parametrized motion planning problem where the order in which the mines are placed is irrelevant.

In this paper, we introduce the notion of invariant parametrized topological complexity for a $G$-fibration $p\colon E\to B$, denoted by $\TC^G[p \colon E\to B]$, which seems to provide a more suitable framework for analyzing the impact of symmetries on parametrized motion planning algorithms.
% $p \colon E \to B$ to estimate the parametrized topological complexity of the induced fibration $\overline{p} \colon \overline{E} \to \overline{B}$ between the orbit spaces.
The invariant parametrized topological complexity is a parametrized analogue of the invariant topological complexity, which was introduced by Lubawski and Marzantowicz \cite{invarianttc}. 
The invariant topological complexity for a $G$-space $X$, denoted by $\TC^G(X)$, behaves well with respect to quotients. 
In particular, if a compact Lie group $G$ acts freely on $X$, then the equality $\TC^G(X) = \TC(X/G)$ holds (see \cite[Theorem 3.10]{invarianttc}). 
Generalizing this to the parameterized setting we establish the following theorem.

\begin{theorem*}
% \label{thm: inv-para-tc under free action}
Suppose $G$ is a compact Lie group. 
Let $p \colon E \to B$ be a $G$-fibration and let $\overline{p} \colon \overline{E} \to \overline{B}$ be the induced fibration between the orbit spaces. 
If the $G$-action on $E$ is free and $\overline{E} \times \overline{E}$ is hereditary paracompact, then
    \[
     \TC^{G}[p \colon E \rightarrow B]
       = \TC[\overline{p} \colon \overline{E} \rightarrow \overline{B}]. 
    \]
% Is the product of hereditary paracompact spaces, hereditary paracompact?
\end{theorem*}

% Lubawski and Marzantowicz \cite{invarianttc} introduced the notion of \emph{invariant topological complexity}. 
% The \emph{invariant topological complexity} of a $G$-space $X$, denoted by $\TC^G(X)$, is defined as the smallest positive integer $k$ such that the product space $X\times X$ can be covered by $(G\times G)$-invariant open sets $\{U_1,\dots, U_k\}$, where each $U_i$ is $(G\times G)$-compressible (see \Cref{def:G-compress}) into the saturated diagonal $\mathbb{k}(X):=(G\times G)\cdot\Delta(X)$.
% They demonstrated that, in certain contexts, it is more suitable than the equivariant topological complexity.
% For instance, for a free $G$-space $X$, the equality $\TC^G(X) = \TC(X/G)$ holds (see \cite[Theorem 3.10]{invarianttc}).

\subsection*{Outline of the paper}

The aim of this paper is twofold. First, we examine various properties of the equivariant sectional category and equivariant parametrized topological complexity. Using these properties, we develop and analyze the new concept of invariant parametrized topological complexity, which we introduce in \Cref{sec: inv-para-tc}.

In \Cref{subsec: eq-fibrations}, we introduce the concept of equivariant fibrations and present several examples, including equivariant covering maps.

In \Cref{subsec: eq-secat}, we study the equivariant sectional category of a $G$-fibration, and establish multiple lower bounds in \Cref{thm: cohomological lower bound for eq-secat}, \Cref{prop: secat(overline(p)) leq secat_G(p)} and \Cref{prop: fixed-point-and-subgroup-ineq for equi-secat}. 
We also provide an equivariant homotopy dimension-connectivity upper bound in \Cref{thm: G-homo-dim ub on equi-secat}. 
Afterwards, we establish product inequalities in \Cref{prop: special-prod-ineq for equi-secat} and \Cref{cor: equi-secat of pullback fibration under diagonal map}. 

In \Cref{subsec:eq-LS-category}, we recall the notion of the equivariant LS category of a $G$-space. 
In this section, we establish a lower bound in terms of fixed point sets, and provide an equivariant homotopy dimension-connectivity upper bound, as stated in \Cref{prop: subgroup-ineq for equi-cat} and \Cref{thm: G-hom-dim on equi-cat}, respectively. 

Subsequently, \Cref{subsec:invtc} devoted to the equivariant and invariant topological complexity of a $G$-space, and we provide an equivariant homotopy dimension-connectivity upper bound for the former in \Cref{thm: G-hom-dim ub for equi-tc}.

In \Cref{sec: equi-para-tc}, we explore various properties of the equivariant parametrized topological complexity of $G$-fibrations $p\colon E\to B$. Our main result \Cref{thm: equivalent defn of a section of equi-para-tc}, characterizes the elements of a parametrized motion planning cover as the $G$-compressible subsets of the fibre product $E\times_B E$ into the diagonal $\Delta(E)$. Furthermore, we establish some lower bounds and the product inequalities in \Cref{prop: subgroup-ineq for equi-para-tc},  \Cref{thm: non-equi coho-lower-bound for equi-para-tc} and \Cref{thm: prod-ineq for equi-para-tc}, respectively.

In \Cref{sec: inv-para-tc}, we introduce the notion of invariant parametrized topological complexity for $G$-fibrations. 
% For a $G$-fibration, we define an invariant version of the parametrized endpoint map (see \Cref{eq:inv-parametrized-fibration}). 
% We show that this $(G\times G)$-map is, in fact, a $(G\times G)$-fibration (see \Cref{prop:Pi-is-G2-fibration}), which allow us to define the invariant parametrized topological complexity (see \Cref{def:inv-para-tc}). 
% This invariant is denoted by $\TC^G[p\colon E \to B]$.
We establish the fibrewise $G$-homotopy invariance of this notion, and show that it generalizes both the parametrized and invariant topological complexity; see \Cref{thm:G-homotopy-invariace-of-inv-para-tc} and \Cref{prop: special cases of inv-para-tc}, respectively. 
% Further, we show that the invariant parametrized topological complexity of a trivial $G$-fibration with $G$ acting trivially on its fibre coincides with the topological complexity of the fibre (see \Cref{prop: inv-para-TC for trivial-G-fib}).
For a $G$-fibration $p\colon E\to B$, in \Cref{thm: equivalent defn of a section of in-para-tc}, we show that the elements of an invariant parametrized motion planning cover can be characterized as the $(G\times G)$-compressible subsets of the fibre product $E\times_{B/G} E$ into the saturated diagonal $\mathbb{k}(E) = E \times_{E/G} E$. 
% As a consequence to this characterization, \Cref{cor:inv-para-tc-as-A-cat} establishes that, $\TC^G[p\colon E\to B]$ can be expressed as the equivariant $\mathbb{k}(E)$-LS category of  $E\times_{B/G} E$.
In \Cref{subsec:prop-bounds}, we investigate various properties and bounds for $\TC^G[p\colon E \to B]$. 
For example, we establish inequality under pullbacks (\Cref{prop: pullback-ineq for inv-para-tc}), dimensional upper bound (\Cref{prop: inva-para-tc bounds}), lower bound (\Cref{prop: subgroup-ineq for inv-para-tc}), cohomological lower bounds (\Cref{thm: cohomological lower bound for inv-para-tc} and \Cref{thm: non-equi coho-lower-bound for inv-para-tc}), and product inequality (\Cref{thm: prod-ineq for inv-para-tc}). 
Finally, we  prove one of our main results, \Cref{thm: inv-para-tc under free action}, which shows that the $\TC^G[p\colon E\to B]$ coincides with the parametrized topological complexity of the corresponding orbit fibration, when $G$ acts freely on $E$.
We conclude the section with some examples illustrating the application of \Cref{thm: inv-para-tc under free action}.

In \Cref{sec: examples}, we compute the invariant parametrized topological complexity of the equivariant Fadell-Neuwirth fibrations. 
Specifically, in \Cref{thm:estimates-for-FN-fibrations} and \Cref{thm:estimates-for-FN-fibrations for d=2}, we establish:

% for configuration spaces of odd-dimensional Euclidean spaces and the Euclidean plane, we compute the exact value of invariant parametrized topological complexity.
% For configuration spaces of even-dimensional Euclidean spaces of dimension greater than 2, we show that the upper and lower bounds differ by $1$.

\begin{theorem*}
% \label{thm:estimates-for-FN-fibrations}
Suppose $s\geq 2$, $t\geq 1$ and $d \geq 3$. Then
$$
\TC^{\Sigma_s}[p \colon F(\R^d,s+t)\to F(\R^d, s)]=
\begin{cases}
   2t+s, & \text{if $d$ is odd},\\
   \text{either } 2t+s-1 \text{ or } 2t+s & \text{if $d$ is even}.
\end{cases}   
$$
\end{theorem*}

\begin{theorem*}
% \label{thm:estimates-for-FN-fibrations}
Suppose $s\geq 2$ and $t\geq 2$. Then
$$
    \TC^{\Sigma_s}[p \colon F(\R^2,s+t)\to F(\R^2, s)] = 2t+s-1.
$$
\end{theorem*}

\vspace{3mm}
\subsection*{Notations and conventions} Throughout the text, $G$ denotes a general topological group acting on Hausdorff spaces, unless stated otherwise. 
We adopt standard terminology and notation from equivariant topology, such as $G$-spaces, $G$-maps, $G$-homotopies and related notions.
In this paper, all category-type invariants are taken to be un-normalized, that is, they equal the number of open sets in the cover. 
Thus, our definitions exceed by one those in \cite{farber-para-tc, PTCcolfree}, but agree with those used in \cite{FarberTC, colmangranteqtc, invarianttc}.

%-------------------------------------------------------------------------------------------------
\section{Preliminaries}
\label{sec:preliminaries} 

In this section, we systematically introduce and study various numerical invariants: equivariant sectional category, equivariant LS-category, equivariant topological complexity, $A$-Lusternik-Schnirelmann $G$-category, and invariant topological complexity.

% NOTE INTRODUCTION TO THIS SECTION BELOW IS CONTAINED INSIDE THE SECTION (IN PARTS) AND IS ALSO THERE IN THE INTRODUCTION SECTION. SO IT IS BETTER TO REMOVE IT REDUCE PAGE LENGTH.

% In \Cref{subsec: eq-secat}, we first establish a cohomological lower bound on the equivariant sectional category of a $G$-map $p \colon E \to B$ using Borel cohomology.
% When $G$ is a compact Hausdorff topological group, we further show that the equivariant sectional category of $p$ is bounded below by the sectional category of $p$ and the sectional cateogry of the induced fibration $\overline{p} \colon  \overline{E} \to \overline{B}$ between the orbit spaces. 
% We then define the notion of $G$-homotopy dimension for $G$-$\mathrm{CW}$-complexes, and establish a $G$-homotopy dimension-connectivity upper bound for equivariant sectional category. 
% Additionally, we prove some properties of the equivariant sectional category.
% As a consequence to \Cref{thm: G-homo-dim ub on equi-secat}, we obtain a $G$-homotopy dimension-connectivity upper bound for equivariant LS category in \Cref{subsec:eq-LS-category}. 
% Later in \Cref{subsec:invtc} we define the invariant topological complexity and recall the basic information related to Clapp-Puppe invariant of LS type in \Cref{subsec:Clapp-Puppe}.

\subsection{$G$-Fibrations}
\hfill\\ \vspace{-0.7em}
\label{subsec: eq-fibrations}

% For a topological group $G$, 
We begin by recalling the definition of a $G$-fibration, followed by a few examples. 
For a more detailed discussion, we refer the reader to \cite{gevorgyan2023equivariant} and \cite[Section 2]{grant2019symmetrized}.

\begin{definition}
A $G$-map $p \colon E \to B$ is called a $G$-fibration if it has the $G$-homotopy lifting property with respect to any $G$-space $X$.
More precisely, if $H \colon X \times I \to B$ is a $G$-homotopy and $h \colon X \to E$ is a $G$-map with $p \circ h = H_0$, then there exists a $G$-homotopy $\widetilde{H} \colon X \times I \to E$ that satisfies $p \circ \widetilde{H} = H$ and $\widetilde{H}_0 = h$.
\end{definition}

\begin{example}
\label{example: G-fibrations}
    Here we list some examples of $G$-fibrations.
    \begin{enumerate}
        \item For any $G$-spaces $B$ and $F$, the projection maps $\pi_1 \colon B \times F \rightarrow B$ and $\pi_2 \colon B \times F \to F$ are $G$-fibrations. 
        % For any space $X$, the projection maps $\pi_1 \colon X \times G \to X$ and $\pi_2 \colon X \times G \to G$ are $G$-fibrations, where $X \times G$ is equipped with $G$-action given by $g' \cdot (x,g) := (x,g'g)$ for $x \in X$ and $g,g' \in G$.

        \item If $G$ is a compact Hausdorff topological group, then every principal $G$-bundle over a paracompact space is a $G$-fibration, where $G$ acts trivially on the base. 
        This can be shown using \cite[Theorem 7]{gevorgyan2023equivariant}. 
        In particular, the Hopf fibration $S^1 \hookrightarrow S^3 \to S^2$ is a $S^1$-fibration.
        
        \item For any space $X$, the free path space fibration $\pi \colon PX \to X \times X$ is a $\mathbb{Z}_2$-fibration, with $\mathbb{Z}_2$-action on $PX$ given by reversal of paths and on $X \times X$ by transposition of factors, see \cite[Example 2.6]{grant2019symmetrized}.

        \item For any $G$-space $X$, the free path space fibration is a $G$-fibration. 
        Moreover, if $x_0 \in X$ is a fixed point under $G$-action, then the path space fibration is also a $G$-fibration. 
        See \Cref{subsec:invtc} and \Cref{subsec:eq-LS-category} respectively for more details.
    \end{enumerate}
\end{example}

The following theorem provides a sufficient condition for a fibration to be a $G$-fibration. 
We note that although the proof is presented for compact Hausdorff groups, it remains valid for arbitrary topological groups. 

\begin{theorem}[{\cite[Lemma 5]{gevorgyan2023equivariant}}]
\label{thm: fibration + unique path lifting implies G-fibration}
If $p \colon E \to B$ is a $G$-map that is a fibration with the unique path-lifting property, then $p$ is a $G$-fibration.
\end{theorem}

The following is an immediate corollary of this theorem.

\begin{corollary}
\label{cor: G-map + covering map implies G-fibration}
If $p \colon E \to B$ is a $G$-map that is a covering map, then $p$ is a $G$-fibration.
In particular, if $p \colon \widetilde{X} \to X$ is a universal covering, then $p$ is a $\pi_1(X)$-fibration, where $\pi_1(X)$ acts on $\widetilde{X}$ via deck transformations and acts trivially on $X$.
\end{corollary}

\begin{example} Here we list some non-trivial examples of $G$-fibrations. 
    \begin{enumerate}
        % \item The universal covering $p \colon \R \to S^1$, given by $t \mapsto e^{2\pi i t}$, is a $\Z$-fibration, where $\Z$ acts on $\R$ by translation and acts trivially on $S^1$.

        \item Suppose $\alpha \in \R$.
        Then the universal covering $p \colon \R \to S^1$, given by $t \mapsto e^{2\pi i t}$, is a $\Z$-fibration, where $\Z$ acts on $\R$ by $n \cdot t = t + n \alpha$, and acts on $S^1$ by $n \cdot z = e^{2 \pi i n \alpha}z$.

        \item The universal covering $p \colon S^n \to \R P^n$ is a $\Sigma_{n+1}$-fibration, where the symmetric group $\Sigma_{n+1}$ acts on $S^n$ by
        $
            \sigma \cdot (x_0 , \dots , x_n) 
                = (x_{\sigma(0)} , \dots , x_{\sigma(n)
                }),
        $ 
        and acts on $\R P^n$ by
        $
            \sigma \cdot [x_0 : \dots : x_n] 
                = [x_{\sigma(0)} : \dots : x_{\sigma(n)
                }].
        $
    \end{enumerate}
\end{example}

%----------------------------------------------------------------
\subsection{Equivariant sectional category} 
\hfill\\ \vspace{-0.7em}
\label{subsec: eq-secat}

Schwarz \cite{Sva} introduced and studied the notion of sectional category of a fibration, and later by Berstein and Ganea in \cite{secat} for any map. 
The corresponding equivariant analogue was introduced by Colman and Grant in \cite{colmangranteqtc}.
% The notion of equivariant sectional category was introduced and studied by Colman and Grant in \cite{colmangranteqtc}. 
% We will now briefly recall its definition and present a cohomological lower bound.

\begin{definition}[{\cite[Definition 4.1]{colmangranteqtc}}\label{def:eqsecat}]
Let $p \colon E\to B$ be a $G$-map. 
The equivariant sectional category of $p$, denoted by $\secat_G(p)$, is the least positive integer $k$ such that there is a $G$-invariant open cover $\{U_1, \ldots, U_k\}$ of $B$ and $G$-maps $s_i \colon U_i \to E$, for $i=1, \ldots, k$, such that $p \circ s_i \simeq_G i_{U_i}$, where $i_{U_i} \colon U_i \hookrightarrow B$ is the inclusion map.
\end{definition}

First we establish a cohomological lower bound on the equivariant sectional category of a $G$-map using Borel cohomology.
This is based on Colman and Grant's work \cite[Theorem 5.15]{colmangranteqtc}, which provides a similar cohomological lower bound on equivariant topological complexity.
To the best of our knowledge, such a bound has not been documented in the literature. 
We believe that this result must already be known to experts in the field. 
Nevertheless, we provide a thorough proof of this result here.

Suppose $EG \to BG$ is a universal principal $G$-bundle. 
For a $G$-space $X$, let $X^h_{G}$ be the homotopy orbit space of $X$ defined as
$$
    X^h_{G} 
        := EG \times_{G} X 
        = \frac{EG \times X}{(e g, x)\sim (e,g^{-1} x)}, \quad \text{for }e \in EG, g \in G, x \in X
$$
and the Borel $G$-equivariant cohomology $H^{*}_{G}(X;R)$ of $X$ with coefficients in a commutative ring $R$ is defined as $H^{*}_{G}(X;R) := H^{*}(X^h_{G};R)$. 
We note that for a $G$-map $p \colon E \to B$, there is an induced map $p^h_G \colon E^h_G \to B^h_G$.

\begin{theorem}[Cohomological lower bound]
\label{thm: cohomological lower bound for eq-secat}
    Suppose $p \colon E \to B$ is a $G$-map. 
    If there are cohomology classes $u_1,\dots,u_k \in \widetilde{H}^{*}_{G}(B;R)$ (for any commutative ring $R$) with 
    $$
        (p^h_G)^*(u_1) = \cdots = (p^h_G)^{*}(u_k) = 0 \quad \text{and} \quad u_1 \smile \dots \smile u_k \neq 0,
    $$
    then $\secat_G(p) > k$.
\end{theorem}

\begin{proof}
Suppose $\secat_G(p) \leq k$. 
Then there exists a $G$-invariant open cover $\{U_1, \dots, U_k\}$ of $B$ such that each $U_i$ admits a $G$-equivariant homotopy section $s_i$ of $p$. 
% Suppose there are cohomology classes $u_1,\dots,u_k \in \widetilde{H}^{*}_{G}(B;R)$ such that $(p^h_G)^*(u_1) = \cdots = (p^h_G)^{*}(u_k) = 0$.
If $j_i \colon U_i \hookrightarrow B$ is the inclusion map, then 
$
    ((j_i)^{h}_G)^*(u_i) = ((s_i)^{h}_G)^*\left((p^{h}_G)^*(u_i)\right) = 0.
$
% Let $j_i \colon U_i \hookrightarrow B$ be the inclusion map. Then 
% $$
%     ((j_i)^{h}_G)^*(u_i) = ((s_i)^{h}_G)^*\left((p^{h}_G)^*(u_i)\right) = 0
% $$
% since $p \circ s_i \simeq_G j_i$ implies $((j_i)^{h}_G)^* = ((s_i)^{h}_G)^* \circ (p^{h}_G)^*$. 
Hence, the long exact sequence in cohomology of the pair $(B^h_G,(U_i)^h_G)$ yields an element $v_i \in H^{*}(B^h_G,(U_i)^h_G;R)$ such that $((q_i)^{h}_G)^*(v_i)=u_i$, where $q_i \colon B \hookrightarrow (B,U_i)$ is the inclusion map. 
Consequently, we obtain
    $$
        v_1 \smile \cdots \smile v_k 
            \in H^*(B^h_G,\cup_{i=1}(U_i)^h_G;R) 
            = H^*(B^h_G,B^h_G;R) = 0.
    $$
Therefore, by the naturality of the cup product, we have $u_1 \smile \cdots \smile u_k = (q^h_G)^*(v_1 \smile \cdots \smile v_k) = 0$, where $q \colon B \hookrightarrow (B,B)$ is the inclusion map.
% Hence, $u_1 \smile \cdots \smile u_k = 0$.
\end{proof}

% \begin{remark}\
% \begin{enumerate}
% \item Observe that if $G$  acts trivially on $X$, then the lower bound in \Cref{thm: cohomological lower bound for eq-secat} recovers the cohomological lower bound given by Schwarz in \cite[Theorem 4]{Sva}.

% \item Note the following commutative diagram of $G$-maps 
% \[
% \begin{tikzcd}
% X \arrow[dr, "\Delta"'] \arrow{rr}{h}
% & & PX \arrow{dl}{\pi} \\
% & X \times X,
% \end{tikzcd}
% \] 
% where $h$ is a $G$-homotopy equivalence. 
% Then, the lower bound in \Cref{thm: cohomological lower bound for eq-secat} recovers the bound on $\TC_G(X)$ that was obtained by Colman and Grant in \cite[Theorem 5.15]{colmangranteqtc}.
% More generally, it also recovers the bound \cite[Theorem 4.25]{D-EqPTC} on the equivariant parametrized topological complexity, which was obtained by the second author.
% \end{enumerate}  

\begin{remark}
Observe that if $G$ acts trivially on $X$, then the lower bound in \Cref{thm: cohomological lower bound for eq-secat} coincides with the cohomological lower bound given by Schwarz in \cite[Theorem 4]{Sva}. 
Moreover, it generalizes the lower bounds for equivariant topological complexity and equivariant parametrized topological complexity previously established by Colman–Grant \cite[Theorem 5.15]{colmangranteqtc} and Daundkar \cite[Theorem 4.25]{D-EqPTC}, respectively.  
\end{remark} 

In practice, however, the difficulty of computing cup products in Borel cohomology (or more generally, in any equivariant cohomology) makes the problem cumbersome. 
We can then ask whether non-equivariant cohomological bounds can be utilized in some way. 
When $G$ is a compact Hausdorff topological group and $p \colon E \to B$ is a $G$-fibration, we will show that the sectional category of $p$ and the sectional category of the induced fibration $\overline{p} \colon \overline{E} \to \overline{B}$ between the orbit spaces are lower bounds for the equivariant sectional category of $p$. 
Note that $\overline{p}$ fits into the commutative diagram
\begin{equation}
\label{diag: orbit-map comm-diag}
    % https://tikzcd.yichuanshen.de/#N4Igdg9gJgpgziAXAbVABwnAlgFyxMJZABgBpiBdUkANwEMAbAVxiRAFEQBfU9TXfIRQBGclVqMWbAELdeIDNjwEiZYePrNWiEAB1dEGjABODLGBjB2XOXyWCio9dU1Sd+wybMXg0m13EYKABzeCJQADNjCABbJDIQHAgkACYXSW09XTQsAH1OagY6ACMYBgAFfmUhEGMsYIALHFsQKNj46iSkUQktNjQWtrjEHq7EAGZ0vvdsvNkeSOjhtMTkiam3LM9Tc0s0fwouIA
    \begin{tikzcd}
    E \arrow[d, "\pi_E"'] \arrow[r, "p"]   & B \arrow[d, "\pi_B"] \\
    \overline{E} \arrow[r, "\overline{p}"] & \overline{B} ,       
    \end{tikzcd}
\end{equation}
where $\pi_B \colon B \to \overline{B}$ and $\pi_E \colon E \to \overline{E}$ are orbit maps.

\begin{proposition} 
\label{prop: secat(overline(p)) leq secat_G(p)}
Let $p \colon E \to B$ be a $G$-fibration.
% such that the induced map $\overline{p} \colon \overline{E} \to \overline{B}$ between orbit spaces is a fibration. 
Then
$
    \secat(\overline{p}) \leq \secat_G(p).
$
\end{proposition}

\begin{proof}
    Suppose $U$ is a $G$-invariant open subset of $B$ with a $G$-equivariant section $s$ of $p$ over $U$. 
    As the orbit map $\pi_B \colon B \to \overline{B}$ is open, we have $\overline{U} := \pi_B(U)$ is an open subset of $\overline{B}$. 
    As $U$ is $G$-invariant, it follows $U$ is saturated with respect to $\pi_B$. 
    Hence, $\pi_B \colon U \to \overline{U}$ is a quotient map.
    Then, by universal property of quotient maps, there exists a unique continuous map $\overline{s} \colon \overline{U} \to \overline{E}$ such that the following diagram 
    \[
    % https://tikzcd.yichuanshen.de/#N4Igdg9gJgpgziAXAbVABwnAlgFyxMJZABgBpiBdUkANwEMAbAVxiRAFUQBfU9TXfIRRkAjFVqMWbADrSINGACcGWMDGDsu3XiAzY8BIiPLj6zVohCz5SlWuABRLV3EwoAc3hFQAM0UQAWyQyEBwIJAAmajMpS1k0LAB9BwACWQBjLEV0lIQeX38gxBCwpGMJcxlpBMSAIRBqBjoAIxgGAAV+AyEQRSx3AAscbQLAsupSxCiQOwsQKDo4AbcGitirOQVlVXU4LUaWts79QTY+weGXLiA
\begin{tikzcd}
U \arrow[r, "\pi_E \circ s"] \arrow[d, "\pi_B"'] & \overline{E} \\
\overline{U} \arrow[ru, "\overline{s}"', dashed] &             
\end{tikzcd}
\]
    commutes. 
    Then 
    $$
    \overline{p} ( \overline{s} (\overline{b}))
        = \overline{p} (\overline{s} (\pi_B(b)))
        = \overline{p}(\pi_E(s(b))) 
        = \pi_B (p(s(b))) 
        = \pi_B(b)
        = \overline{b}
    $$
    implies $\overline{s}$ is a section of $\overline{p}$ over $\overline{U}$. Thus, the result follows since $\pi_B \colon B \to \overline{B}$ is surjective.
\end{proof}

\begin{theorem}
\label{thm: secat(overline(p)) = secat_G(p) for free actions}
    Suppose $G$ is a compact Hausdorff topological group and $p \colon E \to B$ is a $G$-fibration.  
    Then $\overline{p} \colon \overline{E} \to \overline{B}$ is a fibration.
    Furthermore, if 
        % $E$ and $B$ are Hausdorff, and 
        % $G$ is a compact Hausdorff topological group such that 
        $G$ acts freely on $B$, then 
        $$
            \secat(\overline{p}) = \secat_G(p).
        $$
\end{theorem}

\begin{comment}
THIS WAS EARLIER ATTEMPT AT THE PROOF.

Then $U$ is $G$-invariant and $\overline{U}$ is hereditary paracompact. 
Consider the homotopy $\overline{S} \colon \overline{V} \times I \to \overline{E}$ given by $\overline{S}(\overline{e},t) := \overline{s}(\overline{p}(\overline{e}))$. 
Then the following diagram
\[
% https://tikzcd.yichuanshen.de/#N4Igdg9gJgpgziAXAbVABwnAlgFyxMJZABgBpiBdUkANwEMAbAVxiRADUACAHW7wFt4PbsGK8AviHGl0mXPkIoyARiq1GLNl14ChASSkyQGbHgJEATKVXV6zVohC8INGACcGWMDGDtxw3ThOA2lZUwUiAGZrNTtNR2dXDy8fAFFJUOM5M0VkaMpbDQcQVKk1GCgAc3giUAAzNwh+JDIQHAgkZWocOiwGNgALCAgAa0N6xubELraOxCt1ezZeNCwAfVSArEEg3n46HAG3fmAsKHE1kKMGpqQF9qRoxfinbhd3T29gAGUM68mWt05gAWbq9fqOIajcYgG5TUGzR6FJYJbirDZlcRAA
\begin{tikzcd}
V \times \{0\} \arrow[d, hook] \arrow[rrr, hook]    &  &                                                 & E \arrow[d, "\pi_E"] \\
V \times I \arrow[rr, "\pi_E \times \mathrm{id}_I"] &  & \overline{V} \times I \arrow[r, "\overline{S}"] & \overline{E}        
\end{tikzcd}
\]
commutes.
As the $G$-action on $E$ and $B$ is free, it follows the action of $G \times G$ on $E \times_{E/G} E$ is free. 
Hence, by the Covering Homotopy Theorem of Palais \cite[Theorem II.7.3]{bredon-transformation-groups}, it follows there exists a $G$-homotopy $S \colon U \times I \to E$ such that $H_0 = i_{U} \colon U \hookrightarrow E \times_{B/G} E$ and $(\pi_E \times \pi_E) \circ H = \overline{H} \circ ((\pi_E \times \pi_E) \times \mathrm{id}_I)$. 
\end{comment}
    
\begin{proof}
    By \cite[Corollary 2]{gevorgyan2023equivariant}, it follows that $\overline{p}$ is a fibration. 
    Note that the inequality $\secat(\overline{p}) \leq \secat_G(p)$ follows from \Cref{prop: secat(overline(p)) leq secat_G(p)}. 
    Now we will show the reverse inequality. 
    Note that $G$ acts freely on $E$ as well, since $G$ acts freely on $B$ and $p$ is a $G$-map. 
    If $\phi \colon E \to \overline{E} \times_{\overline{B}} B$ denotes the natural $G$-map induced by the universal property of the pullback, then $\phi$ preserves the orbit structure because $G$ acts freely on both $B$ and $E$.
    % As the action of $G$ on $B$ as well as $E$ is free, the map $\overline{p}$ preserves the orbit structure.
    Hence, the diagram \eqref{diag: orbit-map comm-diag} is a pullback in the category of $G$-spaces, see the proof of \cite[Theorem II.7.3]{bredon-transformation-groups}. 
    Suppose $\overline{U}$ is an open subset of $\overline{B}$ and $\overline{s} \colon \overline{U} \to \overline{E}$ is a section of $\overline{p}$. 
    Let $U=\pi_B^{-1}(\overline{U})$.
    Then, by the universal property of pullbacks, there exists a unique $G$-map $s \colon U \to E$ such that $\pi_E \circ s = \overline{s} \circ \pi_B \colon U \to \overline{E}$ and $p \circ s = i_U \colon U \to B$. 
    Hence, $\secat_G(p) \leq \secat(\overline{p})$.
\end{proof}

\begin{remark}
    As mentioned in \Cref{example: G-fibrations} (3), the free path space fibration $\pi \colon PX \to X \times X$ is a $\mathbb{Z}_2$-fibration for any space $X$.
    The $\mathbb{Z}_2$-action on $PX$ is given by reversal of paths, and on $X \times X$ it is given by transposition of factors. %, see \cite[Example 2.6]{grant2019symmetrized}.
    Hence, by \Cref{thm: secat(overline(p)) = secat_G(p) for free actions}, we get $\secat(\overline{\pi}) \leq \secat_{\mathbb{Z}_2}(\pi)$. 
    Thus, we recover the cohomological lower bound on the symmetrized topological complexity  $\TC^{\Sigma}(X)$ in \cite[Theorem 4.6]{grant2019symmetrized}, since the following commutative diagram
    \[
    % https://tikzcd.yichuanshen.de/#N4Igdg9gJgpgziAXAbVABwnAlgFyxMJZAJgBoAGAXVJADcBDAGwFcYkQAdDiWmAJ0ZYwMYAAUAGgF8Qk0uky58hFOQrU6TVuy49+g4cCky5IDNjwEiARlJX1DFm0SduvAUJHiABFzwBbeC8jSXUYKABzeCJQADM+CD8kGxAcCCRVDUdtDmwAgEdjWPjExGTUpDJMrWcdN30RLgARGEYcemkaRnoAIxbRBQtlED4scIALHEKQOIT0mnLESodql113Ay40LGkQySA
\begin{tikzcd}
\overline{X} \arrow[rr, "\simeq"] \arrow[rd, "\overline{\Delta}"'] &                       & \overline{PX} \arrow[ld, "\overline{\pi}"] \\
                                                                   & \overline{X \times X} &                                           
\end{tikzcd}
    \]
    implies the nilpotency of the kernel of $\overline{\Delta}$ and $\overline{\pi}$ are the same.
\end{remark}

Suppose $X$ is a $G$-space. For a subgroup $H$ of $G$, define the $H$-fixed subspace of $X$ as
$$
    X^H := \{x \in X \mid h \cdot x = x \text{ for all } h \in H\}.
$$

\begin{comment}
\begin{proposition}
\label{prop: subgroup-ineq for equi-secat}
Suppose $p \colon E \to B$ is a $G$-fibration. 
If $H$ and $K$ are subgroups of $G$ such that $E^H$ and $B^H$ are $K$-invariant, and the fixed point map $p^H \colon E^H \to B^H$ is a $K$-map, then
$$
    \secat_K(p^H) \leq \secat_G(p).
$$
\end{proposition}

\begin{proof}
Suppose $s \colon U \to E$ is a $G$-equivariant section of $p$. 
Define $V = U \cap B^H$. 
Note that $V$ is the set of $H$-fixed points of $U$ and is $K$-invariant.
Since $s$ is $G$-equivariant, it takes $H$-fixed points to $H$-fixed points, and hence restricts to a $K$-equivariant map $\left.s\right |_{V} \colon V \to E^H$. 
It is clear that $\left.s\right|_{V}$ is a section of $p^H$.
\end{proof}

\begin{proof}
\textcolor{red}{We can say something like as follows instead just writing everything in details}
The desired inequality follows by the restricting the $G$-equivariant sections of $p$ over $G$-invariant open sets $U$ to $V = U \cap B^{H}$, for subgroups $H$ of $G$.    

If $s \colon U \to E$ is a $G$-equivariant section of $p$, then $\left.s\right |_{V} \colon V \to E^H$ is a $K$-equivariant section of $p^H$, where $V=U \cap B^H$. Hence, the desired inequality follows.
\end{proof}
\end{comment}

\begin{proposition}
\label{prop: fixed-point-and-subgroup-ineq for equi-secat}
Suppose $p \colon E \to B$ is a $G$-fibration. 
If $H$ and $K$ are subgroups of $G$ such that $E^H$ and $B^H$ are $K$-invariant, and the fixed point map $p^H \colon E^H \to B^H$ is a $K$-map, then
\begin{equation}
\label{eq: subgroup-ineq for equi-secat}
    \secat_K(p^H) \leq \secat_G(p).
\end{equation}
In particular, if $G$ is a compact Hausdorff topological group, then
\begin{enumerate}
\item the fixed point map $p^H \colon E^H \to B^H$ is a fibration for all closed subgroups $H$ of $G$, and
$$
    \secat(p^H) \leq \secat_G(p).
$$
\item $p \colon E \to B$ is a $K$-fibration for all closed subgroups $K$ of $G$, and 
$$
    \secat(p) \leq \secat_K(p) \leq \secat_G(p).
$$
\end{enumerate}
\end{proposition}

\begin{proof}
If $s \colon U \to E$ is a $G$-equivariant section of $p$, then $\left.s\right |_{V} \colon V \to E^H$ is a $K$-equivariant section of $p^H$, where $V=U \cap B^H$. Hence, the inequality \eqref{eq: subgroup-ineq for equi-secat} follows.

If $G$ is compact Hausdorff, then it follows from \cite[Theorem 4 and Theorem 3]{gevorgyan2023equivariant} that $p^H \colon E^H \to B^H$ and $p \colon E \to B$ are a fibration and a $K$-fibration, respectively.
Hence, the following inequalities follow by taking $K$ and $H$ to be the trivial subgroup, respectively.
\end{proof}

\begin{comment}
\begin{corollary}
\label{cor: fixed-point-and-subgroup-ineq for equi-secat}
    Suppose $G$ is a compact Hausdorff topological group and $p \colon E \to B$ is a $G$-fibration. 
    Then
    \begin{enumerate}
        \item the fixed point map $p^H \colon E^H \to B^H$ is a fibration for all closed subgroups $H$ of $G$, and
        $$
            \secat(p^H) \leq \secat_G(p).
        $$
        \item $p \colon E \to B$ is a $K$-fibration for all closed subgroups $K$ of $G$, and 
        $$
            \secat(p) \leq \secat_K(p) \leq \secat_G(p).
        $$
    \end{enumerate}
\end{corollary}

\begin{proof}
    Suppose $T$ is the trivial subgroup of $G$. 
    By \cite[Theorem 4]{gevorgyan2023equivariant}, we know that for a $G$-fibration $p$, the fixed point map $p^H \colon E^H \to B^H$ is fibration for all closed subgroups $H$ of $G$. 
    Hence, by taking $K=T$ in \Cref{prop: subgroup-ineq for equi-secat}, it follows that $\secat(p^H) =\secat_T(p^H) \leq \secat_G(p)$.

    By \cite[Theorem 3]{gevorgyan2023equivariant}, we know that $p$ is a $K$-fibration for all closed subgroups $K$ of $G$. 
    Hence, by taking $H = T$ in \Cref{prop: subgroup-ineq for equi-secat}, we get the subgroup inequality $\secat_K(p) = \secat_K(p^T) \leq \secat_G(p)$. 
    Note that $T$ is a closed subgroup of a compact Hausdorff topological group $K$. 
    Hence, applying the subgroup inequality for the $K$-fibration $p$, we get $\secat(p) = \secat_T(p) \leq \secat_K(p)$.
\end{proof}
\end{comment}

The following proposition states some basic properties of the equivariant sectional category. Proofs are left to the reader. For analogous results concerning the non-equivariant sectional category, we refer to \cite[Lemma 2.1]{Grant-ParaTC-group}.

\begin{proposition}
\label{prop: properties of equi-sec-cat}
    Suppose $p \colon E \to B$ is a $G$-map.
    \begin{enumerate}
        \item If $p' \colon E \to B$ is $G$-homotopic to $p$, then $\secat_G(p') = \secat_G(p)$.
        \item If $h \colon E' \to E$ is $G$-homotopy equivalence, then $\secat_G(p \circ h) = \secat_G(p)$.
        \item If $f \colon B \to B'$ is a $G$-homotopy equivalence, then $\secat_G(f \circ p) = \secat_G(p)$.
    \end{enumerate}
\end{proposition}

\begin{corollary}
\label{cor: equi-sec-cat under homotopy equivalence pullback}
    Suppose $p \colon E \to B$ is a $G$-fibration. 
    If $g \colon B' \to B$ is a $G$-homotopy equivalence and $p' \colon E' \to B'$ is the pullback of $p$ along $g$, then
        $$
            \secat_G(p') = \secat_G(p).
        $$
\end{corollary}
    
\begin{proof}
    Suppose the following diagram is a pullback
    \[
    % https://tikzcd.yichuanshen.de/#N4Igdg9gJgpgziAXAbVABwnAlgFyxMJZABgBpiBdUkANwEMAbAVxiRAFEByEAX1PUy58hFAEZyVWoxZt2vfiAzY8BImVGT6zVohAAhbnwHLhRcRupaZuvb0kwoAc3hFQAMwBOEALZIATNQ4EEgAzJbSOiCO8u5evojiIEGh4dpsaDEgnj5IZEnBCanWIAAWmdnxecmIAVJpumjc1Ax0AEYwDAAKgioiIB5YjiU4djxAA
    \begin{tikzcd}
    E' \arrow[r, "h"] \arrow[d, "p'"'] & E \arrow[d, "p"] \\
    B' \arrow[r, "g"]                  & B.               
    \end{tikzcd}
    \]
    Since $g$ is a $G$-homotopy equivalence and $p$ is a $G$-fibration, it follows that $h$ is also a $G$-homotopy equivalence. Hence, we get
    $$
        \secat_G(p') = \secat_G(g \circ p') = \secat_G(p \circ h) = \secat_G(p).
    $$
    by \Cref{prop: properties of equi-sec-cat}.
\end{proof}

Generalizing Schwarz's dimension-connectivity upper bound on the sectional category, Grant established the corresponding equivariant analogue for the equivariant sectional category in \cite[Theorem 3.5]{grant2019symmetrized}. 
We extend this approach to derive an equivariant homotopy dimension-connectivity upper bound for equivariant sectional category. 
To achieve this, we first introduce the notion of $G$-homotopy dimension for $G$-$\mathrm{CW}$-complexes.

\begin{definition}
    Suppose $X$ is a $G$-$\mathrm{CW}$-complex. 
    The $G$-homotopy dimension of $X$, denoted $\mathrm{hdim}_{G}(X)$, is defined to be
    $$
        \mathrm{hdim}_{G}(X) 
            := \min\{\dim(X') \mid X' \text{ is a $G$-}\mathrm{CW}\text{-complex}, X'\simeq_G X\}.
    $$
\end{definition}

% ALREADY THERE IN ABOVE PARAGRAPH ABOVE DEFINITION

% We are now ready to state a homotopy-dimension connectivity upper bound for equivariant sectional category.

\begin{theorem}
\label{thm: G-homo-dim ub on equi-secat}
Suppose $G$ is a compact Lie group.
Suppose $p \colon E \to B$ is a $G$-fibration with fibre $F$, whose base $B$ is a $G$-$\mathrm{CW}$-complex of dimension at least $2$. 
    If there exists $s \geq 0$ such that the fibre of $p^H \colon E^H \to B^H$ is $(s-1)$-connected for all subgroups $H$ of $G$, then
    % If $\pi_n(F^H)=0$ for all subgroups $H$ of $G$ and for all $n < s$, where $s \geq 0$, then 
    $$
        \secat_{G}(p) < \frac{\mathrm{hdim_{G}(B)+1}}{s+1} + 1.
    $$
\end{theorem}

\begin{proof}
    % It is enough to show that for any $G$-$\mathrm{CW}$-complex $B'$ which is $G$-homotopy equivalent to $B$, we have
    % $$
    %     \secat_{G}(p) < \frac{\mathrm{dim(B')+1}}{s+1}+1.
    % $$
    Suppose $f \colon B' \to B$ is a $G$-homotopy between $G$-$\mathrm{CW}$-complexes $B'$ and $B$, and $p' \colon E' \to B'$ is the pullback of $p$ along $f$. 
    Then, by \Cref{cor: equi-sec-cat under homotopy equivalence pullback}, we have $\secat_{G}(p') = \secat_{G}(p)$.
    Since the fibre of $p'$ is also $F$, we get 
    $$
        \secat_{G}(p') < \frac{\mathrm{dim(B')+1}}{s+1}+1,
    $$
    by \cite[Theorem 3.5]{grant2019symmetrized}.
\end{proof}

Our next aim is to establish product inequalities for the equivariant sectional category. 
% Before proving it, we define some useful notions.

\begin{definition}
% \begin{enumerate}
    % \item A topological space $X$ is called completely normal if, for any two subsets $A$ and $B$ of $X$ with $\overline{A} \cap B = A \cap \overline{B} = \emptyset$, there exist disjoint open subsets of $X$ containing $A$ and $B$, respectively.
    %\item 
    A $G$-space $X$ is called $G$-completely normal if for any two $G$-invariant subsets $A$ and $B$ of $X$ with $\overline{A} \cap B = A \cap \overline{B} = \emptyset$, there exist disjoint $G$-invariant open subsets of $X$ containing $A$ and $B$, respectively.
%\end{enumerate}
\end{definition}

\begin{comment}
\begin{lemma}[{\cite[Lemma 3.12]{colmangranteqtc}}]
\label{lemma: completely normal G-space is G-completely normal}
    Suppose that $G$ is a compact Hausdorff topological group acting continuously on a Hausdorff topological space $X$. If $X$ is completely normal, then $X$ is $G$-completely normal.
\end{lemma}
\end{comment}

\begin{proposition}
\label{prop: special-prod-ineq for equi-secat}
Suppose $p_i \colon E_i \to B_i$ is a $G$-fibration for $i =1,2$. 
If $G$ is compact Hausdorff, then $p_1 \times p_2 \colon E_1 \times E_2 \to B_1 \times B_2$ is a $G$-fibration, where $G$ acts on $E_1 \times E_2$ and $B_1 \times B_2$ diagonally. 
Furthermore, if 
    % $B_1$ and $B_2$ are Hausdorff, and 
$B_1 \times B_2$ is completely normal, then
$$
    \secat_{G}(p_1\times p_2) \leq \secat_{G}(p_1) + \secat_{G}(p_2) - 1.
$$
\end{proposition}

\begin{proof}
% Suppose $G$ is compact Hausdorff. Then 
Identifying $G$ with the diagonal subgroup of $G\times G$, we see that it is a closed subgroup of $G\times G$.
Hence, by \cite[Theorem 3]{gevorgyan2023equivariant}, it follows that $p_1 \times p_2$ is a $G$-fibration.
% where $G$ acts diagonally on the spaces $E_1 \times E_2$ and $B_1 \times B_2$. If $B_1$ and $B_2$ are Hausdorff, and $B_1 \times B_2$ is completely normal, then 
% \Cref{lemma: completely normal G-space is G-completely normal} 
Furthermore, by \cite[Lemma 3.12]{colmangranteqtc}, it follows that $B_1 \times B_2$ is $(G \times G)$-completely normal. 
Hence, the desired inequality 
$$
\secat_{G}(p_1 \times p_2) 
         \leq \secat_{G\times G}(p_1 \times p_2) 
         \leq \secat_G(p_1) + \secat_G(p_2) - 1
$$
follows from \Cref{prop: fixed-point-and-subgroup-ineq for equi-secat} (2) and \cite[Proposition 3.7]{A-D-S}.
\end{proof}

\begin{corollary}
\label{cor: equi-secat of pullback fibration under diagonal map}
    Suppose $p_i \colon E_i \to B$ is a $G$-fibration for $i =1,2$. 
    Let $E_1 \times_B E_2 = \{(e_1,e_2) \in E_1 \times E_2 \mid p_1(e_1) = p_2(e_2)\}$ and let $p\colon  E_1 \times_B E_2 \to B$ be the $G$-map given by $p(e_1,e_2)=p_1(e_1)=p_2(e_2)$, where $G$ acts on $E_1 \times_B E_2$ diagonally. 
    If $G$ is compact Hausdorff, then $p$ is a $G$-fibration. 
Furthermore, if 
    % $B$ is Hausdorff and 
$B \times B$ is completely normal, then
    \[
        \secat_G(p) \leq \secat_G(p_1) + \secat_G(p_2) -1.
    \]
\end{corollary}

\begin{proof}
Note that the following diagram
\begin{equation}
\label{diag: pullback under diagonal of product fibration}
\begin{tikzcd}
    E_1 \times_B E_2 \arrow[r, hook] \arrow[d, "p"'] & E_1 \times E_2 \arrow[d, "p_1 \times p_2"] \\
    B \arrow[r, "\Delta"]                               & B \times B                 
\end{tikzcd}   
\end{equation}
is a pullback in the category of $G$-spaces, where $\Delta \colon B \to B \times B$ is the diagonal map. 
    In \Cref{prop: special-prod-ineq for equi-secat}, we showed that $p_1 \times p_2$ is a $G$-fibration if $G$ is compact Hausdorff. Hence, $p$ is a $G$-fibration.
    Thus, the desired inequality 
    \begin{align*}
        \secat_{G}(p)  
             \leq \secat_{G}(p_1 \times p_2) 
             \leq \secat_G(p_1) + \secat_G(p_2) - 1
    \end{align*}
    follows from \cite[Proposition 4.3]{colmangranteqtc} and \Cref{prop: special-prod-ineq for equi-secat}.
\end{proof}

%------------------------------------------------------------------------------------------------
\subsection{Equivariant LS-category}
\hfill\\ \vspace{-0.7em}
\label{subsec:eq-LS-category}

The notion of Lusternik-Schnirelmann (LS) category was introduced by Lusternik and Schnirelmann in \cite{LScat}.
In this section, we recall the corresponding equivariant analogue. 
% We provide several inequalities relating equivariant category of a $G$-space to the equivariant and nonequivariant categories of the various fixed point sets.
% We also establish an equivariant homotopy dimension-connectivity upper bound for equivariant LS category.

% We now define some notions for stating and proving our desired result of this subsection.
\begin{definition}
A $G$-invariant subset $U$ of a $G$-space $X$ is said to be \emph{$G$-categorical} if the inclusion map $i_U \colon U \hookrightarrow X$ is $G$-homotopy equivalent to a map which takes values in a single orbit.
\end{definition}

\begin{definition}[\cite{Fadelleqcat}]
The \emph{equivariant LS-category} of a $G$-space $X$, denoted by $\ct_G(X)$, is the least positive integer $k$ such that there exists a $G$-categorical open cover $\{U_1,\dots,U_k\}$ of $X$.
\end{definition}

\begin{definition}
A $G$-space $X$ is said to be \emph{$G$-connected} if $X^H$ is path-connected for every closed subgroup $H$ of $G$.
\end{definition}

Let $X$ be a $G$-space, and $x_0 \in X$. Define the path space of $(X,x_0)$ as
$$
    P_{x_0}X = \{\alpha \colon I \to X \mid \alpha(0) = x_0\}.
$$
Then the map $e_X \colon P_{x_0}X \to X$, given by $e_X(\alpha)=\alpha(1)$, is a fibration. 
Moreover, if the point $x_0$ is fixed under the $G$-action, then $e_X$ is a $G$-fibration, where $P_{x_0}X$ admits a $G$-action via $(g\cdot \alpha)(t) := g\cdot \alpha(t)$.
We note that the fibre of $e_X$ is the based loop space $\Omega X = (e_X)^{-1}(x_0)$ of $X$, and the $G$-action on $P_{x_0}X$ restricts to a $G$-action on $\Omega X$. 
Furthermore, we have a commutative diagram of $G$-maps
\begin{equation}
\label{diag: G-cat comm-diag}
    % https://tikzcd.yichuanshen.de/#N4Igdg9gJgpgziAXAbVABwnAlgFyxMJZABgBpiBdUkANwEMAbAVxiRAB13gAPAfWM4BfEINLpMufIRQBGUjKq1GLNgA0RYkBmx4CRAEzlF9Zq0QgACrx79B6wYphQA5vCKgAZgCcIAWyRkIDgQSIZKpmwAFiDUDHQARjAMFhK60iAMMB44Gp4+-oiBwUhyQXRYDFEQEADWMeEq5lj1cYnJqVJsXljOkTmieX6h1MWIpSaNIDC89hSCQA
\begin{tikzcd}
\{x_0\} \arrow[rr, "h"] \arrow[rd, "i"', hook] &   & P_{x_0}X \arrow[ld, "e_X"] \\
                                               & X, &                           
\end{tikzcd}
\end{equation}
where $h$ is a $G$-homotopy equivalence and $i \colon \{x_0\} \hookrightarrow X$ is the inclusion map.

\begin{lemma}[{\cite[Corollary 4.7]{colmangranteqtc}}] 
\label{lemma: catG(X) = secatG(e_X)}
Suppose $G$ is a compact Hausdorff topological group. 
If $X$ is a $G$-space such that $X$ is $G$-connected and $x_0 \in X^G$, then
    $
        \ct_G(X) = \secat_G(e_X).
    $
\end{lemma}

We now present inequalities relating $\ct_G(X)$ to the non-equivariant category of fixed point sets and to the equivariant category of $X$ viewed as a $K$-space, for each closed subgroup $K$ of $G$.

\begin{comment}
\begin{proposition}
% \label{prop: subgroup-ineq for equi-cat}
    Suppose $X$ is a $G$-connected space with $X^G \neq \emptyset$. 
    If $H$ and $K$ are closed subgroups of $G$ such that $X^H$ is $K$-invariant, then 
    $$
        \ct_K(X^H) \leq \ct_G(X).
    $$
    In particular, if $G$ is Hausdorff, then 
    \begin{enumerate}
        \item $\ct(X^H) \leq \ct_G(X)$ for all closed subgroups $H$ of $G$.
        \item $\ct_K(X) \leq \ct_G(X)$ for all closed subgroups $K$ of $G$.
    \end{enumerate}
\end{proposition}

\begin{proof}
    We note that $X^H$ is $K$-connected since $X$ is $G$-connected, and $H$ and $K$ are closed subgroups of $G$. 
    If $x_0 \in X^G$, then $x_0 \in (X^H)^K = X^{H \cap K}$. 
    Hence, by \Cref{lemma: catG(X) = secatG(e_X)}, it is enough to show that $\secat_K(e_{X^H}) \leq \secat_G(e_X)$.

    Suppose $U$ is a $G$-invariant open subset of $X$ and $s \colon U \to P_{x_0}X$ is $G$-equivariant section of $e_X$.
    Set $V := U \cap X^H$. 
    Then $V$ is a $K$-invariant open subset of $X^H$.
    As $s$ is $G$-equivariant, it restricts to a $K$-equivariant map $\left.s\right|_{V} \colon V \to (P_{x_0}X)^H = P_{x_0}(X^H)$.
    Clearly, $\left.s\right|_{V}$ is a $K$-equivariant section of $e_{X^H} \colon P_{x_0}(X^H) \to X^H$.
\end{proof}
\end{comment}

\begin{proposition}
\label{prop: subgroup-ineq for equi-cat}
    Let $G$ be a compact Hausdorff topological group, and let $X$ be a $G$-connected space with $X^G \neq \emptyset$. 
    If $H$ and $K$ are closed subgroups of $G$ such that $X^H$ is $K$-invariant and $HK' = K'H$ for all closed subgroups $K'$ of $K$, then 
    $$
        \ct_K(X^H) \leq \ct_G(X).
    $$
\end{proposition}

\begin{proof}
    Suppose $K'$ is a closed subgroup of $K$. 
    Then $HK' = K'H$ implies $HK'$ is a subgroup of $G$ and $\langle H, K' \rangle = HK'$, where $\langle H, K' \rangle$ is the subgroup generated by $H$ and $K'$.
    We note that $HK'$ is closed, as it is the image of the compact space $H \times K'$ under the group operation $G \times G \to G$ into the Hausdorff space $G$.
    Hence, $X^H$ is $K$-connected, since
    $$
        (X^H)^{K'} 
            = X^H \cap X^{K'} 
            = X^{\langle H, K' \rangle}
            = X^{HK'},
    $$
    and $X$ is $G$-connected. 
    Moreover, if $x_0 \in X^G$, then $x_0 \in (X^H)^K$. 
    Hence, by \Cref{lemma: catG(X) = secatG(e_X)}, it is enough to show that $\secat_K(e_{X^H}) \leq \secat_G(e_X)$.

    Suppose $U$ is a $G$-invariant open subset of $X$ and $s \colon U \to P_{x_0}X$ is $G$-equivariant section of $e_X$.
    Set $V := U \cap X^H$. 
    Then $V$ is a $K$-invariant open subset of $X^H$.
    As $s$ is $G$-equivariant, it restricts to a $K$-equivariant map $\left.s\right|_{V} \colon V \to (P_{x_0}X)^H = P_{x_0}(X^H)$.
    Clearly, $\left.s\right|_{V}$ is a $K$-equivariant section of $e_{X^H} \colon P_{x_0}(X^H) \to X^H$.
\end{proof} 

\begin{corollary}
\label{cor: fixed-point-and-subgroup-ineq for equi-cat}
    Let $G$ be a compact Hausdorff topological group, and let $X$ be a $G$-connected space with $X^G \neq \emptyset$. Then
    \begin{enumerate}
        \item $\ct(X^H) \leq \ct_G(X)$ for all closed subgroups $H$ of $G$.
        \item $\ct_K(X) \leq \ct_G(X)$ for all closed subgroups $K$ of $G$.
    \end{enumerate}
\end{corollary}

Now, as a consequence of \Cref{thm: G-homo-dim ub on equi-secat}, we obtain an equivariant homotopy dimension-connectivity upper bound for equivariant LS category.

\begin{theorem}
\label{thm: G-hom-dim on equi-cat}
Suppose $G$ is a compact Lie group.
Suppose $X$ is a $G$-$\mathrm{CW}$-complex of dimension at least $2$ such that $X^G \neq \emptyset$.
If there exists $s \geq 0$ such that $X^H$ is $s$-connected for all subgroups $H$ of $G$, then
$$
    \ct_G(X) < \frac{\mathrm{hdim}_G(X)+1}{s+1}+1.
$$
\end{theorem}

\begin{proof}
    If $x_0 \in X^G$, then $e_X \colon P_{x_0}X \to X$ is a $G$-fibration with fibre $\Omega X$, which also admits a $G$-action. 
    Note that $(\Omega X)^H = \Omega(X^H)$. 
    Since $X^H$ is $s$-connected, the loop space $\Omega (X^H)$ is $(s-1)$-connected.
    Hence, by \Cref{thm: G-homo-dim ub on equi-secat}, we get
    $$
        \secat_{G}(e_X) < \frac{\mathrm{hdim}_G(X)+1}{s+1}+1.
    $$

    As $X^H$ is $s$-connected, it follows that $X^H$ is path-connected. 
    Hence, $X$ is $G$-connected, and the theorem follows by \Cref{lemma: catG(X) = secatG(e_X)}.
\end{proof}

%----------------------------------------------------------------------------------
\subsection{Equivariant and invariant topological complexity} 
\hfill\\ \vspace{-0.7em}
\label{subsec:invtc}

We recall the concept of equivariant topological complexity introduced by Colman and Grant in \cite{colmangranteqtc}.
Let $X$ be a $G$-space. 
Observe that the free path space $PX$ admits a $G$-action via $(g\cdot \alpha)(t):=g\cdot \alpha(t)$. 
Similarly, the product space $X^k$ is a $G$-space with the diagonal action. 
The fibration
$$
e_{k,X} \colon PX \to X^k, \quad \alpha \mapsto \left(\alpha(0), \alpha\left(\frac{1}{k-1}\right), \ldots, \alpha\left(\frac{i}{k-1}\right), \ldots, \alpha\left(\frac{k-2}{k-1}\right), \alpha(1)\right)
$$
is a $G$-fibration.

\begin{definition}[{\cite{BaySarkarheqtc}}]
The sequential equivariant topological complexity of a $G$-space $X$ is defined as 
    $$
        \TC_{k,G}(X) := \secat_G(e_{k,X}).
    $$
In particular, when $k=2$, we will denote $e_{2,X}$ by $\pi$ and $\TC_{2,G}(X)$ by $\TC_G(X)$.
\end{definition}

In \cite[Proposition 3.40]{A-D-S}, Sarkar and the authors of this paper provided a dimension-connectivity upper bound on the sequential equivariant topological complexity. 
We improve their result by establishing an equivariant homotopy dimension-connectivity upper bound. 
We omit the proof, as it is similar to the original and follows from the homotopy dimension-connectivity upper bound on the equivariant sectional category in \Cref{thm: G-homo-dim ub on equi-secat}.

\begin{theorem}
\label{thm: G-hom-dim ub for equi-tc}
Suppose $G$ is a compact Lie group.
If $X$ is a $G$-$\mathrm{CW}$-complex of dimension at least 1 such that $X^H$ is $s$-connected for all subgroups $H \leq G$, then
$$
    \TC_{k,G}(X) < \frac{k\ \mathrm{hdim}_G(X)+1}{s+1}+1.
$$
\end{theorem}

\begin{example}
Consider $G = \Z_2$ acting on $X = S^n$ with $n \geq 1$, by reflection, given by multiplication by $-1$ in the last coordinate.
Then $X^G = S^{n-1}$ and $X^{\{e\}} = S^n$ are $(n-2)$-connected and $(n-1)$-connected, respectively. 
Then, by \Cref{thm: G-hom-dim on equi-cat}, it follows that
$$
    \ct_G(S^n) < \frac{n+1}{n-1} + 1, 
        \quad \text{for }n \geq 2,
$$
which implies $\ct_G(S^n) \leq 2$ for $n \geq 3$.  
Hence, $\ct_G(S^n) = 2$ for $n \geq 3$, since $S^n$ is clearly not $G$-contractible.
Moreover, by \Cref{thm: G-hom-dim ub for equi-tc}, it follows that
$$
    \TC_{k,G}(S^n) < \frac{k\cdot n+1}{n-1} + 1,
        \quad \text{for }n \geq 2
$$
which implies $\TC_{k,G}(S^n) \leq k+1$ for $k \leq n-2$. 
Hence, it follows that $\TC_{k,G}(S^n) = k+1$ for $k \leq n-2$, since 
$$
    \TC_k(S^{n-1}) \leq \TC_{k,G}(S^n) 
        \quad \text{and} \quad
    \TC_k(S^n) \leq \TC_{k,G}(S^n)
$$
by \cite[Proposition 3.14 (2)]{BaySarkarheqtc}, and 
$$
    \TC_k(S^n) = 
        \begin{cases}
            k, & \text{if } k \text{ is odd},\\
            k+1, & \text{if } k \text{ is even},
        \end{cases}
$$
by \cite[Section 4]{RUD2010}.

We note that this computation can be carried out more directly and efficiently as follows.
Observe that $\ct_G(S^n) = 2$ for all $n \geq 2$, see \cite[Example 5.9]{colmangranteqtc}.
Hence, it follows from \cite[Proposition 3.17]{BaySarkarheqtc} that $\TC_{k,G}(S^n) \leq k+1$ for all $n \geq 2$ and all $k \geq 2$. 
Moreover, $\TC_{k,G}(S^1) = \infty$, since $\TC_k(S^0) \leq \TC_{k,G}(S^1)$.
Hence,
$$
    \TC_{k,G}(S^n) = 
        \begin{cases}
            \infty, & \text{if } n=1,\\
            k+1, & \text{if } n \geq 2.
        \end{cases}
$$
\end{example}

%--------------------------------------------------------------------------------
It is important to note that the equivariant topological complexity of $G$-spaces does not necessarily relate to the topological complexity of their orbit spaces. 
However, Lubawski and Marzantowicz \cite{invarianttc} introduced a new notion of topological complexity for $G$-spaces, designed to facilitate such a comparison.
We now present their definition and recall the corresponding result.

Suppose $X$ is a $G$-space. 
Let $\pi_X \colon X \to X/G$ denote the orbit map.
Define
$$
PX \times_{X/G} PX 
    := \{ (\gamma,\delta) \in PX \times PX \mid G\cdot \gamma(1) = G \cdot \delta(0)\}
$$
Then the following diagram
\[
% https://tikzcd.yichuanshen.de/#N4Igdg9gJgpgziAXAbVABwnAlgFyxMJZABgBpiBdUkANwEMAbAVxiRAFEA9YASQF8A+gCEABAB0xeALbwBwIQHoA4nxFdeg0SD6l0mXPkIoAjOSq1GLNuv7DtukBmx4CRMsfP1mrRB262hez1nQyJTD2ovK19FJW1zGCgAc3giUAAzACcIKSRTEBwIJABmSMsfEAk0LGFxMQBjLEz6kTQ6xuaRGAFiIJAsnKQAJmpCkrLvNiqa0QkOlra5ppbu4xBqBjoAIxgGAAV9FyMQTKwkgAscPoHcxDICosQRi0nfaYE1je3dg5DXX1OFyuOgy2Vu9zGiHyUQq7yG8T4QA
\begin{tikzcd}
PX \times_{X/G} PX  \arrow[d, "\pi_1"'] \arrow[r, "\pi_2"] & PX \arrow[d, "\pi_X \circ e_0"] \\
PX \arrow[r, "\pi_X \circ e_1"']                        & X/G                                         
\end{tikzcd}
\]
is a pullback.
Define the map
\begin{equation}
\label{eq: inv-tc}
    \p \colon  PX \times_{X/G} PX \to X \times X, \quad (\gamma,\delta) \mapsto (\gamma(0),\delta(1)).
\end{equation}
It was shown in \cite[Proposition 3.7]{invarianttc} that the map $\p$ is a $(G \times G)$-fibration.

\begin{definition}
 Let $X$ be a $G$-space. The invariant topological complexity of $X$ denoted by $\TC^G(X)$, is defined as 
    \[
        \TC^G(X):=\secat_{G\times G}(\p).
    \]   
\end{definition}

The following theorem relates the invariant topological complexity of a free $G$-space $X$ with that of the topological complexity of its corresponding orbit space.

\begin{theorem}[{\cite[Theorem 3.9 and 3.10]{invarianttc}}]
\label{thm: invariance theorem for TC}
Let $G$ be a compact Lie group and $X$ be a compact $G$-ANR. 
Then 
$$
    \TC(X/G)\leq \TC^G(X).
$$ 
Moreover, if $X$ has only one orbit type, then
$$
    \TC^G(X)=\TC(X/G).
$$  
\end{theorem}

%-------------------------------------------------------------------------------------
\subsection{Clapp-Puppe invariant of Lusternik-Schnirelmann type} 
\hfill\\ \vspace{-0.7em}
\label{subsec:Clapp-Puppe}

In this section, we recall the equivariant version of Clapp-Puppe invariant \cite{clapp-puppe-ls-category}, introduced by Lubawski and Marzantowicz in \cite{invarianttc}.

\begin{definition}\label{def:G-compress}
    Let $A$ be a $G$-invariant closed subset of a $G$-space $X$. 
    A $G$-invariant open subset of $X$ is said to be $G$-compressible into $A$ if the inclusion map $i_{U} \colon U \rightarrow X$ is $G$-homotopic to a $G$-map $c \colon U \to X$ which takes values in $A$.
\end{definition}

\begin{definition}
    Let $A$ be a $G$-invariant closed subset of a $G$-space $X$. 
    The $A$-Lusternik-Schnirelmann $G$-category of $X$, denoted $~_{A}\ct_G(X)$, is the least positive integer $k$ such that there exists a $G$-invariant open cover $\{U_1,\dots,U_k\}$ of $X$ such that each $U_i$ is $G$-compressible into $A$.
\end{definition}

    Colman and Grant in \cite[Lemma 5.14]{colmangranteqtc} showed that for a $G$-invariant open subset $U$ of $X \times X$ the following are equivalent:
    \begin{enumerate}
        \item there exists a $G$-equivariant section of $e_X \colon PX \to X \times X$ over $U$,
        \item $U$ is $G$-compressible into the diagonal $\Delta(X) \subset X \times X$. 
    \end{enumerate}
    In particular,
    $$
        \TC_G(X) = ~_{\Delta(X)}\ct_G(X \times X).
    $$
    Later, Lubawski and Marzantowicz in \cite[Lemma 3.8]{invarianttc} showed a similar result for invariant topological complexity. More precisely, for a $(G \times G)$-invariant open subset of $U$ of $X \times X$ the following are equivalent:
    \begin{enumerate}
        \item there exists a $(G \times G)$-equivariant section of $\mathfrak{p} \colon PX \times_{X/G} X \to X \times X$ over $U$,
        \item $U$ is $(G \times G)$-compressible into the saturation of the diagonal $\mathbb{k}(X) := (G \times G)\cdot \Delta(X) \subset X \times X$.
    \end{enumerate}
    In particular,
    $$
        \TC^G(X) = ~_{\mathbb{k}(X)}\ct_{G \times G}(X \times X).
    $$
    In \Cref{sec: equi-para-tc} and \Cref{sec: inv-para-tc}, we give analogous results for equivariant parametrized topological complexity and invariant parametrized topological complexity, respectively. We use these results to prove \Cref{thm: inv-para-tc under free action}.

%--------------------------------------------------------------------------------------------------------------
\section{Equivariant parametrized topological complexity}
\label{sec: equi-para-tc}

% In this section, we define the equivariant parametrized topological complexity of $G$-fibrations $p\colon E\to B$. 
% We show that it can be expressed as the equivariant $\Delta(E)$-LS category of the fibre product $E\times_B E$. 
% Additionally, we establish the product inequality for the equivariant parametrized topological complexity.

For a $G$-fibration $p \colon E\to B$, consider the subspace $E^I_B$ of the free path space $E^I$ of $E$ defined by
\[
    E^I_B := \{\gamma\in E^I \mid \gamma(t) \in p^{-1}(b) ~\text{for some}~ b\in B ~\text{and for all}~ t\in[0,1] \}.
\]
Consider the pullback corresponding to the fibration $p \colon E\to B$ defined by 
\[
    E \times_B E = \{(e_1,e_2) \in E \times E \mid p(e_1)=p(e_2) \}.
\]
It is clear that the $G$-action on $E^I$ given by
$$
    (g \cdot \gamma)(t) := g \cdot \gamma(t) \quad \text{for all } g \in G, \gamma \in E^I, t \in I;
$$
and the diagonal action of $G$ on $E \times E$ restricts to $E^I_B$ and $E \times_B E$, respectively. Then the map
\begin{equation}
\label{eq: para-tc}
    \Pi \colon E^I_B \to E\times_B E, \quad \Pi(\gamma) = (\gamma(0),\gamma(1))
\end{equation}
is a $G$-fibration, see \cite[Corollary 4.3]{D-EqPTC}.

\begin{definition}[{\cite[Definition 4.1]{D-EqPTC}}]
The equivariant parametrized topological complexity of a $G$-fibration $p\colon E\to B$, denoted by $\TC_G[p\colon E\to B]$, is defined as 
\[
    \TC_G[p\colon E \to B] := \secat_G\left( \Pi \right).
\]    
\end{definition}

Suppose $\Delta \colon E \to E \times E$ is the diagonal map.  Then it is clear that the image $\Delta(E)$ is a $G$-invariant subset of $E \times_B E$. In the next theorem, we prove the  parametrized  analogue of \cite[Lemma 5.14]{colmangranteqtc} in the equivariant setting.

\begin{theorem}
\label{thm: equivalent defn of a section of equi-para-tc}
Let $p\colon E\to B$ be a $G$-fibration. 
For a $G$-invariant (not necessarily open) subset $U$ of $E \times_B E$ the following are equivalent:
\begin{enumerate}
\item there exists a $G$-equivariant section of $\Pi \colon E^I_B\to E\times_B E$ over $U$.
% \item there exists a $G$-invariant homotopy section of $\Pi \colon E^I_B\to E\times_B E$.
\item there exists a $G$-homotopy between the inclusion map $i_U\colon U \hookrightarrow E \times_B E$ and a $G$-map $f\colon U \to E \times_B E$ which takes values in $\Delta(E)$.
\end{enumerate}
\end{theorem}

\begin{proof}
$(1) \implies (2)$. 
Suppose $\sigma \colon U \to E^{I}_B$ is a $G$-equivariant section of $\Pi$. 
Let $H \colon E^{I}_B \times I \to E^{I}_B$ be given by
$$
    H(\gamma,t)(s) 
        = \gamma(s(1-t)), 
            \quad \text{for }
            \gamma \in E^{I}_B \text{ and } s,t \in I.
$$
% It is clear that $H(\gamma,t) \in E^{I}_B$ for all $\gamma \in E^{I}_B$ and $t \in I$.
% Hence, $H$ is well-defined. 
Clearly, $H$ is well-defined and $G$-equivariant, such that $H_0 = \mathrm{id}_{E^{I}_B}$ and $H(\gamma,1) = c_{\gamma(0)}$, where $c_{e}$ is the constant path in $E$ taking the value $e \in E$. Then
$$
    F := \Pi \circ H \circ (\sigma \times \mathrm{id}_I) \colon U \times I \to E \times_{B} E    
$$
is a $G$-homotopy such that $F_0 
% = \Pi \circ \mathrm{id}_{E^{I}_B} \circ \sigma = 
= i_U$ and $F_1(u) 
% = \Pi(H_1(\sigma(u))) = \Pi(c_{\sigma(u)(0)}) 
= (\sigma(u)(0),\sigma(u)(0)) \in \Delta(E)$. 
Hence, $F_1$ is the desired map.
    
$(2) \implies (1)$. 
Suppose $H \colon U \times I \to E \times_{B} E$ is a $G$-homotopy between $f$ and $i_U$. 
Let $\sigma \colon U \to E^{I}_B$ be the $G$-map given by $\sigma(u) = c_{\pi_1(f(u))} = c_{\pi_2(f(u))}$, where $\pi_i \colon E \times_B E \to E$ is the projection map onto the $i$-th factor. 
% By the $G$-homotopy lifting property of $\Pi$, there exists a $G$-homotopy $\widetilde{H} \colon U \times I \to E^{I}_B$ such that the following diagram 
% \[
% % https://tikzcd.yichuanshen.de/#N4Igdg9gJgpgziAXAbVABwnAlgFyxMJZABgBpiBdUkANwEMAbAVxiRAFUACAHW7wFt4PbsGK8AviHGl0mXPkIoATOSq1GLNgFEAesACS4gPoAhKTJAZseAkTIBGNfWatEHYQKH7zs6wqIqjtTOmm5aHliCcEbAJuKcWlJqMFAA5vBEoABmAE4Q-EhkIDgQSCrFdFgMbAAWEBAA1j4gufll1CVIAMzBGq4gABLNrQWI9h2liD3qLmy8AApYw3mjRZ1jvbNuvNip-HTLbYjl6+MgDFhg-VB0cDUpIJuhILwA7liweAywwAOS4hRxEA
% \begin{tikzcd}
% U \times \{0\} \arrow[d, hook] \arrow[rr, "\sigma"]             &  & E^{I}_B \arrow[d, "\Pi"] \\
% U \times I \arrow[rr, "H"] \arrow[rru, "\widetilde{H}", dashed] &  & E \times_{B} E          
% \end{tikzcd}
% \]
% commutes. 
Since $\Pi$ is a $G$-fibration, there exists a $G$-homotopy $\widetilde{H} \colon U \times I \to E^{I}_B$ such that $\widetilde{H}_0 = \sigma$ and $\Pi \circ \widetilde{H} = H$.
Then $\Pi \circ \widetilde{H}_1 = H_1 = i_U$ implies $\widetilde{H}_1$ is a $G$-equivariant section of $\Pi$ over $U$.  
\end{proof}

As a consequence of the previous theorem, we can now express the equivariant parametrized topological complexity as the equivariant $\Delta(E)$-LS category of the fibre product. 

\begin{corollary}
\label{cor: equivalent defn of eq-para-tc}
    For a $G$-fibration $p \colon E \to B$, we have
    \[
    \TC_{G}[p \colon E \to B] = ~_{\Delta(E)}\ct_G(E\times_B E).
    \]
\end{corollary}

%--------------------------------------------------------------------
% \subsection{Properties and Bounds}
% \hfill\\ \vspace{-0.7em}

\begin{proposition}
\label{prop: subgroup-ineq for equi-para-tc}
    Suppose $p \colon E \to B$ is a $G$-fibration. 
    If $H$ and $K$ are subgroups of $G$ such that $E^H$ and $B^H$ are $K$-invariant, and the fixed point map $p^H \colon E^H \to B^H$ is a $K$-fibration, then
    $$
        \TC_K[p^H \colon E^H \to B^H] \leq \TC_G[p \colon E \to B].
    $$
\end{proposition}

\begin{proof}
Suppose $\Pi \colon E^I_B \to E \times_B E$ is the $G$-equivariant parametrized fibration corresponding to $p$. 
Then it is easily checked that
    $$
        (E^I_B)^H = (E^H)^I_{B^H} \quad \text{and} \quad (E \times_B E)^H = E^H \times_{B^H} E^H,
    $$ 
and the $K$-equivariant parameterized fibration corresponding to $p^H$ is given by $\Pi^H$.
Hence, it follows that
    $$
        \TC_K[p^H \colon E^H \to B^H] 
            = \secat_K(\Pi^H)
            \leq \secat_G(\Pi)
            = \TC_G[p \colon E \to B]
    $$
by \Cref{prop: fixed-point-and-subgroup-ineq for equi-secat}.
\end{proof}

\begin{comment}
\begin{proposition}
    Suppose $p \colon E \to B$ is a $G$-fibration. 
    If $H$ is a subgroup of $G$ such that $p$ is also a $H$-fibration, then
    $$
        \TC_H[p \colon E \to B] \leq \TC_G[p \colon E \to B].
    $$
    In particular, if $G$ is a compact Hausdorff topological group, then
    $$
        \TC[p \colon E \to B] \leq \TC_H[p \colon E \to B] \leq \TC_G[p \colon E \to B].
    $$
    for all closed subgroups $H$ of $G$.
\end{proposition}

\begin{proof} 
    Note that if $H$ is a subgroup of $G$ such that $p \colon E \to B$ is a $H$-fibration, then $\Pi \colon E^I_B \to E \times_B E$ is a $H$-fibration. 
    Hence, the result follows by applying \Cref{prop: subgroup-ineq for equi-secat} to the $G$-fibration $\Pi \colon E^I_B \to E \times_B E$ since $\TC_H[p\colon E \to B] = \secat_H(\Pi \colon E^I_B \to E \times_B E)$.
\end{proof}
\end{comment}

Applying \Cref{prop: subgroup-ineq for equi-para-tc} and \Cref{prop: fixed-point-and-subgroup-ineq for equi-secat}, we obtain the following corollary.

\begin{corollary}
\label{cor: fixed-point-and-subgroup-ineq for equi-para-tc}   
    Suppose $G$ is a compact Hausdorff topological group and $p \colon E \to B$ is a $G$-fibration. 
    Then
    \begin{enumerate}
        \item the fixed point map $p^H \colon E^H \to B^H$ is a fibration for all closed subgroups $H$ of $G$, and
        $$
            \TC[p^H \colon E^H \to B^H] \leq \TC_G[p \colon E \to B].
        $$
        \item $p \colon E \to B$ is a $K$-fibration for all closed subgroups $K$ of $G$, and 
        $$
            \TC[p \colon E \to B] \leq \TC_K[p \colon E \to B] \leq \TC_G[p \colon E \to B].
        $$
    \end{enumerate}
\end{corollary}

% \subsubsection{A cohomological lower bound}
% \hfill\\ \vspace{-0.7em}

A cohomological lower bound for the equivariant parametrized topological complexity was established by the second author in \cite[Theorem 4.5]{D-EqPTC} using Borel cohomology. 
In the following theorem, we provide an alternative cohomological lower bound based on ordinary cohomology, which should, in principle, be easier to compute.
The proof follows arguments similar to those in \cite[Theorem 4.6]{grant2019symmetrized}.

Let $E_{B,G} := (E \times_B E)/G$ and let $d_GE \subseteq E_{B,G}$ denote the image of the diagonal subspace $\Delta(E) \subseteq E \times_B E$ under the orbit map $\rho \colon E \times_B E \to E_{B,G}$.

\begin{theorem}
\label{thm: non-equi coho-lower-bound for equi-para-tc}
    Suppose $p \colon E \to B$ is a $G$-fibration. 
    If there exists cohomology classes $u_1,\dots,u_k \in H^*(E_{B,G};R)$ (for any commutative ring $R$) such that
    \begin{enumerate}
        \item $u_i$ restricts to zero in $H^*(d_GE;R)$ for $i=1,\dots,k$;
        \item $u_1 \smile \dots \smile u_k \neq 0$ in $H^*(E_{B,G};R)$, 
    \end{enumerate}
    then $\TC_G[p \colon E \to B]>k$.
\end{theorem}

\begin{proof} 
Suppose $\TC_G[p \colon E \to B] \leq k$.
Then there exists a $G$-invariant open cover $\{U_1,\dots,U_k\}$ of $E \times_B E$ such that each $U_i$ admits a $G$-equivariant section of $\Pi$.
By \Cref{thm: equivalent defn of a section of equi-para-tc}, for each $i =1, \dots, k$, there exists a $G$-homotopy $H_i \colon U_i \times I \to E \times_B E$ from the inclusion map $j_{U_i} \colon U_i \hookrightarrow E \times_B E$ to a $G$-map $f_i \colon U_i \to E \times_B E$ which takes values in $\Delta(E)$. 
Let $\overline{U_i} := \rho(U_i)$. 
As $I$ is locally compact, $H_i$ induces a homotopy $\overline{H_i} \colon \overline{U_i} \times I \to E_{B,G}$ from the inclusion map $j_{\overline{U_i}} \colon \overline{U_i} \hookrightarrow E_{B,G}$ to a map $\overline{f_i} \colon \overline{U_i} \to E_{B,G}$ which takes values in $d_GE$.
Thus, the following diagram 
    \[
% https://tikzcd.yichuanshen.de/#N4Igdg9gJgpgziAXAbVABwnAlgFyxMJZABgBoBGAXVJADcBDAGwFcYkQAdDiWmAJ0ZYwMYAFUA+lgC+IKaXSZc+QinIVqdJq3YBRccABCpAOIy5C7HgJE1xDQxZtEIKDtkaYUAObwioAGZ8EAC2SGQgOBBIahH0WIzsABYQEADWIDQO2s4AVvpcPPyCwmKSUjI0jPQARjCMAAqKViogfFheiTiy8iCBIUgATDSR0cNxCc7JaRmajux5wFAAGmY9faGI4SOIQ7PZnNy8AkIi-mXuUkA
\begin{tikzcd}
                                                                                   & d_GE \arrow[d, "j_{dX}", hook] \\
\overline{U_i} \arrow[r, "j_{\overline{U_i}}"', hook] \arrow[ru, "\overline{f_i}"] & {E_{B,G}}                   
\end{tikzcd}
    \]
    is commutative.
    Hence, by hypothesis (1), each $u_i$ restricts to zero in $H^{*}(\overline{U_i};R)$.
    By the long exact sequence of the pair $(E_{B,G},d_GE)$, there exist classes $v_i \in H^{*}(E_{B,G},\overline{U_i};R)$ such that $v_i$ maps to $u_i$ under the coboundary map $H^{*}(E_{B,G},\overline{U_i};R) \to H^{*}(E_{B,G};R)$.
    Hence, we get
    $$
        v_1 \smile \dots \smile v_k 
            \in H^{*}(E_{B,G},\cup_{i=1}^{k} \overline{U_i};R)
            = H^{*}(E_{B,G},E_{B,G};R)
            = 0.
    $$
    Thus, by the naturality of cup products, we get $u_1 \smile \dots \smile u_k = 0 \in H^{*}(E_{B,G};R)$, contradicting the hypothesis (2).
\end{proof}

When $B$ is a point, the above theorem yields a cohomological lower bound for $\TC_G(E)$, as illustrated in the example below.

\begin{example}
Let $G = \Z_2$ act on $S^{n}$ via the antipodal action.  
Then $(E \times E)/\Z_2$ is the projective product space in the sense of Davis \cite{Davis}. 
If $n$ is odd, the mod-2 cohomology ring of $(E \times E)/\Z_2$ is given by 
$$
    \frac{\Z_2[y]}{\langle y^{n+1} \rangle} \otimes \Lambda_{\Z_2}[x],
        \quad \text{with } |y|=1,\,|x|=n \text{ (cohomological degrees)},
$$ 
where $\Z_2[y]/\langle y^{n+1} \rangle$ is the mod-2 cohomology ring of $d_{\Z_2}(S^n) = \R P^n$.
If $n$ is even, the mod-2 cohomology ring of $(E \times_B E)/\Z_2$ is the same as above with the additional relation $x^2 = y^n x$, see \cite[Theorem 2.1]{Davis}. 
Hence, the cohomology classes of $(E \times E)/\Z_2$ that restrict to zero on $d_{Z_2}(S^n)$ are precisely those in $\Lambda_{\Z_2}[x]$. 
We note that:
\begin{itemize}
    \item If $n$ is odd, then $x^2 = 0$; see Equation (2.5) in \cite[Theorem 2.1]{Davis}.
    \item If $n$ is even, then $x^3 = (y^nx)x = y^n x^2 = y^n (y^n x) = y^{2n}x = 0$, since $y^{n+1} = 0$.
\end{itemize}
Hence, by \Cref{thm: non-equi coho-lower-bound for equi-para-tc}, the lower bound on $\TC_{\Z_2}(S^n)$ is 2 if $n$ is odd and 3 if $n$ is even.
We note that
$$
    \TC_{\Z_2}(S^n) = 
        \begin{cases}
            2, & \text{for }n\text{ odd},\\
            3, & \text{for }n\text{ even},
        \end{cases}
$$
see \cite[Corollary 4.4]{eqtcprodineq}.
Hence, the lower bound we obtain is optimal.
We note that we can obtain the same lower bounds using the inequality $\TC(S^n) \leq \TC_{\Z_2}(S^n)$.
\end{example}

% \subsubsection{Product Inequalities}
% \hfill\\ \vspace{-0.7em}

The product inequality for parametrized topological complexity was proved in \cite[Proposition 6.1]{farber-para-tc}. 
We now establish the corresponding equivariant analogue.

\begin{theorem}
\label{thm: prod-ineq for equi-para-tc}
Let $p_1 \colon E_1 \to B_1$ be a $G_1$-fibration and $p_2 \colon E_2 \to B_2$ be a $G_2$-fibration. 
If $(E_1 \times E_1) \times (E_2 \times E_2)$ is $(G_1 \times G_2)$-completely normal, then
\[
    \TC_{G_1 \times G_2}[p_1 \times p_2 \colon E_1 \times E_2 \to B_1 \times B_2]
        \leq \TC_{G_1}[p_1 \colon E_1 \to B_1]+\TC_{G_2}[p_2 \colon E_2\to B_2] - 1,
\]
where $G_i$ acts on $E_i \times E_i$ diagonally for $i=1,2$; and $G_1 \times G_2$ acts on $(E_1 \times E_1) \times (E_2 \times E_2)$ componentwise.
\end{theorem}

\begin{proof}
Let $\Pi_1$ and $\Pi_2$ denote the equivariant parametrized fibrations corresponding to $p_1$ and $p_2$, respectively.
Then the equivariant parametrized fibration corresponding to the $(G_1 \times G_2)$-fibration $p_1 \times p_2$ is equivalent to the product $(G_1 \times G_2)$-fibration $\Pi_1 \times \Pi_2$, see \cite[Proposition 6.1]{farber-para-tc} for this identification.
% Let $\Pi_1 \colon (E_1)^I_{B_1} \to E_1 \times_{B_1} E_1$ and $\Pi_2 \colon (E_2)^{I}_{B_2} \to E_2 \times_{B_2} E_2$ be the equivariant parametrized fibrations corresponding to $p_1$ and $p_2$, respectively. 
% If $E := E_1 \times E_2$, $B := B_1 \times B_2$ and $p := p_1 \times p_2$ is the product $(G_1 \times G_2)$-fibration, then it easily checked that
%     $$
%         E^I_B = (E_1)^I_{B_1} \times (E_2)^I_{B_2} 
%             \quad \text{and} \quad 
%         E \times_B E = (E_1 \times_{B_1} E_1) \times (E_2 \times_{B_2} E_2)
%     $$
% and the $(G_1 \times G_2)$-equivariant parametrized fibration $\Pi \colon E^{I}_B \to E \times_B E$ corresponding to $p$ is equivalent to the product $(G_1 \times G_2)$-fibration
%     $$
%     \Pi_1 \times \Pi_2 
%         \colon (E_1)^I_{B_1} \times (E_2)^{I}_{B_2} 
%         \to (E_1 \times_{B_1} E_1) \times (E_2 \times_{B_2} E_2).
%     $$
Since a subspace of a $(G_1 \times G_2)$-completely normal space is itself $(G_1 \times G_2)$-completely normal, it follows that $(E_1 \times_{B_1} E_1) \times (E_2 \times_{B_2} E_2)$ is $(G_1 \times G_2)$-completely normal. 
Hence,
    \begin{align*}
        \TC_{G_1 \times G_2}[p_1 \times p_2 \colon E_1 \times E_2 \to B_1 \times B_2] 
            & = \secat_{G_1 \times G_2}(\Pi_1 \times \Pi_2) \\
            & \leq \secat_{G_1}(\Pi_1) + \secat_{G_2}(\Pi_2) - 1\\
            & = \TC_{G_1}[p_1 \colon E_1 \to B_1]+\TC_{G_2}[p_2 \colon E_2\to B_2] -1,
    \end{align*}
by \cite[Proposition 3.7]{A-D-S}.
\end{proof}

\begin{corollary}
\label{cor: special-prod-ineq for equi-para-tc}
Suppose $p_i \colon E_i \to B_i$ is a $G$-fibration for $i =1,2$. 
If $G$ is compact Hausdorff, then $p_1 \times p_2 \colon E_1 \times E_2 \to B_1 \times B_2$ is a $G$-fibration, where $G$ acts diagonally on the spaces $E_1 \times E_2$ and $B_1 \times B_2$.
Furthermore, if 
% $E_1$ and $E_2$ are Hausdorff, and 
$E_1\times E_1 \times E_2 \times E_2$ is completely normal, then
\[
    \TC_{G}[p_1 \times p_2 \colon E_1 \times E_2 \to B_1 \times B_2]
        \leq \TC_{G}[p_1 \colon E_1 \to B_1] + \TC_{G}[p_2 \colon E_2\to B_2]-1.
\]
% where $G$ acts on $E_i \times E_i$ diagonally for $i=1,2$; and $G \times G$ acts on $(E_1 \times E_1) \times (E_2 \times E_2)$ componentwise and $p_1 \times p_2 \colon E_1 \times E_1 \to B_1 \times B_2$ is the product $(G_1 \times G_2)$-fibration.
\end{corollary}

\begin{proof}
In \Cref{prop: special-prod-ineq for equi-secat}, we showed that $p_1 \times p_2$ is a $G$-fibration. 
% If $E_1$ and $E_2$ are Hausdorff, and $(E_1\times E_1) \times (E_2 \times E_2)$ is completely normal, then 
Moreover, by \cite[Lemma 3.12]{colmangranteqtc}, it follows that $(E_1\times E_1) \times (E_2 \times E_2)$ is $(G \times G)$-completely normal. 
Hence, the desired inequality 
    \begin{align*}
        \TC_{G}[p_1 \times p_2 \colon E_1 \times E_2 \to B_1 \times B_2]
            & \leq \TC_{G \times G}[p_1 \times p_2 \colon E_1 \times E_2 \to B_1 \times B_2] \\
            & \leq \TC_{G}[p_1 \colon E_1 \to B_1]+\TC_{G}[p_2 \colon E_2\to B_2]-1
    \end{align*}
follows from \Cref{cor: fixed-point-and-subgroup-ineq for equi-para-tc} (2) and \Cref{thm: prod-ineq for equi-para-tc}.
\end{proof}

\begin{corollary}
\label{cor: equi-para-tc of pullback fibration under diagonal map}
Suppose $p_i \colon E_i \to B$ is a $G$-fibration for $i =1,2$. 
Let $E_1 \times_B E_2 = \{(e_1,e_2) \in E_1 \times E_2 \mid p_1(e_1) = p_2(e_2)\}$ and let $p\colon  E_1 \times_B E_2 \to B$ be the $G$-map given by $p(e_1,e_2)=p_1(e_1)=p_2(e_2)$, where $G$ acts on $E_1 \times_B E_2$ diagonally. 
If $G$ is compact Hausdorff, then $p$ is a $G$-fibration. 
Furthermore, if 
    % $E_1$ and $E_2$ are Hausdorff, and 
$E_1 \times E_1 \times E_2 \times E_2$ is completely normal, then
    \[
        \TC_{G}[p \colon E_1 \times_B E_2 \to B]
            \leq \TC_{G}[p_1 \colon E_1 \to B]+\TC_{G}[p_2 \colon E_2\to B]-1,
    \]
% where $G$ acts on $E_i \times E_i$ diagonally for $i=1,2$; and $G \times G$ acts on $(E_1 \times E_1) \times (E_2 \times E_2)$ componentwise.
\end{corollary}

\begin{proof}
In \Cref{cor: equi-secat of pullback fibration under diagonal map}, we established that $p$ is a $G$-fibration.
Then, by identifying $B$ with its image under the diagonal map $\Delta \colon B \to B \times B$ in the base of the fibration $p_1 \times p_2$ (see \eqref{diag: pullback under diagonal of product fibration}), the desired inequality follows from \cite[Proposition 4.6]{D-EqPTC} and \Cref{cor: special-prod-ineq for equi-para-tc}.
\end{proof}

% \begin{proof}
% Note the following diagram 
% \[
%     \begin{tikzcd}
%     E_1 \times_B E_2 \arrow[r, hook] \arrow[d, "p"'] & E_1 \times E_2 \arrow[d, "p_1 \times p_2"] \\
%     B \simeq \Delta(B) \arrow[r, hook]                               & B \times B                 
%     \end{tikzcd}
% \]
% is a pullback in the category of $G$-spaces, where $\Delta \colon B \to B \times B$ is the diagonal map.
% In \Cref{cor: equi-secat of pullback fibration under diagonal map}, we showed that $p$ is a $G$-fibration.
% Hence, the desired inequality 
% \begin{align*}
%     \TC_{G}[p \colon E_1 \times_B E_2 \to B]  
%         & \leq \TC_{G}[p_1 \times p_2 \colon E_1 \times E_2 \to B \times B] \\
%         & \leq \TC_{G}[p_1 \colon E_1 \to B]+\TC_{G}[p_2 \colon E_2\to B]-1
% \end{align*}
% follows from \cite[Proposition 4.6]{D-EqPTC} and \Cref{cor: special-prod-ineq for equi-para-tc}.
% \end{proof}

%------------------------------------------------------------------------------------------
\section{Invariant Parametrized Topological Complexity}
\label{sec: inv-para-tc}

In this section, we introduce  the main object of our study, the invariant parametrized topological complexity. 

Suppose $p \colon E \to B$ is a $G$-fibration. Define the space
$$
E^{I}_B \times_{E/G} E^{I}_B 
    := \{ (\gamma,\delta) \in E^{I}_B \times E^I_B \mid G\cdot \gamma(1) = G \cdot \delta(0)\}.
$$
Then the following diagram
\[
% https://tikzcd.yichuanshen.de/#N4Igdg9gJgpgziAXAbVABwnAlgFyxMJZABgBpiBdUkANwEMAbAVxiRAFEA9YASQF8A+gCEABAB0xeALbwBwdgHoA4nxFdeg0SD6l0mXPkIoAjOSq1GLNuv7DtukBmx4CRMsfP1mrRB262hez1nQyJTD2ovK19FJW1zGCgAc3giUAAzACcIKSRTEBwIJABmSMsfEAk0LAF2CVJxMQBjLEym+pEYAWIgkCycpAAmakKSsu82Kpq6sQaJFraOruMQagY6ACMYBgAFfRcjEAYYdJxe-tzEMgKixGGLCd8pgRW1ze29kNdfTKwkgAszjoMtlLtdRoh8lEKs9BvE+EA
\begin{tikzcd}
E^{I}_B \times_{E/G} E^{I}_B  \arrow[d, "\pi_1"'] \arrow[r, "\pi_2"] & E^{I}_B \arrow[d, "{\pi_E\, \circ\, e_0}"] \\
E^{I}_B \arrow[r, "{\pi_E\, \circ\, e_1}"]                           & E/G                                       
\end{tikzcd}
\]
is a pullback. 
For each path $\alpha \in E^{I}_B$, let $b_{\alpha}$ denote the element in $B$ such that $\alpha$ takes values in the fibre $p^{-1}(b_{\alpha})$.
Define the map
\begin{equation}\label{eq:inv-parametrized-fibration}
    \Psi \colon E^{I}_B \times_{E/G} E^{I}_B  \to E \times_{B/G} E, \quad \text{by} \quad \Psi(\gamma,\delta) = (\gamma(0),\delta(1)).
\end{equation}
The map $\Psi$ is well-defined as $\gamma(1) = g \cdot \delta(0)$ for some $g \in G$ and $\gamma,\delta \in E^{I}_B$ implies that $b_{\gamma} = g \cdot b_{\delta}$. Hence, $p(\gamma(0)) = b_{\gamma} = g \cdot b_\delta = g \cdot p(\delta(1))$ implies $(\gamma(0),\delta(1)) \in E \times_{B/G} E$.

As $E^{I}_B \times_{E/G} E^{I}_B$ and $E \times_{B/G} E$ are $(G \times G)$-invariant subsets of $E^{I}_B \times E^{I}_B$ and $E \times E$ respectively, we get a $(G \times G)$-action on $E^{I}_B \times_{E/G} E^{I}_B$ and $E \times_{B/G} E$, and $\Psi$ becomes a $(G \times G)$-equivariant map.

\begin{proposition}\label{prop:Pi-is-G2-fibration}
    If $p \colon E \to B$ is a $G$-fibration, then the map $\Psi \colon E^{I}_B \times_{E/G} E^I_B \to E \times_{B/G} E$ is a $(G \times G)$-fibration.
\end{proposition}

\begin{proof}
    Suppose $E^I_B \to E \times_B E$ is the equivariant parametrized fibration corresponding to $p$.
    Suppose $\widehat{p} \colon E^{I}_B \times E^{I}_B \to (E \times_B E) \times (E \times_B E)$ is the product $(G \times G)$-fibration. 
    %After the natural identification of 
    %$$
    %    (E \times_B E) \times (E \times_B E) = (E \times E) \times_{B \times B} (E \times E), \quad (e_1,e_2,e_3,e_4) \mapsto (e_1,e_3,e_2,e_4).
    %$$n
    Define
    $$
    S := \{(e_1,e_2,e_3,e_4) \in (E \times_B E) \times (E \times_B E) \mid (e_2, e_3) \in E \times_{E/G} E\}.
    $$
    It is readily checked that $(\gamma,\delta) \in E^I_B \times_{E/G} E^{I}_B$ if and only if $(\gamma,\delta) \in (\widehat{p})^{-1}(S)$.
    % $$
    % \Psi(\gamma,\delta) = (\gamma(0),\gamma(1), \delta(0),\delta(1)) 
    %    \in (E \times_B E) \times (E \times_B E) 
    %    \text{ satisfies } (\gamma(1), \delta(0)) \in E \times_{E/G} E.
    % $$
    Since $S$ is $(G \times G)$-invariant, it follows that the restriction
    $$
    \left.\widehat{p}\right|_{E^I_B \times_{E/G} E^{I}_B} \colon 
        E^I_B \times_{E/G} E^{I}_B \to S
    $$
    is a $(G \times G)$-fibration. 
    
    Now consider the pullback diagram
    \[
    \begin{tikzcd}
    E \times_B E \arrow[r, "\pi_2"] \arrow[d, "\pi_1"'] & E \arrow[d, "p"] \\
    E \arrow[r, "p"]                                    & B .              
    \end{tikzcd}
    \]
    As $p$ is a $G$-fibration, it follows that $\pi_1$ and $\pi_2$ are $G$-fibrations. Hence, the projection map $\pi_{1,4} := \pi_1 \times \pi_4 \colon (E \times_B E) \times (E \times_B E) \to E \times E$, given by $(e_1,e_2,e_3,e_4) \mapsto (e_1,e_4)$, is a $(G \times G)$-fibration. 
    %Note that $G \times G$ acts on $(E \times_B E) \times (E \times_B E)$ as $(g_1,g_2)(e_1,e_2,e_3,e_4) = (g_1e_1,g_1e_2,g_2e_3,g_2e_4)$ and $G \times G$ acts on $E \times E$ componentwise. 
    It is readily checked that $(e_1,e_2,e_3,e_4) \in S$ if and only if $(e_1,e_2,e_3,e_4) \in (\pi_{1,4})^{-1}(E \times_{B/G} E)$. Since $E \times_{B/G} E$ is $(G \times G)$-invariant, it follows that
    $$ 
        \left.\pi_{1,4}\right|_{S} \colon S \to E \times_{B/G} E
    $$
    is a $(G \times G)$-fibration. 
    Hence, $\Psi = \left.\pi_{1,4}\right|_{S} \circ \left.\widehat{p}\right|_{E^I_B \times_{E/G} E^{I}_B}$ is a $(G \times G)$-fibration.
\end{proof}

We now introduce the main object of our study, which is a parametrized analogue of invariant topological complexity introduced by Lubawski and Marzantowicz in \cite{invarianttc}.

\begin{definition}\label{def:inv-para-tc} 
Suppose $p\colon E \to B$ is a $G$-fibration. The invariant parametrized topological complexity, denoted by $\TC^G[p \colon E \to B]$ is defined as
\[
    \TC^{G}[p \colon E \rightarrow B] := \secat_{G\times G} \left(\Psi \right).
\]    
\end{definition}

The $G$-homotopy invariance of the invariant topological complexity was established by Lubawski and Marzantowicz in \cite[Proposition 2.4 and Lemma 3.8]{invarianttc}. We will now establish the corresponding parametrized analogue. In particular, we establish the equivariant fibrewise homotopy invariance of invariant parametrized topological complexity. We refer the reader to \cite[Section 4.1]{D-EqPTC} for basic information about fibrewise equivariant homotopy equivalence.

\begin{theorem}
\label{thm:G-homotopy-invariace-of-inv-para-tc}
    If $p \colon E \to B$ and $p' \colon E' \to B$ are $G$-fibrations which are fibrewise $G$-homotopy equivalent, then
    \[
        \TC^{G}[p \colon E \to B] = \TC^{G}[p' \colon E' \to B].
    \]
\end{theorem}

\begin{proof}
Suppose we have a fibrewise $G$-homotopy equivalence given by the following commutative diagram:
\[
    % https://tikzcd.yichuanshen.de/#N4Igdg9gJgpgziAXAbVABwnAlgFyxMJZABgBpiBdUkANwEMAbAVxiRAFEQBfU9TXfIRQAmclVqMWbdgHJuvEBmx4CRAIyk14+s1aIQAIW7iYUAObwioAGYAnCAFskZEDghINIOAAss1nEgAtJ46UvrW8jb2Toiebs7UPn4BiMHUoXog1nI8UY4Jru6Iol6+-h7pkploINQMdABGMAwACvwqQiC2WGbeAblZ0RWFSCVJ5akhVWxoORRcQA
\begin{tikzcd}
E \arrow[rr, "f", shift left] \arrow[rd, "p"', shift right] &   & E' \arrow[ll, "f'", shift left] \arrow[ld, "p'", shift left] \\
                                                            & B. &                                                             
\end{tikzcd}
\]
Suppose $\tilde{f} = f\times f$, $\tilde{f}^I(\gamma,\delta) = (f\circ \gamma,f\circ \delta)$ and $\tilde{f'}$, $(\tilde{f'})^I$ are defined similarly. 
Note that $\tilde{f}$ and $\tilde{f'}$ are $(G\times G)$-maps. 
Then we have the following commutative diagram.
\[ 
\begin{tikzcd}
E^I_B\times_{E/G} E^I_B \arrow{r}{\tilde{f}^I} \arrow[swap]{d}{\Psi}    & (E')^I_B\times_{E'/G} (E')^I_B \arrow{d}{\Psi'} \arrow{r}{(\tilde{f'})^I} & E^I_B \times_{E/G} E^I_B \arrow{d}{\Psi} \\%
E\times_{B/G} E \arrow{r}{\tilde{f}} & E'\times_{B/G} E' \arrow{r}{\tilde{f'}} & E\times_{B/G} E.
    \end{tikzcd}
\]
Since the maps $f'\circ f$ and $\mathrm{id}_E$ are fibrewise $G$-homotopy equivalent, it follows that the maps $\tilde{f'}\circ \tilde{f}$ and $\mathrm{id}_{E\times_{B/G} E}$ are $(G \times G)$-homotopy equivalent. 
Then, using \cite[Lemma 4.10 (2)]{D-EqPTC}, we obtain the inequality
\[
    \TC^G[p \colon E\to B] 
        = \secat_{G\times G}(\Psi)
        \leq \secat_{G\times G}(\Psi')
        = \TC^G[p' \colon E'\to B].
\]
Similarly, we can derive the reverse inequality, which completes the proof.   
\end{proof}

The next proposition shows that the invariant parametrized topological complexity of a $G$-fibration is a generalization of both the parametrized topological complexity of a fibration \cite{farber-para-tc} and the invariant topological complexity of a $G$-space \cite{invarianttc}.

\begin{proposition}
\label{prop: special cases of inv-para-tc}
Suppose $p \colon E \to B$ is a $G$-fibration.
\begin{enumerate}
\item If $G$ acts trivially on $E$ and $B$, then $\TC^{G}[p \colon E \rightarrow B] = \TC[p \colon E \rightarrow B]$.

\item If $B=\{*\}$, then $\TC^{G}[p \colon E \rightarrow \{*\}] = \TC^{G}(E)$.

% \item If $E = B \times F$ and $p \colon B \times F \to B$ is the trivial $G$-fibration with $G$ acting trivially on $F$, then
% $
%     \TC^G[p \colon B \times F \to B] = \TC(F).
% $
\end{enumerate}
\end{proposition}

\begin{proof} 
(1) If $G$ acts trivially on $E$, then 
% $\pi_E \colon E \to E/G$ is the identity map. Hence, 
$E^{I}_B \times_{E/G} E^{I}_B = E^{I}_B \times_{E} E^{I}_B$. 
We note that the map 
$$
    E^{I}_B \to E^{I}_B \times_{E} E^{I}_B, 
        \quad
    \alpha \mapsto \left(\left.\alpha\right|_{\left[0,\frac{1}{2}\right]},\left.\alpha\right|_{\left[\frac{1}{2},1\right]}\right)    ,
$$
is a homeomorphism, whose inverse is given by concatenation of paths.
% which is homeomorphic to $E^{I}_B$ via the map $(\gamma, \delta) \mapsto \gamma * \delta$, where $\gamma * \delta$ is the concatenation of paths $\gamma$ and $\delta$. 
% The inverse of this homeomorphism is given by $\alpha \mapsto \left(\left.\alpha\right|_{\left[0,\frac{1}{2}\right]},\left.\alpha\right|_{\left[\frac{1}{2},1\right]}\right)$ for $\alpha \in E^{I}_B$. 
If $G$ acts trivially on $B$, then 
% $\pi_B \colon B \to B/G$ is the identity map. Hence, 
$E \times_{B/G} E = E \times_B E$. 
Therefore, the corresponding fibration $\Psi$ is given by \eqref{eq: para-tc}.
% $$
%     \Psi \colon E^{I}_B \to E \times_{B} E, \quad \Psi(\alpha) = (\alpha(0),\alpha(1)).
% $$
Hence, we get $\TC^{G}[p \colon E \rightarrow B] = \TC[p \colon E \rightarrow B]$.

(2) If $B=\{*\}$, then $E^{I}_B = E^{I}$ and $E \times_{B/G} E = E \times E$. 
Hence, the corresponding fibration $\Psi$ is given by \eqref{eq: inv-tc}.
% $$
%     \Psi \colon E^{I} \times_{E/G} E^I \to E \times E, \quad \Psi(\gamma,\delta) = (\gamma(0),\delta(1)).
% $$
Therefore, $\TC^{G}[p \colon E \rightarrow \{*\}] = \TC^{G}(E)$.
%
% (3) 
% % Let $E = B \times F$. Then 
% We note that $E^I_B = B \times F^I$ and $E \times_{B/G} E = (B \times_{B/G} B) \times (F \times F)$.  
% Since $E/G = (B \times F)/G = (B/G) \times F$, it follows that
% $$
%     E^I_B \times_{E/G} E^I_B 
%         = (B \times_{B/G} B) \times (F^I \times_{F} F^I) 
%         \cong_G (B \times_{B/G} B) \times F^I,
% $$
% where the last $G$-homeomorphism is induced by $(\gamma,\delta) \in F^{I} \times_F F^{I} \mapsto \gamma*\delta \in F^{I}$.
% Consequently, the corresponding fibration $\Psi$ is given by $\Psi = \mathrm{id}_{B \times_{B/G} B} \times e_F$, where $e_F \colon F^I \to F\times F$ denotes the free path space fibration of $F$. 
% Hence,
% \[
%     \TC^G[p \colon E\to B] 
%         = \secat_{G\times G}(\Psi)
%         = \secat_{G\times G}(\mathrm{id} \times e_F)
%         = \secat(e_F)
%         %= \secat(\Pi_F)
%         = \TC(F),
% \]
% since $G$ acts trivially on $F$.
\end{proof}

\begin{proposition}
\label{prop: inv-para-TC for trivial-G-fib}
Let $p \colon B \times F \to B$ be the trivial $G$-fibration with $G$ acting trivially on $F$. Then
\[
    \TC^G[p \colon B \times F \to B] = \TC(F).
\]
\end{proposition}

\begin{proof}
Let $E = B \times F$.
Then we note that $E^I_B = B \times F^I$ and $E \times_{B/G} E = (B \times_{B/G} B) \times (F \times F)$.  
Since $E/G = (B \times F)/G = (B/G) \times F$, it follows that
$$
    E^I_B \times_{E/G} E^I_B 
        = (B \times_{B/G} B) \times (F^I \times_{F} F^I) 
        \cong_G (B \times_{B/G} B) \times F^I,
$$
where the last $G$-homeomorphism is induced by $(\gamma,\delta) \in F^{I} \times_F F^{I} \mapsto \gamma*\delta \in F^{I}$.
Consequently, the corresponding fibration $\Psi$ is given by $\Psi = \mathrm{id}_{B \times_{B/G} B} \times e_F$, where $e_F \colon F^I \to F\times F$ denotes the free path space fibration of $F$. 
Hence,
\[
    \TC^G[p \colon E\to B] 
        = \secat_{G\times G}(\Psi)
        = \secat_{G\times G}(\mathrm{id} \times e_F)
        = \secat(e_F)
        %= \secat(\Pi_F)
        = \TC(F),
\]
since $G$ acts trivially on $F$.
\end{proof}

Suppose $\mathbb{k}(E)$ is the saturation of the diagonal $\Delta(E)$ with respect to the $(G \times G)$-action on $E \times E$, i.e.,
$$
    \mathbb{k}(E) := (G \times G)\cdot \Delta(E) \subseteq E \times E.
$$
If $E \times_{E/G} E$ is the pullback corresponding to $\pi_E \colon E \to E/G$, i.e.,
\[
    E \times_{E/G} E := \{ (e_1,e_2) \in E \times E \mid \pi_E(e_1) = \pi_E(e_2)\},
\]
then it is readily checked that $\mathbb{k}(E) = E \times_{E/G} E \subseteq E \times_{B/G} E$. Hence, we will use the notation $\mathbb{k}(E)$ and $E \times_{E/G} E$ interchangeably.

In the next theorem, we establish the parametrized analogue of \cite[Lemma 3.8]{invarianttc}.

\begin{theorem}
\label{thm: equivalent defn of a section of in-para-tc}
    Suppose $p\colon E \to B$ is a $G$-fibration. For a $(G\times G)$-invariant (not necessarily open) subset $U$ of $E \times_{B/G} E$ the following are equivalent:
    \begin{enumerate}
        \item there exists a $(G \times G)$-equivariant section of $\Psi \colon E^{I}_B \times_{E/G} E^{I}_B  \to E \times_{B/G} E$ over $U$.
        \item there exists a $(G \times G)$-homotopy between the inclusion map $i_U \colon U \hookrightarrow E \times_{B/G} E$ and a $(G \times G)$-map $f \colon U \to E \times_{B/G} E$ which takes values in $E \times_{E/G} E$.
    \end{enumerate}
\end{theorem}

\begin{proof}
    $(1) \implies (2)$. 
    Suppose $\sigma = (\sigma_1,\sigma_2) \colon U \to E^{I}_B \times_{E/G} E^{I}_B$ is a $(G \times G)$-equivariant section of $\Psi$. 
    Let $H \colon (E^{I}_B \times_{E/G} E^{I}_B) \times I \to E^{I}_B \times_{E/G} E^{I}_B$ be given by
    $$
    H(\gamma,\delta,t) = (\gamma_t',\delta_t'), \quad \text{for } (\gamma,\delta) \in E^{I}_B\times_{E/G} E^{I}_B, \text{ and } t \in I,
    $$
    where $\gamma_t'(s) = \gamma(s+t(1-s))$ and $\delta_t'(s) = \delta(s(1-t))$. 
    It is clear that $\gamma_t', \delta_t' \in E^{I}_B$, and $\gamma_t'(1) = \gamma(1)$ and $\delta_t'(0) = \delta(0)$ for all $(\gamma,\delta) \in E^{I}_B\times_{E/G} E^{I}_B$ and for all $t \in I$. 
    Hence, $H$ is well-defined. 
    Clearly, $H$ is $(G\times G)$-equivariant such that $H(\gamma,\delta,0) = (\gamma, \delta)$ and $H(\gamma,\delta,1) = (c_{\gamma(1)},c_{\delta(0)})$, where $c_{e}$ is the constant path in $E$ taking the value $e \in E$. 
    Then
    $$
    F := \Psi \circ H \circ (\sigma \times \mathrm{id}_I) \colon U \times I \to E \times_{B/G} E
    $$
    is a $(G \times G)$-homotopy such that $F_0 = \Psi \circ \mathrm{id}_{E^{I}_B \times_{E/G} E^{I}_B} \circ \sigma = i_U$ and $F_1(u) = \Psi(H_1(\sigma(u))) = ((\sigma_1(u))(1),(\sigma_2(u))(0))$. 
    As $\sigma(u) = (\sigma_1(u), \sigma_2(u)) \in E^{I}_B \times_{E/G} E^{I}_B$ for all $u \in U$, it follows $F_1(u) = ((\sigma_1(u))(1),(\sigma_2(u))(0)) \in E \times_{E/G} E$. 
    Hence, $F_1$ is the desired $(G\times G)$-homotopy.
    
    $(2) \implies (1)$. 
    Suppose $H \colon U \times I \to E \times_{B/G} E$ is a $(G \times G)$-homotopy between $f$ and $i_U$. 
    Let $\sigma \colon U \to E^{I}_B \times_{E/G} E^{I}_B$ be the $(G \times G)$-map given by $\sigma(u)=(c_{\pi_1(f(u))},c_{\pi_2(f(u))})$, where $\pi_i \colon E \times_{B/G} E \to E$ is the projection map onto the $i$-th factor. 
    The map $\sigma$ is well-defined, since $f$ takes values in $E \times_{E/G} E$. 
    By the $G$-homotopy lifting property of $\Psi$, there exists a $(G \times G)$-homotopy $\widetilde{H} \colon U \times I \to E^{I}_B \times_{E/G} E^{I}_B$ such that the following diagram 
\[
% https://tikzcd.yichuanshen.de/#N4Igdg9gJgpgziAXAbVABwnAlgFyxMJZABgBpiBdUkANwEMAbAVxiRAFUACAHW7wFt4PbsGK8AviHGl0mXPkIoATOSq1GLNgFEAesACS4gPoAhYQPhHgWgPQBxcZ10HjwE5OmzseAkTIBGNXpmVkQOcyxBOE59KRkQDG8FIhVA6mDNMK0IqKsTe0ctKTUYKABzeCJQADMAJwh+JDIQHAgkFRa6LAY2AAsICABrOJr6xsQO1qQAZnSNUJAACRGQOoakf2opxFn1ELZeAAVsFbXx5u3NvcyQXmwy-jpTsfattsQrhiwwBag6OF6pRAc32YV4AHcsLA8AxYMBFh4KOIgA
\begin{tikzcd}
U \times \{0\} \arrow[d, hook] \arrow[rr, "\sigma"]             &  & E^{I}_B \times_{E/G} E^{I}_{B} \arrow[d, "\Psi"] \\
U \times I \arrow[rr, "H"] \arrow[rru, "\widetilde{H}", dashed] &  & E \times_{B/G} E                                
\end{tikzcd}
\]
    commutes. 
    Then $\Psi \circ \widetilde{H}_1 = H_1 = i_U$ implies $\widetilde{H}_1$ is a $(G\times G)$-equivariant section of $\Psi$ over $U$.
\end{proof}

\begin{corollary}\label{cor:inv-para-tc-as-A-cat}
    For a $G$-fibration $p \colon E \to B$, we have
    \[
        \TC^{G}[p \colon E \to B] = ~_{\mathbb{k}(E)}\ct_{G \times G}(E \times_{B/G} E).
    \]
\end{corollary}

%-------------------------------------------------------------------------------------
\subsection{Properties and bounds}
% \hfill\\ \vspace{-0.7em}
\label{subsec:prop-bounds}

\begin{proposition}
\label{prop: pullback-ineq for inv-para-tc}
    Suppose $p \colon E \to B$ is a $G$-fibration and $B'$ is a $G$-invariant subset of $B$. 
    If $E' = p^{-1}(B')$ and $p' \colon E' \to B'$ is the $G$-fibration obtained by restriction of $p$, then 
    $$
        \TC^{G}[p' \colon E' \to B'] \leq \TC^{G}[p \colon E \to B].
    $$
    In particular, if $b \in B^G$, then the fibre $F=p^{-1}(b)$ is a $G$-space and
    $$
        \TC^{G}(F) \leq \TC^{G}[p \colon E \to B].
    $$
\end{proposition}

\begin{proof}
Note that we have the following commutative diagram
    \[
\begin{tikzcd}
(E')^I_{B'} \times_{E'/G} (E')^I_{B'} \arrow[rr, hook] \arrow[d, "\Psi'"'] &  & E^I_B \times_{E/G} E^I_B \arrow[d, "\Psi"] \\
E' \times_{B'/G}E' \arrow[rr, hook]                                         &  & E \times_{B/G} E,                         
\end{tikzcd}
    \]
    where $\Psi'$ and $\Psi$ are the fibrations corresponding to $p'$ and $p$, respectively.
    We will now show that this diagram is a pullback.
    
    Suppose $Z$ is a topological space with $(G \times G)$-maps $k=(k_1,k_2) \colon Z \to E^I_B \times_{E/G} E^I_B$ and $h = (h_1,h_2) \colon Z \to E' \times_{B'/G} E'$ such that $\Psi \circ k = h$. 
    % Suppose $z \in Z$. 
    Then we have
    $$
        k_1(z)(0) = h_1(z) \in E' \quad \text{and} \quad k_2(z)(1) = h_2(z) \in E'.
    $$
    % As $h(z) = (h_1(z),h_2(z)) \in E' \times_{B'/G} E'$, we have 
    % $$
    %    p(h_1(z)) = p'(h_1(z)) = g_{h(z)} \cdot p'(h_2(z)) = g_{h(z)} \cdot p(h_2(z)) \in B',
    % $$
    % for some $g_{h(z)} \in G$. 
    As $k(z) = (k_1(z),k_2(z)) \in E^I_B \times_{E/G} E^I_B$, we have 
    $$
        p(k_1(z)(t))=b_{k_1(z)},\quad p(k_2(z)(t))=b_{k_2(z)}\quad \text{and} \quad k_1(z)(1) = g_{k(z)} \cdot k_2(z)(0)
    $$
    for some $b_{k_1(z)}, b_{k_2(z)} \in B$, $g_{k(z)} \in G$ and for all $t\in I$.

    Note that $b_{k_1(z)} = p(k_1(z)(t)) = p(k_1(z)(0)) = p(h_1(z))$ implies $b_{k_1(z)} \in B'$ since $h_1(z) \in E' = p^{-1}(B')$. 
    Hence, $k_1(z) \in (E')^{I}_{B'}$ since $k_1(z)(t) \in p^{-1}(b_{k_1(z)}) \subset p^{-1}(B') = E'$  for all $t \in I$.
    Similarly, $b_{k_2(z)} \in B'$ and $k_2(z) \in (E')^{I}_{B'}$. 
    Hence, $k_1(z)(1) = g_{k(z)} \cdot k_2(z)(0)$ implies $\mathrm{Im}(k) \subseteq (E')^I_{B'} \times_{E'/G} (E')^I_{B'}$. Hence, the diagram above is a pullback. Then the required inequality
    $$
        \TC^{G}[p' \colon E' \to B'] 
            = \secat_{G\times G}(\Psi') 
            \leq \secat_{G \times G}(\Psi) 
            = \TC^{G}[p \colon E \to B].
    $$
    follows from \cite[Proposition 4.3]{colmangranteqtc}.
\end{proof}

\begin{proposition}
\label{prop: inva-para-tc bounds}
Suppose $G$ is a compact Hausdorff topological group, and let $p \colon E \to B$ be a $G$-fibration.
If $e \in E^G$, then the fibre $F = p^{-1}(p(e))$ is a $G$-space, and 
    $$
        \ct_G(F) \leq \TC^G(F) \leq \TC^G[p \colon E\to B].
    $$ 
Furthermore, 
\begin{enumerate}
\item if $E\times_{B/G}E$ is $(G\times G)$-connected, then 
$$
    \TC^G[p \colon E\to B]
        \leq \ct_{G\times G}(E\times_{B/G} E).  
$$

\item if $G$ is a compact Lie group and $E \times_{B/G} E$ is a connected $(G\times G)$-$\mathrm{CW}$-complex, then
$$
    \ct_{G\times G}(E\times_{B/G} E) \leq \mathrm{dim} \left(\frac{E \times_{B/G} E}{G \times G}\right)+1.
$$
\end{enumerate}
Consequently, if $E \times_{B/G} E$ is $(G\times G)$-connected $(G\times G)$-$\mathrm{CW}$-complex, then 
$$
    \TC^G[p \colon E\to B] \leq \mathrm{dim} \left(\frac{E \times_{B/G} E}{G \times G}\right)+1.
$$
   % \[
   % \TC^G[p \colon E\to B]
   %     \leq \mathrm{dim}(E\times_{B/G} E)-\mathrm{dim}(P)+1,
   % \]
   % where $P$ is the principal orbit of $(G\times G)$-action on $E\times_{E/G} E$.
\end{proposition}

\begin{proof}
If $e \in E^G$, then $b = p(e) \in B^G$. 
Hence, by \Cref{prop: pullback-ineq for inv-para-tc}, $F := p^{-1}(b)$ admits a $G$-action and 
$
        \TC^G(F)  \leq \TC^G[p \colon E\to B].
$
Observe that $e \in F^G$. 
Therefore, the inequality $\ct_G(F)\leq \TC^G(F)$ follows from \cite[Proposition 2.7]{ZK}.

(1) Note that if $c_e$ is the constant path in $E$ which takes the value $e \in E$, then $(c_e,c_e) \in (E^I_B \times_{E/G} E^I_B)^{(G\times G)}$. 
Moreover, since $E\times_{B/G}E$ is $(G\times G)$-connected, it follows that 
$$
    \TC^G[p \colon E\to B]
        = \secat_{G \times G}(\Psi)
        \leq \ct_{G\times G}(E\times_{B/G} E).  
$$
by \cite[Proposition 4.4]{colmangranteqtc}. 

(2) Since $E \times_{B/G} E$ is connected and $(e, e) \in (E \times_{B/G} E)^{(G \times G)}$, it follows that
    $$
        \ct_{G\times G}(E\times_{B/G} E) 
            \leq \mathrm{dim} \left(\frac{E \times_B E}{G \times G}\right) + 1,
    $$
by \cite[Corollary 1.12]{M}. 
Now the last inequality follows from $(1)$ and $(2)$.
%  Note that the inequality $\ct_{G\times G}(E\times_{B/G}E)\leq \mathrm{dim}\bigg(\frac{E\times_{B/G} E}{G\times G}\bigg)+1$ follows from \cite[Corollary 1.12]{M}.
% Additionally, the equality $\mathrm{dim}\bigg(\frac{E\times_{B/G} E}{G\times G}\bigg)=\mathrm{dim}(E\times_{B/G} E)-\mathrm{dim}(P)$ is a consequence of \cite[Theorem IV.3.8]{bredon-transformation-groups}, thus establishing our desired assertion.
\end{proof}

\begin{corollary}\label{cor: G-contratibility-of-fibre}
Suppose $G$ is a compact Hausdorff topological group, and let $p \colon E\to B$ be a $G$-fibration such that $\TC^G[p \colon E \to B] = 1$. 
If $e \in E^G$, then the fibre $F = p^{-1}(p(e))$ is a $G$-contractible space.
\end{corollary}

\begin{proof}
By \Cref{prop: inva-para-tc bounds}, we have $\ct_G(F)=1$, i.e., $F$ is $G$-contractible.
\end{proof}

We now establish sufficient conditions for $\TC^G[p\colon E\to B]$ to be $1$. 
This serves as a converse of \Cref{cor: G-contratibility-of-fibre}.

\begin{theorem}
\label{thm: suff-cond for inv-para-tc = 1 with fixed point}
Suppose $G$ is a compact Hausdorff topological group.
Suppose $p \colon E\to B$ is a $G$-fibration such that $E \times_{B/G} E$ is a $G$-$\mathrm{CW}$-complex. 
Let $e \in E^G$.
If the fibre $F = p^{-1}(p(e))$ satisfies either
\begin{itemize}
    \item $F$ is $G$-connected, $G$-contractible and $F^G = \{e\}$, or
    \item $F$ strongly $G$-deformation retracts to the point $e$, 
\end{itemize}
then $\TC^G[p \colon E \to B] = 1$.
\end{theorem}

\begin{proof}
    Note that 
    $$
        \Psi^{-1}(e,e) 
            = \{(\alpha,\beta) \in E^I_B \times E^I_B \mid \alpha(0) = \beta(1) = e, \alpha(1) = g \cdot \beta(0) \text{ for some } g \in G\}.
    $$
    Since $\alpha(0) = \beta(1) = e$ and $\alpha,\beta \in E^I_B$, it follows that the fibre $\Psi^{-1}(e,e)$ is $(G \times G)$-homeomorphic to
    $$
        \mathcal{F} = \{\gamma \in F^I \mid \gamma(1/2) = e, \gamma(0) = g \cdot \gamma(1) \},
    $$
    where $(G\times G)$-action on $\mathcal{F}$ is given by 
    $$
        ((g_1,g_2)\cdot \gamma)(t) 
        =   \begin{cases}
                g_1 \cdot \gamma(t) & 0 \leq t \leq 1/2,\\
                g_2 \cdot \gamma(t) & 1/2 \leq t \leq 1.
            \end{cases}
    $$
    This action is well-defined since $\gamma(1/2) = e \in E^G$.
    
    Suppose $F$ is $G$-connected, $G$-contractible and $F^G = \{e\}$.
    Since $F$ is $G$-connected, we have $_{\{e\}}\ct_G(F) = \ct_G(F)$, see \cite[Remark 2.3]{invarianttc} and \cite[Lemma 3.14]{colmangranteqtc}. 
    Hence, $_{\{e\}}\ct_G(F) = 1$ as $F$ is $G$-contractible.
    Thus, there exists a $G$-homotopy $H \colon F \times I \to F$ such that $H(f,0)=f$ and $H(f,1)=e$ for all $f\in F$. 
    Let $K \colon F^I \times I \to F^I$ be the homotopy given by $K(\delta,t)(s) = H(\delta(s),t)$ for all $s,t \in I$ and $\delta \in F^I$. 
    Note that $K$ is a $G$-homotopy.
    If $\gamma \in \mathcal{F}$, then
    $$
        g \cdot K(\gamma,t)(1/2) = g \cdot H(\gamma(1/2),t) = g \cdot H(e,t) = H(g \cdot e, t) = H(e,t)
    $$
    for all $g \in G$, i.e., $K(\gamma,t)(1/2) \in F^G$.
    Since $F^G = \{e\}$, we get $K(\gamma,t)(1/2) = e$ for all $t\in I$.

    Suppose $F$ strongly $G$-deformation retracts to the point $e$, then there exists a $G$-homotopy $H \colon F \times I \to F$ such that $H(f,0)=f$ and $H(f,1)=e$ and $H(e,t)=e$ for all $f\in F$ and $t\in I$. 
    Then the homotopy $K$ defined on $F^I$ as above satisfies $K(\gamma,t)(1/2) =e$ due to the condition $H(e,t) =e$ for all $t \in I$.

    Moreover,
    $$
        K(\gamma,t)(0) = H(\gamma(0),t) = H(g \cdot \gamma(1),t) = g \cdot H(\gamma(1),t) = g \cdot K(\gamma,t)(1)
    $$
    where $\gamma(0) = g \cdot \gamma(1)$. 
    Hence, if $\gamma \in \mathcal{F}$, we have $K(\gamma,t) \in \mathcal{F}$. 

    Hence, in both cases, $K$ restricts to a $(G \times G)$-homotopy on $K \colon \mathcal{F} \times I \to \mathcal{F}$ such that $K(\gamma,0) = \gamma$ and $K(\gamma,1) = c_e$, where $c_e$ is the constant path in $E$ taking the value $e$. In particular, $\mathcal{F}$ is $(G\times G)$-contractible. Hence, by equivariant obstruction theory, $\Psi$ admits a $(G \times G)$-section.
\end{proof}

Later, we will also provide sufficient conditions for $\TC^G[p \colon E \to B] = 1$ and its converse when the group action on the base is free, as stated in \Cref{cor: suff-cond for inv-para-tc = 1 with free action} and \Cref{cor: subgroup-ineq for inv-para-tc and inv-para-tc = 1 implies F is contractible with free action}, respectively.

\begin{proposition}
\label{prop: subgroup-ineq for inv-para-tc}
    Suppose $p \colon E \to B$ is a $G$-fibration such that $G$ acts freely on $B$. 
    If $K$ is a subgroup of $G$ such that $p \colon E \to B$ is also a $K$-fibration, then
    $$
    \TC^{K}[p \colon E \to B] \leq \TC^{G}[p \colon E \to B].
    $$
\end{proposition}

\begin{proof}
    Suppose $\Psi_K \colon E^I_{B} \times_{E/K} E^I_{B} \to E \times_{B/K} E$ is the invariant parametrized fibration corresponding to $K$-fibration $p$. 
    Then the following diagram
    \[
    \begin{tikzcd}
    E^I_{B} \times_{E/K} E^I_{B} \arrow[d, "\Psi_K"'] \arrow[rr, hook] &  & E^I_B \times_{E/G} E^I_B \arrow[d, "\Psi"] \\
    E \times_{B/K} E \arrow[rr, hook]                                         &  & E \times_{B/G} E                         
    \end{tikzcd}
    \]
    is commutative.
    Suppose $U$ is a $(G \times G)$-invariant open subset of $E \times_{B/G} E$ with a $(G \times G)$-equivariant section $s \colon U \to E^I_B \times_{E/G} E^I_B$ of $\Psi$. 
    
    Define $V := U \cap \left(E \times_{B/K} E\right)$. 
    Then $V$ is $(K \times K)$-invariant open subset of $E \times_{B/K} E$.
    Suppose $(e_1,e_2) \in V$ and $s(e_1,e_2) = (\gamma,\delta) \in E^I_B \times_{E/G} E^I_B$. We claim that $s(e_1,e_2) = (\gamma,\delta)$ lies in $E^I_{B} \times_{E/K} E^I_{B}$. 
    Note that $p(e_1) = k \cdot p(e_2)$ for some $k \in K$, as $(e_1,e_2) \in E \times_{B/K} E$.
    Since $s$ is a section of $\Psi$, we have
    $$
    b_\gamma = p(\gamma(0)) = p(e_1) = k \cdot p(e_2) = k \cdot p(\delta(1)) = k \cdot b_\delta,
    $$
    where $\gamma(t) \in p^{-1}(b_\gamma)$ and $\delta(t) \in p^{-1}(b_\delta)$ for some $b_\gamma, b_\delta \in B$ and for all $t \in I$. Since $(\gamma,\delta) \in E^I_B \times_{E/G} E^I_B$, we have $\gamma(1) = g \cdot \delta(0)$ for some $g \in G$. 
    Hence, 
    $$
    b_\gamma = p(\gamma(1)) = p(g \cdot \delta(0)) = g \cdot p(\delta(0)) = g \cdot b_\delta.
    $$

    Thus, we get $g \cdot b_\delta = k \cdot b_\delta$. It follows that $g=k$ since $G$ acts freely on $B$. Thus, $\gamma(1) = k \cdot \delta(0)$ implies $(\gamma,\delta) \in E^I_{B} \times_{E/K} E^I_{B}$. Hence, the restriction $\left.s\right|_{V} \colon V \to E^I_{B} \times_{E/K} E^I_{B}$ is a $(K \times K)$-equivariant section of $\Psi_K$.    
\end{proof}

\begin{corollary}
\label{cor: subgroup-ineq for inv-para-tc and inv-para-tc = 1 implies F is contractible with free action}
Suppose $G$ is a compact Hausdorff topological group. 
If $p \colon E \to B$ is a $G$-fibration such that $G$ acts freely on $B$, then
$$
    \TC^{K}[p \colon E \to B] \leq \TC^{G}[p \colon E \to B]
$$
for all closed subgroups $K$ of $G$. 
In particular,
% Needed for the identity subgroup to be closed. Note this is an if and only condition.
$$
\TC(F) 
    \leq \TC[p \colon E \to B] 
    \leq \TC^{G}[p \colon E \to B],
$$
where $F$ denotes the fibre of $p$.
Moreover, if $\TC^{G}[p \colon E \to B] = 1$, then $F$ is contractible.
\end{corollary}

\begin{proof}
Note that, by \cite[Theorem 3]{gevorgyan2023equivariant}, the map $p \colon E \to B$ is a $K$-fibration.
Hence, the desired inequalities follow from \Cref{prop: subgroup-ineq for inv-para-tc} and \cite[Page 235]{farber-para-tc}.
The last statement follows from $\TC(F) = 1$ if and only if $F$ is contractible, see \cite[Theorem 1]{FarberTC}.
\end{proof}

\begin{comment}
\begin{corollary}
\label{cor: inv-para-tc = 1 implies F is contractible with free action}
Suppose $G$ is a compact Hausdorff topological group and $p \colon E \to B$ is a $G$-fibration with fibre $F$ such that $G$ acts freely on $B$. 
If $\TC^{G}[p \colon E \to B] = 1$, then the $F$ is contractible.
\end{corollary}

\begin{proof}
This follows from \Cref{cor: subgroup-ineq for inv-para-tc} and the fact that $\TC(F) = 1$ if and only if $F$ is contractible.
\end{proof}
\end{comment}
%---------------------------------------------------------------------
\subsubsection{Cohomological Lower Bounds}
% \hfill\\ \vspace{-0.7em}

\begin{lemma}
\label{lemma: homotopy equivalence of the saturation of the diagonal with total space}
    Suppose $p \colon E \to B$ is a $G$-fibration. Then the map $c \colon E \times_{E/G} E \to E^{I}_B \times_{E/G} E^{I}_B$, given by $c(e_1,e_2)=(c_{e_1},c_{e_2})$ where $c_{e_i}$ is the constant path in $E$ taking the value $e_i \in E$, is a $(G \times G)$-homotopy equivalence.
\end{lemma}

\begin{proof}
    Let $f \colon E^{I}_B \times_{E/G} E^{I}_B \to E \times_{E/G} E$ be the map given by $f(\gamma,\delta) = (\gamma(1),\delta(0))$. 
    Then $f$ is $(G \times G)$-equivariant such that $(c \circ f) (\gamma,\delta) = (c_{\gamma(1)},c_{\delta(0)})$ and $f \circ c$ is the identity map of $E \times_{E/G} E$. 
    Let $H \colon (E^{I}_B \times_{E/G} E^{I}_B) \times I \to E^{I}_B \times_{E/G} E^{I}_B$ be the homotopy given by
    $$
        H(\gamma,\delta,t) = (\gamma'_t, \delta'_t),
    $$
    where $\gamma'_t(s) = \gamma(s + t(1-s))$ and $\delta'_t(s) = \delta(s(1-t))$. Then following the proof of \Cref{thm: equivalent defn of a section of in-para-tc}, we see that $H$ is well-defined, $(G\times G)$-equivariant, $H(\gamma,\delta,0)=(\gamma,\delta)$, and $H(\gamma,\delta,1)=(c_{\gamma(1)},c_{\delta(0)})$. Hence, $c \circ f$ is $(G \times G)$-homotopic to the identity map of $E^{I}_B \times_{E/G} E^{I}_B$.
\end{proof}

Note that the following diagram
\[
% https://tikzcd.yichuanshen.de/#N4Igdg9gJgpgziAXAbVABwnAlgFyxMJZAJgBoAGAXVJADcBDAGwFcYkQBRAPWAEkBfAPoAhAAQAdcXgC28QcA4B6AOL9R3PkOEh+pdJlz5CKchWp0mrdhwlSssuPKWr1onXpAZseAkQCMpH7mDCxsiJy2MnLAwir8HDrmMFAA5vBEoABmAE4Q0kgBIDgQSKYWoewAxu5ZufmIhcVIZEX0WIzsABYQEADWIDQhVuFYAyCM9ABGMIwACgY+xiDZWCmdODUgOXmlNE2ILUNhIJKz2In8QA
\begin{tikzcd}
E \times_{E/G} E  \arrow[rr, "c"] \arrow[rd, "i"', hook] &                 & E^{I}_B \times_{E/G} E^{I}_B \arrow[ld, "\Psi"] \\
                                                         & E \times_{B/G}E &                                                
\end{tikzcd}
\]
is commutative, where $i \colon E \times_{E/G} E \hookrightarrow E \times_{B/G} E$ is the inclusion map. 
In other words, $\Psi$ is a $(G \times G)$-fibrational substitute for the $(G \times G)$-map $i$. 

For ease of notation in the upcoming theorem, let $G^2$ denote the product $G \times G$. 

\begin{theorem}
\label{thm: cohomological lower bound for inv-para-tc}
    Suppose $p \colon E \to B$ is a $G$-fibration. Suppose there exists cohomology classes $u_1,\dots,u_k \in \widetilde{H}^{*}_{G^2}(E \times_{B/G} E;R)$ (for any commutative ring R) such that
    $$
        (i^h_{G^2})^*(u_1) = \cdots = (i^h_{G^2})^*(u_k) = 0 \quad \text{and} \quad u_1 \smile \cdots \smile u_k \neq 0,
    $$
    then $\TC^G[p \colon E \to B] >k$.
\end{theorem}

\begin{proof}
Note that $\Psi \circ c = i$ implies $(c^h_{G^2})^* \circ (\Psi^h_{G^2})^* = (i^h_{G^2})^*$. 
Since $c$ is a $(G \times G)$-homotopy equivalence (see \Cref{lemma: homotopy equivalence of the saturation of the diagonal with total space}), it follows $c^h_{G^2}$ is a homotopy equivalence. 
Hence, $(c^h_{G^2})^*$ is an isomorphism. 
Thus, the result follows from \Cref{thm: cohomological lower bound for eq-secat}.
\end{proof}

% \begin{remark}
% We note that any $G$-map $f \colon X \to Y$ that is a non-equivariant homotopy equivalence induces an isomorphism on the level of Borel cohomology, see \cite{may1987characteristic}. 
% Hence, for \Cref{thm: cohomological lower bound for inv-para-tc}, we don't need $c$ to be a $(G\times G)$-homotopy equivalence, we only require $c$ to be a $(G \times G)$-map and a non-equivariant homotopy equivalence.
% \end{remark}

Now we give a non-equivariant cohomological lower bound for the invariant parametrized topological complexity.
Let $E_{B,G^2} := (E \times_{B/G} E)/(G \times G)$ and let $\mathbb{k}_{G^2}E \subseteq E_{B,G^2}$ denote the image of the saturated diagonal subspace $ \mathbb{k}(E) = E \times_{E/G} E \subseteq E \times_{B/G} E$ under the orbit map $\rho \colon E \times_{B/G} E \to E_{B,G^2}$. 
By using \Cref{thm: equivalent defn of a section of in-para-tc} and following the arguments in \Cref{thm: non-equi coho-lower-bound for equi-para-tc}, one can establish the following theorem. 
The proof is left to the reader.

\begin{theorem}
\label{thm: non-equi coho-lower-bound for inv-para-tc}
    Suppose $p \colon E \to B$ is a $G$-fibration. 
    If there exists cohomology classes $u_1,\dots,u_k \in H^*(E_{B,G^2};R)$ (for any commutative ring $R$) such that
    \begin{enumerate}
        \item $u_i$ restricts to zero in $H^*(\mathbb{k}_{G^2}E;R)$ for $i=1,\dots,k$;
        \item $u_1 \smile \dots \smile u_k \neq 0$ in $H^*(E_{B,G^2};R)$, 
    \end{enumerate}
    then $\TC^G[p \colon E \to B]>k$.
\end{theorem}

\subsubsection{Product Inequalities}
% \hfill\\ \vspace{-0.7em}

\begin{theorem}
\label{thm: prod-ineq for inv-para-tc}
% Let $G_1$ and $G_2$ be compact Lie groups (don't need it if are proving using complete normality assumption).
Let $p_1 \colon E_1 \to B_1$ be a $G_1$-fibration and $p_2 \colon E_2 \to B_2$ be a $G_2$-fibration. 
If $E_1 \times E_1 \times E_2 \times E_2$ is $(G_1 \times G_1 \times G_2 \times G_2)$-completely normal, then 
\[ 
    \TC^{G_1 \times G_2}[p_1\times p_2 \colon E_1\times E_2\to B_1\times B_2]
        \leq \TC^{G_1}[p_1 \colon E_1 \to B_1] + \TC^{G_2}[p_2 \colon E_2\to B_2]-1.
\]
\end{theorem}

\begin{proof}
    Let $\Psi_1 \colon (E_1)^I_{B_1} \times_{E_1/G_1} (E_1)^I_{B_1} \to E_1 \times_{B_1/G_1} E_1$ and $\Psi_2 \colon (E_2)^{I}_{B_2} \times_{E_2/G_2} (E_2)^I_{B_2} \to E_2 \times_{B_2/G_2} E_2$ be the invariant parametrized fibrations corresponding to $p_1$ and $p_2$, respectively. If $E := E_1 \times E_2$, $B := B_1 \times B_2$, $G := G_1 \times G_2$, and $p := p_1 \times p_2$ is the product $G$-fibration, then it easily checked that
    $$
        E^I_B \times_{E/G} E^I_B 
            = \left((E_1)^I_{B_1} \times_{E_1/G_1} (E_1)^I_{B_1}\right) \times \left((E_2)^{I}_{B_2} \times_{E_2/G_2} (E_2)^I_{B_2}\right),
    $$ % \quad \text{and} \quad 
    and
    $$
        E \times_{B/G} E = \left(E_1 \times_{B_1/G_1} E_1 \right) \times \left(E_2 \times_{B_2/G_2} E_2\right),
    $$
    and the invariant parametrized fibration $\Psi \colon E^I_B \times_{E/G} E^I_B \to E \times_{B/G} E$ corresponding to $p$ is equivalent to the product fibration $\Psi_1 \times \Psi_2$.
    % As a subspace of $(G_1 \times G_2)$-completely normal space is $(G_1 \times G_2)$-completely normal, it follows that $(E_1 \times_{B_1} E_1) \times (E_2 \times_{B_2} E_2)$ is $(G_1 \times G_2)$-completely normal. 
    Hence,
    \begin{align*}
        \TC^{G_1 \times G_2}[p_1 \times p_2 \colon E_1 \times E_2 \to B_1 \times B_2] 
            & = \secat_{(G_1 \times G_2) \times (G_1 \times G_2)}(\Psi) \\
            & = \secat_{(G_1 \times G_1) \times (G_2 \times G_2)}(\Psi_1 \times \Psi_2) \\
            & \leq \secat_{G_1 \times G_1}(\Psi_1) + \secat_{G_2 \times G_2}(\Psi_2) - 1\\
            & = \TC^{G_1}[p_1 \colon E_1 \to B_1]+\TC^{G_2}[p_2 \colon E_2\to B_2] -1,
    \end{align*}
    by \cite[Proposition 3.7]{A-D-S}.
\end{proof}

The product inequality for invariant topological complexity was proved in \cite[Theorem 3.18]{invarianttc}. In the following corollary, we show that the cofibration hypothesis assumed in \cite[Theorem 3.18]{invarianttc} can be removed by using \Cref{thm: prod-ineq for inv-para-tc}.

\begin{corollary}
    Suppose $X$ is a $G$-space and $Y$ is a $H$-space. If $X \times X \times Y \times Y$ is $(G \times G \times H \times H)$-completely normal, then
    $$
        \TC^{G \times H}(X \times Y) \leq \TC^{G}(X) + \TC^H(Y) - 1.
    $$
\end{corollary}
    
\begin{proof}
    Note that $X \to \{*_1\}$ is a $G$-fibration and $Y \to \{*_2\}$ is a $H$-fibration. 
    Hence,
    \begin{align*}
        \TC^{G \times H}(X \times Y) 
            & = \TC^{G \times H}[X\times Y \to \{*_1\} \times \{*_2\}] \\ 
            & \leq \TC^{G}[X \to \{*_1\}] + \TC^H[Y \to \{*_2\}] - 1 \\
            & = \TC^{G}(X) + \TC^H(Y) - 1,
    \end{align*}
    by \Cref{prop: special cases of inv-para-tc} and \Cref{thm: prod-ineq for inv-para-tc}.
\end{proof}

The proof of the following corollary is similar to \Cref{cor: special-prod-ineq for equi-para-tc} and can be shown using \Cref{cor: subgroup-ineq for inv-para-tc and inv-para-tc = 1 implies F is contractible with free action} and \Cref{thm: prod-ineq for inv-para-tc}. 

\begin{corollary}
\label{cor: special-prod-ineq for inv-para-tc}
    Suppose $p_i \colon E_i \to B_i$ is a $G$-fibration such that $G$ acts on $B_i$ freely for $i = 1,2$. 
    If $G$ is compact Hausdorff, then $p_1\times p_2 \colon E_1\times E_2\to B_1\times B_2$ is a $G$-fibration, where $G$ acts diagonally on the spaces $E_1 \times E_2$ and $B_1 \times B_2$.
Furthermore, if 
    % $E_1$ and $E_2$ are Hausdorff, and 
$E_1 \times E_1 \times E_2 \times E_2$ is completely normal, then
    \[ 
    \TC^{G}[p_1\times p_2 \colon E_1\times E_2\to B_1\times B_2]
        \leq \TC^{G}[p_1 \colon E_1 \to B_1] + \TC^{G}[p_2 \colon E_2\to B_2]-1.
    \]
\end{corollary}

The proof of the following corollary is similar to that of \Cref{cor: equi-para-tc of pullback fibration under diagonal map} and follows from \Cref{prop: pullback-ineq for inv-para-tc} and \Cref{cor: special-prod-ineq for inv-para-tc}.

\begin{corollary}
    Suppose $p_i \colon E_i \to B$ is a $G$-fibration, for $i =1,2$, such that $G$ acts on $B$ freely. 
    Let $E_1 \times_B E_2 = \{(e_1,e_2) \in E_1 \times E_2 \mid p_1(e_1) = p_2(e_2)\}$ and let $p\colon  E_1 \times_B E_2 \to B$ be the $G$-map given by $p(e_1,e_2)=p_1(e_1)=p_2(e_2)$, where $G$ acts on $E_1 \times_B E_2$ diagonally. 
    If $G$ is compact Hausdorff, then $p$ is a $G$-fibration. 
Furthermore, if 
    % $E_1$ and $E_2$ are Hausdorff, and 
$E_1 \times E_1 \times E_2 \times E_2$ is completely normal, then
    \[
        \TC^{G}[p \colon E_1 \times_B E_2 \to B]
            \leq \TC^G[p_1 \colon E_1 \to B]+\TC^G[p_2 \colon E_2\to B]-1.
    \]
\end{corollary}

%------------------------------------------------------------------------------------------
\subsection{Some technical results}
\hfill\\ \vspace{-0.7em}

In this subsection, we establish two technical results which will help us compute the invariant parametrized topological complexity of Fadell-Neuwirth fibrations in \Cref{sec: examples}.

\begin{definition}
%[{\cite[Section 5]{gevorgyan2023equivariant}}]
Suppose $p\colon E \to B$ is a $G$-map and $F$ is a $G$-space. We say that $p$ is a locally trivial $G$-fibration with fibre $F$ if for each point $b \in B$ there exists a $G$-invariant open subset $U$ containing $b$ and a $G$-equivariant homeomorphism $\phi: p^{-1}(U) \to U \times F$ such that the following diagram
\[
% https://tikzcd.yichuanshen.de/#N4Igdg9gJgpgziAXAbVABwnAlgFyxMJZABgBpiBdUkANwEMAbAVxiRDQD1gBaARgF8AFAFUAlCH6l0mXPkIoATOSq1GLNsIAEAHW14AtvE0AxCVPYy8BIr1K8V9Zq0QhhElTCgBzeEVAAzACcIfSQyEBwIJFtVJzZdNAALLAB9N0kA4NDEcMikJVj1FzQQagY6ACMYBgAFSzk2QKwvRJwzTJDo6jzEAscikATU3nd+IA
    \begin{tikzcd}
    p^{-1}(U) \arrow[rr, "\phi"] \arrow[rd, "p"'] &   & U \times F \arrow[ld, "\pi_1"] \\
                                                & U &                               
    \end{tikzcd}
\]
commutes, where $G$ acts on $U \times F$ diagonally. The map $\phi$ is called a $G$-trivialization of $p$.
\end{definition}

\begin{proposition}
\label{prop: quotient-fibration}
    % Suppose $p \colon E \to B$ is a surjective $G$-map and 
    Suppose $p \colon E \to B$ is a locally trivial $G$-fibration with fibre $F$. 
    If $G$ acts trivially on $F$, then the induced map $\overline{p} \colon \overline{E} \to \overline{B}$ between the orbit spaces is locally trivial with fibre $F$. 
    % Need G is a compact Lie group and B is Hausdorff for the next part
    % Furthermore, if $G$ is a compact Lie group and $B$ is paracompact Hausdorff, then $\overline{p}$ is a fibration.
\end{proposition}

\begin{proof}
    Suppose $\phi \colon p^{-1}(U) \to U \times F$ is a $G$-trivialization of $p$ over $U$. 
    As the quotient map $\pi_B \colon B \to \overline{B}$ is open, it follows $\overline{U} := \pi_B(U)$ is an open subset of $\overline{B}$. 
    Further, $U$ is $G$-invariant implies $U$ is saturated with respect to $\pi_B$.
    Hence, the restriction $\left.\pi_B\right|_{U} \colon U \to \overline{U}$ is an open quotient map and so is the product map $\left(\left.\pi_E\right|_{U}\right) \times \mathrm{id}_F \colon U \times F \to \overline{U} \times F$. 
    Hence, the induced natural map $(U \times F)/G \to \overline{U} \times F$ is a homeomorphism. 
    If $(\overline{p^{-1}(U)}) :=  \pi_E(p^{-1}(U))$, then $(\overline{p^{-1}(U)}) = (\overline{p})^{-1}(\overline{U})$ since $U$ is $G$-invariant. 
    Similarly, $\left.\pi_E\right|_{p^{-1}(U)} \colon p^{-1}(U) \to (\overline{p})^{-1}(\overline{U})$ is an open quotient map, and the induced natural map $p^{-1}(U)/G \to (\overline{p})^{-1}(\overline{U})$ is a homeomorphism.
    Hence, the homeomorphism $\phi/G \colon p^{-1}(U)/G \to (U \times F)/G$ induced by $\phi$ gives a trivialization 
    $$
    \overline{\phi} \colon (\overline{p})^{-1}(\overline{U}) \to \overline{U} \times F
    $$
    for $\overline{p}$ over $\overline{U}$. 
    As $p$ is surjective, it follows $\overline{p}$ is locally trivial with fibre $F$. 
    % If $G$ is a compact Lie group and $B$ is a paracompact Hausdorff space, by \cite[Proposition 1.4]{MR718960} and \cite[Section 31, Exercise 8]{Munkres-Topology}, it follows $\overline{B}$ is paracompact Hausdorff. Thus, $\overline{p}$ is a fibration. 
\end{proof}

    As noted in \Cref{thm: secat(overline(p)) = secat_G(p) for free actions}, the induced map $\overline{p} \colon \overline{E} \to \overline{B}$ is a $G$-fibration when $G$ is a compact Hausdorff topological group.
    However, to compute the invariant parametrized topological complexity of the equivariant Fadell-Neuwirth fibration, defined in \cite{D-EqPTC}, we require \Cref{prop: quotient-fibration}, which says $\overline{p}$ is also locally trivial.
    We will introduce equivariant Fadell-Neuwirth fibrations and calculate their invariant parametrized topological complexity in \Cref{sec: examples}. 
    Now, we present one more result which will be required in \Cref{sec: examples}.

\begin{proposition}
\label{prop: homotopy dimension of total space}
    Suppose $p \colon E \to B$ is a fibre bundle with fibre $F$, where the spaces $E,B,F$ are $\mathrm{CW}$-complexes. Then 
    $$
        \mathrm{hdim}(E) \leq \mathrm{hdim}(B) + \dim(F).
    $$
\end{proposition}

\begin{proof}
    Since $p$ is locally trivial, it follows that $\dim(E) \leq \dim(B) + \dim(F)$. 
    In particular, $\mathrm{hdim}(E) \leq \dim(B) + \dim(F)$.
    If $h \colon B' \to B$ is a homotopy equivalence and $E'$ is the pullback of $E$ along $h$, then $E'$ is a fibre bundle over $B'$ with fibre $F$. 
    Thus, we have $\dim(E') \leq \dim(B') + \dim(F)$.
    Note that $E'$ is homotopy equivalent to $E$ as $h$ is a homotopy equivalence.
    Hence, we get $\mathrm{hdim}(E) \leq \dim(B') + \dim(F)$ and the result follows.
\end{proof}
%-------------------------------------------------------------------------------------
\subsection{Invariance Theorem} 
\hfill\\ \vspace{-0.7em}
\label{subsec:invariance-thm}

Suppose $p\colon E \to B$ is a $G$-fibration such that the induced map
% Let $\overline{E}=E/G$, $\overline{B}=B/G$ and $\overline{p}=p/G$. 
% Suppose
$\overline{p} \colon \overline{E} \to \overline{B}$ between the orbit spaces is a fibration. 
If $\overline{\Pi} \colon (\overline{E})^{I}_{\overline{B}} \to \overline{E} \times_{\overline{B}} \overline{E}$ is the parametrized fibration induced by $\overline{p} \colon \overline{E} \to \overline{B}$, then we have a commutative diagram
\[
\begin{tikzcd}
E^I_B \times_{E/G} E^I_B \arrow[rr, "\Psi"] \arrow[d,swap,"f"] &  & E \times_{B/G} E \arrow[d,"\pi_E \times \pi_E"] \\
(\overline{E})^{I}_{\overline{B}} \arrow[rr,"\overline{\Pi}"]                                    &  & \overline{E} \times_{\overline{B}} \overline{E},
\end{tikzcd}
\]
where $f(\gamma,\delta) = \overline{\gamma}*\overline{\delta}$, where $\overline{\gamma} = \pi_E \circ \gamma$.

\begin{lemma}
\label{lemma: pi_E x pi_E is an open quotient map}  
    The restriction $\pi_E \times \pi_E \colon E \times_{B/G} E \to \overline{E} \times_{\overline{B}} \overline{E}$ is an open quotient map.
\end{lemma}

\begin{proof} 
As $\pi_E \colon E \to \overline{E}$ is an open quotient map, it follows $\pi_E \times \pi_E \colon E \times E \to \overline{E} \times \overline{E}$ is also an open quotient map. 
The subset $E \times_{B/G} E$ of $E \times E$ is saturated with respect to $\pi_E \times \pi_E$, since $E \times_{B/G} E$ is $(G \times G)$-invariant. 
Thus, $\pi_E \times \pi_E \colon E \times_{B/G} E \to (\pi_E \times \pi_E)(E \times_{B/G} E)$ is an open quotient map.
Note that
\begin{align*}
    (\overline{e_1},\overline{e_2}) \in \overline{E} \times_{\overline{B}} \overline{E}
        & \iff \overline{p} (\overline{e_1}) = \overline{p} (\overline{e_2}) \in \overline{B} \\
        & \iff \overline{p(e_1)} = \overline{p(e_2)} \in \overline{B} \\
        & \iff p(e_1) = g \cdot p(e_2) \text{ for some } g \in G\\
        & \iff (e_1,e_2) \in E \times_{B/G} E.
\end{align*}
Hence, the result follows.
\end{proof}

\begin{proposition} 
Suppose $p \colon E \to B$ is a $G$-fibration such that $\overline{p} \colon \overline{E} \to \overline{B}$ is a fibration. Then
    \[
    \TC[\overline{p} \colon \overline{E} \rightarrow \overline{B}]
        \leq \TC^{G}[p \colon E \rightarrow B].
    \]
\end{proposition}

\begin{proof}
    Suppose $U$ is a $(G \times G)$-invariant open subset of $E \times_{B/G} E$ with a $(G \times G)$-equivariant section $s$ of $\Psi$ over $U$. 
    Then $\overline{U} := (\pi_E \times \pi_E)(U)$ is an open subset of $\overline{E} \times_{\overline{B}} \overline{E}$, by \Cref{lemma: pi_E x pi_E is an open quotient map}. 
    As $U$ is $(G \times G)$-invariant, it follows that $U$ is saturated with respect to $\pi_E \times \pi_E$. 
    Hence, $\pi_E \times \pi_E \colon U \to \overline{U}$ is a quotient map.
    Then, by the universal property of quotient maps, there exists a unique continuous map $\overline{s} \colon \overline{U} \to E^{I}_B$ such that the following diagram 
    \[
    % https://tikzcd.yichuanshen.de/#N4Igdg9gJgpgziAXAbVABwnAlgFyxMJZABgBpiBdUkANwEMAbAVxiRAFUQBfU9TXfIRRkAjFVqMWbADrSINGACcGWMDGDsu3XiAzY8BIiPLj6zVohCz5SlWuABRLgD0AkgH0AQt3EwoAc3giUAAzRQgAWyQyEBwIJAAmajMpSxCAAlkAYyxFLPSEHlDwqMQYuKRjCXMZaTQsdwdM6TwI+Gb6xpBqBjoAIxgGAAV+AyEQRSx-AAscbWLIyuoKxCSQOwsQKDo4ab9u6tSrOQVlVXU4LR7+wZH9QTZJmbmuCi4gA
    \begin{tikzcd}
    U \arrow[r, "f \circ s"] \arrow[d, "\pi_E \times \pi_E"'] & \overline{E}^I_B \\
    \overline{U} \arrow[ru, "\overline{s}"', dashed]          &                 
    \end{tikzcd}
    \]
    commutes. 
    Then 
    $$
    \overline{\Pi} ( \overline{s} (\overline{e_1},\overline{e_2})) 
        = \overline{\Pi}(f(s(e_1,e_2))) 
        = (\pi_E \times \pi_E) (\Psi(s(e_1,e_2)))
        = (\pi_E \times \pi_E)(e_1,e_2)
        = (\overline{e_1},\overline{e_2})
    $$
    implies $\overline{s}$ is a section of $\overline{\Pi}$ over $\overline{U}$. Thus, the result follows since $\pi_E \times \pi_E \colon E \times_{B/G} E \to \overline{E} \times_{\overline{B}} \overline{E}$ is surjective.
\end{proof}

Recall that a space $X$ is called \emph{hereditary paracompact} if every subspace of $X$ is paracompact. Equivalently, every open subspace of $X$ is paracompact.

\begin{theorem}
\label{thm: inv-para-tc under free action}
Suppose $G$ is a compact Lie group. 
Let $p \colon E \to B$ be a $G$-fibration and let $\overline{p} \colon \overline{E} \to \overline{B}$ be the induced fibration between the orbit spaces. 
If the $G$-action on $E$ is free and $\overline{E} \times \overline{E}$ is hereditary paracompact, then
    \[
     \TC^{G}[p \colon E \rightarrow B]
        = \TC[\overline{p} \colon \overline{E} \rightarrow \overline{B}]. 
    \]
% Is the product of hereditary paracompact spaces, hereditary paracompact?
\end{theorem}

\begin{proof}
We note that, in \Cref{thm: secat(overline(p)) = secat_G(p) for free actions}, it was established that $\overline{p}$ is a fibration.
Suppose $\overline{U}$ is an open subset of $\overline{E} \times_{\overline{B}} \overline{E}$ with section $\overline{s}$ of $\overline{\Pi}$ over $\overline{U}$. 
Then, by \Cref{thm: equivalent defn of a section of equi-para-tc} for the trivial group action, there exists a homotopy $\overline{H} \colon \overline{U} \times I \to \overline{E} \times_{\overline{B}} \overline{E}$ such that $\overline{H}_0$ is the inclusion map of $i_{\overline{U}} \colon \overline{U} \hookrightarrow \overline{E} \times_{\overline{B}} \overline{E}$ and $\overline{H}_1$ takes values in $\Delta(\overline{E})$.

Let $U = (\pi_E \times \pi_E)^{-1}(\overline{U})$. 
Then $U$ is $(G \times G)$-invariant and $\overline{U}$ is hereditary paracompact.
Note that the following diagram
\[
% https://tikzcd.yichuanshen.de/#N4Igdg9gJgpgziAXAbVABwnAlgFyxMJZABgBpiBdUkANwEMAbAVxiRAFUACAHW7wFt4PbsGK8AviHGl0mXPkIoyARiq1GLNl14ChASSkyQGbHgJEATKVXV6zVohC8INGACcGWMDGDtxw3ThOA2lZUwUiAGZrNTtNR2dXDy8fAFF-HSxBOAB9YET3T29gACFxDO4XQpTgdMMw+XMUaMpbDQcQVICs+DySgHoAcX9UqTUYKABzeCJQADM3CH4kMhAcCCRlahw6LAY2AAsICABrepAFpc3tjcQrdXs2AApeNCwcrszs1-fUgEputlhPw6DgDm5+MAsFBxDkQkZLss7jckNEHvEnJUkkUfAAJSShC6LJGrdZIAAs212+0cR1O50RFJRiDRcQ6Pw+gPgHNG4go4iAA
    \begin{tikzcd}
    U \times \{0\} \arrow[d, hook] \arrow[rrr, hook]                  &  &                                                 & E \times_{B/G} E \arrow[d, "\pi_E \times\pi_E"] \\
    U \times I \arrow[rr, "(\pi_E \times\pi_E) \times \mathrm{id}_I"] &  & \overline{U} \times I \arrow[r, "\overline{H}"] & \overline{E} \times_{\overline{B}} \overline{E}
    \end{tikzcd}
\]
commutes.
As the $G$-action on $E$ is free, it follows the action of $G \times G$ on $E \times_{B/G} E$ and $U$ is free. 
Hence, the homotopy $\overline{H}$ preserves the orbit structure.
Since $G$ is a compact Lie group, by the Covering Homotopy Theorem of Palais \cite[Theorem II.7.3]{bredon-transformation-groups} and \Cref{lemma: pi_E x pi_E is an open quotient map}, it follows that there exists a $(G \times G)$-homotopy $H \colon U \times I \to E \times_{B/G} E$ such that $H_0 = i_{U} \colon U \hookrightarrow E \times_{B/G} E$ and $(\pi_E \times \pi_E) \circ H = \overline{H} \circ ((\pi_E \times \pi_E) \times \mathrm{id}_I)$. 
As $\overline{H}_1$ takes value in $\Delta(\overline{E})$, it follows $H_1$ takes values in $E \times_{E/G} E$. 
Hence, by \Cref{thm: equivalent defn of a section of in-para-tc}, we get a $(G \times G)$-equivariant section of $\Psi$ over $U$. 
Thus, $\TC^{G}[p \colon E \rightarrow B] \leq \TC[\overline{p} \colon \overline{E} \rightarrow \overline{B}]$.
\end{proof}

We note that the main theorem in Lubawski and Marzantowicz's paper, as stated in \Cref{thm: invariance theorem for TC}, can be recovered from \Cref{thm: inv-para-tc under free action} by taking the base space $B$ to be a point. Now, we state some corollaries of this theorem.

\begin{corollary}
\label{cor: para-tc leq para-tc of orbit fibration for free action}
Suppose $G$ is a compact Lie group. 
Let $p \colon E \to B$ be a $G$-fibration and let $\overline{p} \colon \overline{E} \to \overline{B}$ be the induced fibration between the orbit spaces. 
If the $G$-action on $B$ is free and $\overline{E} \times \overline{E}$ is hereditary paracompact, then
\[
    \TC(F) 
        \leq \TC[p \colon E \rightarrow B]
        \leq \TC^{G}[p \colon E \rightarrow B] 
        = \TC[\overline{p} \colon \overline{E} \rightarrow \overline{B}], 
\]
where $F$ is the fibre of $p$.
\end{corollary}

\begin{proof}
Observe that $G$ acts freely on $E$ as well. 
Hence, the result follows from \Cref{cor: subgroup-ineq for inv-para-tc and inv-para-tc = 1 implies F is contractible with free action} and \Cref{thm: inv-para-tc under free action}.
\end{proof}

\begin{corollary}
\label{cor: suff-cond for inv-para-tc = 1 with free action}
    Suppose $G$ is a compact Lie group and $p$ is locally trivial $G$-fibration with fibre $F$, such that $G$ acts trivially on $F$, $G$ acts freely on $B$ and $\overline{E} \times \overline{E}$ is hereditary paracompact.
    If $F$ is contractible and $\overline{E} \times_{\overline{B}}  \overline{E}$ is homotopy equivalent to a $\mathrm{CW}$-complex, then $\TC^G[p \colon E \to B] =1$.
\end{corollary}

\begin{proof}
    By \Cref{prop: quotient-fibration}, we have $\overline{p} \colon \overline{E} \rightarrow \overline{B}$ is a locally trivial fibration with fibre $F$. 
    We note that the fibre of the parametrized fibration $\overline{\Pi} \colon (\overline{E})^{I}_{\overline{B}} \to \overline{E} \times_{\overline{B}} \overline{E}$ induced by $\overline{p} \colon \overline{E} \to \overline{B}$ is the loop space $\Omega F$, which is contractible since $F$ is contractible.
    Hence, $\TC[\overline{p} \colon \overline{E} \rightarrow \overline{B}] = 1$ by obstruction theory.
    Thus, the result follows from the Invariance \Cref{thm: inv-para-tc under free action}.
\end{proof}

\begin{remark}
Suppose $p \colon B \times F \to B$ is the trivial $G$-fibration.
% In general, if $p \colon B \times F \to B$ is the trivial $G$-fibration and
If $G$ acts trivially on $F$, then in \Cref{prop: inv-para-TC for trivial-G-fib}, we showed that
$$
    \TC^G[p \colon B \times F \to B] = \TC(F).
$$
In general, however, if $G$ acts non-trivially on $F$, then the following inequality need not hold:
$$
    \TC^G[p \colon B \times F \to B] = \TC^G(F).
        % \quad \text{and} \quad
    % \TC^G[p \colon B \times F \to B] = \TC_G(F).
$$
For example, let $E=S^1 \times S^1$ and $B=S^1$. If $G=S^1$ acts on $B$ by left multiplication and diagonally on $E$, then
\begin{align*}
     \TC^{S^1}[p \colon S^1 \times S^1 \to S^1] 
        & = \TC[p/S^1 \colon (S^1 \times S^1)/S^1 \to S^1/S^1] & \text{by \Cref{thm: inv-para-tc under free action}} \\
        & = \TC[S^1 \to \{*\}] \\
        & = \TC(S^1) \\ % & \text{by \Cref{prop: special cases of inv-para-tc}}
        & = 2. 
\end{align*}
But $\TC^{S^1}(S^1) = \TC(\{*\}) = 1$ by \Cref{thm: invariance theorem for TC}. % \cite[Theorem 3.10]{invarianttc}
\end{remark}

\begin{example}
Suppose $G$ is a compact Lie group.
Let $\tau \colon P \to B$ be a principal $G$-bundle such that $\overline{P} \times \overline{P}$ is hereditary paracompact and $B$ is paracompact.
% Then $\tau$ is a $G$-fibration, by \Cref{example: G-fibrations} (2).
Let $X$ be a path-connected $G$-space, and $q \colon P \times_G X \to B$ be the associated fibre bundle with fibre $X$ and structure group $G$. 
Then, by \cite[Theorem 3.4]{Sequential-PTC-and-related-invariants}, it follows that
$$
    \TC(X) \leq \TC[q \colon P \times_G X \to B] \leq \TC_G(X),
$$
where $p \colon P \times X \to P$ is the trivial $G$-fibration.
Moreover, we have
$$
    \TC^G[p \colon P \times X \to P] = \TC[q \colon P \times_G X \to B]
$$
by \Cref{thm: inv-para-tc under free action}.
\end{example}

\begin{example} 
Now we list some examples of $G$-spaces $X$ for which $\TC(X) = \TC_G(X)$. 
Consequently, for each these examples and for any principal $G$-bundles $\tau \colon P \to B$, one has
$$
    \TC^G[p \colon P \times X \to P] = \TC(X),
$$
where $p \colon P \times X \to P$ is the trivial $G$-fibration.
\begin{enumerate}
    
\item Suppose $n_1,\dots,n_m$ are positive integers.
Let $G = \Z_2$ act on each $S^{n_i}$ via the antipodal action. 
If $G$ acts diagonally on the product $\prod_{i=1}^{m} S^{n_i}$, then by \cite[Corollary 4.4]{eqtcprodineq} we have
\begin{align*}
    \TC_G\left(\prod_{i=1}^{m} S^{n_i}\right) 
        = \TC\left(\prod_{i=1}^{m} S^{n_i}\right) 
        = m + l + 1,
\end{align*}
where $l$ denotes the number of even-dimensional spheres among the factors.

\item Suppose $n_1,\dots,n_m$ are integers with $n_i \geq 2$ for each $i$.
Let $G=\Z_2$ act on $S^{n_i}$ by reflection, defined as multiplication by $-1$ in one of the coordinates. 
If $G$ acts diagonally on the $\prod_{i=1}^{m} S^{n_i}$, then
\begin{align*}
    \TC_G\left(\prod_{i=1}^{m} S^{n_i}\right)  
        & \leq \sum_{i=1}^{m} \TC_G(S^{n_i}) - (m-1) & \text{\cite[Theorem 4.2]{eqtcprodineq}} \\
        & = 3m - (m-1) & \text{\cite[Example 5.9]{colmangranteqtc}} \\
        & = 2m+1.
\end{align*}
Moreover, it follows from \cite[Corollary 3.12]{Rudyak2014} that
$$
    \TC\left(\prod_{i=1}^{m} S^{n_i}\right) = m + l + 1,
$$
where $l$ denotes the number of even-dimensional spheres among the factors. 
Hence, if $X = \prod_{i=1}^{m} S^{n_i}$ with all $n_i$ even, then $\TC_G(X) = \TC(X) = 2m+1$.

\item Let $F$ be a $2n$-dimensional quasitoric manifold. 
The $\Z_2$-action on $F$ is described in \cite[Equation~3.2]{B-D-S}, and it is shown in \cite[Corollary~2.10 and Proposition~3.11]{B-D-S} that $\TC_{\Z_2}(F) = \TC(F) = 2n + 1$. 
We refer the reader to \cite{B-D-S} for further details.

\item Let $F = \mathrm{Gr}_d(\C^n)$ be the complex Grassmannian manifold of complex dimension $d(n - d)$. 
Then $F$ admits a $\Z_2$-action induced by complex conjugation. 
It follows from \cite[Example~6.6]{D-S} that $\TC_{\Z_2}(F) = 2d(n - d) + 1 = \TC(F)$. 
In particular, for complex projective space $\C P^n$, we have $\TC_{\Z_2}(\C P^n) = 2n + 1 = \TC(F)$. 
We refer the reader to \cite[Section 5]{D-S} for further details.
\end{enumerate}
\end{example}

\begin{example} 
Here we list some examples in which the invariant topological complexity of $G$-fibration is trivial.
\begin{enumerate}

\item Suppose $G$ is a compact Lie group.
Let $\tau \colon P \to B$ be a principal $G$-bundle such that $\overline{P} \times \overline{P}$ is hereditary paracompact and $B$ is paracompact.
% Then $\tau$ is a $G$-fibration, by \Cref{example: G-fibrations} (2). 
Then
$$
    \TC^{G}[p \colon P \to B] 
        = \TC[\overline{p} \colon \overline{P} \to B]
        = 1,
$$
since $\overline{p}$ is a homeomorphism. %, by \cite[Example 3.1.7]{dieck-algtop}.

\item 
Let $E$ and $B$ be path-connected and locally path-connected spaces such that $\overline{E} \times \overline{E}$ is hereditary paracompact.
Suppose $p \colon E \to B$ is a regular covering map and $G$ is its group of covering transformations. 
If $G$ is a compact Lie group, then
$$
    \TC^{G}[p \colon E \to B] 
        = \TC[\overline{p} \colon \overline{E} \to B]
        = 1,
$$
since $\overline{p}$ is a homeomorphism by \cite[Theorem 81.6]{Munkres-Topology}.
\end{enumerate}
\end{example}

%-----------------------------------------------------------------------------
\section{Computations for equivariant Fadell-Neuwirth fibrations}
\label{sec: examples}

In this section, we provide estimates for the invariant parametrized topological complexity of equivariant Fadell-Neuwirth fibrations. 
% We begin by defining the configuration spaces and the associated Fadell-Neuwirth fibrations, along with introducing an appropriate symmetric group action on the configuration spaces to ensure they possess an equivariant fibration structure. 
The \emph{ordered configuration space} of $s$ points on $\R^d$, denoted by $F(\R^d,s)$, is defined as
\[
    F(\R^d,s):=\{(x_1,\dots,x_s)\in (\R^d)^s \mid x_i\neq x_j ~\text{ for }~ i\neq j\}.
\]
\begin{definition}[{\cite{FadellNeuwirth}}]
\label{def:fadellNeuiwirth}
The maps 
\[
    p \colon F(\R^d,s+t)\to F(\R^d,s) \quad \text{defined by} \quad p(x_1,\dots,x_{s},y_1,\dots,y_t)=(x_1,\dots,x_s)
\]
are called Fadell-Neuwirth fibrations.    
\end{definition}

The space $F(\R^d,s+t)$ admits a free action of the permutation group $\Sigma_s$, defined as follows.
For $\sigma\in \Sigma_s$, let
\[
    \sigma\cdot (x_1,\dots,x_s,y_1,\dots,y_t)=(x_{\sigma(1)},\dots,x_{\sigma(s)},y_1,\dots,y_t).
\]
Similarly, $\Sigma_s$ acts freely on $F(\R^d,s)$ by permuting the coordinates $(x_1,\dots,x_s)$.
Notice that the map $p$ in \Cref{def:fadellNeuiwirth} is $\Sigma_s$-equivariant. 
In fact, in \cite{D-EqPTC}, it was demonstrated that this map is a $\Sigma_s$-fibration.

For the rest of the section, we will use the notation $p \colon E \to B$ for the equivariant Fadell-Neuwirth fibration.
The fibre $F$ of $p$ is the configuration space of $t$ points on $\R^d \setminus \mathcal{O}_s$, where $\mathcal{O}_s$ denotes a set of $s$ distinct points in $\R^d$.
More precisely, $F = F(\R^d \setminus \mathcal{O}_s, t)$. 

\begin{remark}
\label{rem: special cases of Fadell}
    If $t = 0$, then $p$ is the identity map. 
    In particular, $p$ is a trivial $\Sigma_s$-fibration with fibre a point $\{*\}$. 
    Hence, by \Cref{prop: inv-para-TC for trivial-G-fib}, it follows that 
    $$
        \TC^{\Sigma_s}[p \colon F(\R^d,s)\to F(\R^d, s)] 
            = \TC(\{*\}). 
    $$
    Thus, we will assume that $t \geq 1$.
    
    If $s=1$, then the permutation group $\Sigma_1$ is trivial and $F(\R^d,1) = \R^d$. 
    In particular, the fibration $p$ is trivial. 
    Hence, it follows that
    \begin{align*}
        \TC^{\Sigma_1}[p \colon F(\R^d,1+t)\to F(\R^d, 1)] 
            & = \TC[p \colon F(\R^d,1+t)\to F(\R^d, 1)] & \text{by \Cref{prop: special cases of inv-para-tc} (1)} \\
            & = \TC(F(\R^d \setminus \mathcal{O}_1, t)) & \text{by \cite[Example 4.2]{farber-para-tc}} \\
            & = \TC(F(\R^d,t+1)),
    \end{align*}
    since $F(\R^d \setminus \mathcal{O}_1, t)$ is homotopy equivalent to $F(\R^d,t+1)$, see \cite[Page 15]{fadell-husseini-conf-spaces}. 
    We note that topological complexity of configuration spaces was computed by Farber and Grant in \cite{TC-of-configuration-spaces}.
    Thus, we will assume that $s \geq 2$.
    
    If $d = 1$, then the fibre $F(\R\setminus \mathcal{O}_s,t)$ of $p$ is not path-connected.
    Hence, by \Cref{cor: subgroup-ineq for inv-para-tc and inv-para-tc = 1 implies F is contractible with free action}, it follows that
    $$
        \infty 
            = \TC(F(\R\setminus \mathcal{O}_s,t)) 
            \leq \TC^{\Sigma_s}[p \colon F(\R,s+t) \to F(\R, s)].
    $$
    Thus, we will assume that $d \geq 2$.
\end{remark}

The parametrized topological complexity of these fibrations was computed in \cite{farber-para-tc} and \cite{PTCcolfree}. 
In particular, they proved the following result: 

\begin{theorem}[{\cite[Theorem 9.1]{farber-para-tc}} and {\cite[Theorem 4.1]{PTCcolfree}}]
\label{thm: ptc-Fadell_Neuwirth}
Suppose $s\geq 2$, $t\geq 1$ and $d \geq 2$. 
Then
\[
    \TC[p \colon F(\R^d,s+t)\to F(\R^d, s)] =
\begin{cases}
    2t+s, & \text{if $d$ is odd},\\
    2t+s-1 & \text{if $d$ is even}.
\end{cases}
\]
\end{theorem}

% whereas the equivariant parametrized topological complexity was computed by the second author in \cite[Theorem 5.4]{D-EqPTC}. We will now obtain bounds on the invariant parametrized topological of these fibrations.

We will now demonstrate that the invariant parametrized topological complexity of the Fadell–Neuwirth fibrations coincides with that of the corresponding orbit fibrations. 
Furthermore, it is bounded below by the parametrized topological complexity of the Fadell–Neuwirth fibrations. 
% This implies that the complexity of the universal motion planning algorithm is greater when the order in which obstacles (mines) are placed is irrelevant. 
% This is something one would expect, since the motion planners, in a sense, satisfy an extra condition, i.e., they remain unchanged under the reordering of the obstacles (mines).
Interpreting the parametrized topological complexity of the Fadell–Neuwirth fibration as describing the motion planning problem of $t$ submarines navigating in the presence of $s$ mines whose positions are unknown, and that of the orbit Fadell–Neuwirth fibration as describing the corresponding problem when the order of the mines is irrelevant, the resulting increase in complexity is entirely consistent with intuition.
When the order of the mines is irrelevant, the motion planners must satisfy an additional constraint: they are required to remain invariant under any reordering of the mines. In other words, the planner must yield the same motion plan regardless of how the mines are permuted. This symmetry condition imposes additional structural restrictions on the class of admissible planners, thereby increasing the overall complexity of constructing such a universal algorithm.

\begin{theorem}
\label{thm: tc-Fadell-Neuwith-orbit-fib}
% Let $t\geq 1$, $s\geq 2$. 
The induced map $\overline{p} \colon \overline{F(\R^d,s+t)} \to \overline{F(\R^d, s)}$ is a locally trivial fibration with fibre $F$. 
Moreover,
    \begin{align*}
        \TC[p \colon F(\R^d,s+t) \to F(\R^d, s)] 
            & \leq \TC^{\Sigma_s}[p \colon F(\R^d,s+t)\to F(\R^d, s)] \\
            & = \TC[\overline{p} \colon \overline{F(\R^d,s+t)} \to \overline{F(\R^d, s)}]
        %\TC[\overline{p} \colon \overline{F(\R^d,s+t)} \to \overline{F(\R^d, s)}]
        %    & = \TC^{\Sigma_s}[p \colon F(\R^d,s+t)\to F(\R^d, s)] \\
        %    & \geq \TC[p \colon F(\R^d,s+t) \to F(\R^d, s)].
    \end{align*}
\end{theorem}

\begin{proof}
    We note that the equivariant Fadell-Neuwirth fibrations are locally $\Sigma_s$-trivial with $\Sigma_s$ acting trivially on the fibre $F$, see \cite[Section 5.1]{D-EqPTC}. 
    Since $F(\R^d, s)$ is a manifold, it is paracompact Hausdorff. 
    Hence, by \Cref{prop: quotient-fibration}, it follows that the induced map
    $\overline{p}$ is a locally trivial fibration with fibre $F$. 
    
    As the action of $\Sigma_s$ on $F(\R^d,s+t)$ and $F(\R^d,s)$ is free, and since $F(\R^d,s+t)$ and $F(\R^d,s)$ are manifolds, it follows that $\overline{F(\R^d,s+t)}$ and $\overline{F(\R^d,s)}$ are also manifolds. 
    % Hence, $\overline{F(\R^d,s+t)} \times \overline{F(\R^d,s+t)}$ is hereditary paracompact. 
    Thus, the result follows from \Cref{cor: para-tc leq para-tc of orbit fibration for free action}.
\end{proof}

\begin{proposition}
\label{prop: Fadell-Neuwirth-dim}
    Suppose $p \colon E \to B$ denotes the equivariant Fadell-Neuwirth fibration with fibre $F$. 
    Then 
    $\mathrm{hdim}(\overline{E} \times_{\overline{B}} \overline{E}) 
            % \leq \mathrm{hdim}(\overline{B}) + 2 \dim(F) 
            % = (d-1)(s-1) + 2dt$.
            \leq (d-1)(s-1) + 2dt$.
\end{proposition}

\begin{proof}
Since $\overline{p} \colon \overline{E} \to \overline{B}$ is a locally trivial fibration with fibre $F$, it follows that the natural map $\overline{E} \times_{\overline{B}} \overline{E} \to \overline{B}$ is also a locally trivial fibration with fibre $F \times F$. 
The homotopy dimension of the unordered configuration space $\overline{B}$ is $(d-1)(s-1)$ (see \cite[Theorem 3.13]{B-Z}).
Hence, by \Cref{prop: homotopy dimension of total space}, the desired claim follows.
\end{proof}

% We would like to thank Professor Jesus Gonzalez for pointing us to an appropriate reference on the equivariant $\mathrm{CW}$-complex structure on the ordered configuration space, from which we can determine the homotopy dimension of the unordered configuration space.
% This result is used in the previous proposition, the following theorem, and \Cref{thm:estimates-for-FN-fibrations for d=2}.

We are now ready to present our computations for the invariant parametrized topological complexity of the Fadell-Neuwirth fibrations for the case $d \geq 3$.

\begin{theorem}
\label{thm:estimates-for-FN-fibrations}
Suppose $s\geq 2$, $t\geq 1$ and $d \geq 3$. Then
\begin{equation}
\label{eq:invptc-FN-estimates}
\TC^{\Sigma_s}[p \colon F(\R^d,s+t)\to F(\R^d, s)]=
\begin{cases}
   2t+s, & \text{if $d$ is odd},\\
   \text{either } 2t+s-1 \text{ or } 2t+s & \text{if $d$ is even}.
\end{cases}   
\end{equation}  
\end{theorem}

\begin{proof}
    It suffices to show that $\TC^{\Sigma_s}[p \colon F(\R^d,s+t)\to F(\R^d, s)] \leq 2t+s$, since the inequality 
    \[
        \TC[p \colon F(\R^d,s+t)\to F(\R^d, s)] 
        \leq \TC[\overline{p} \colon \overline{F(\R^d,s+t)} \to \overline{F(\R^d, s)}]
    \]
    established in \Cref{thm: tc-Fadell-Neuwith-orbit-fib}, together with \Cref{thm: ptc-Fadell_Neuwirth} yields the desired lower bound.
    
    Observe that 
    $$
        \TC^{\Sigma_s}[p \colon F(\R^d,s+t)\to F(\R^d, s)] 
            = \TC[\overline{p} \colon \overline{F(\R^d,s+t)} \to \overline{F(\R^d, s)}]
            \leq 2t+s
    $$
    if and only if $(2t+s)$-fold fibrewise  join 
    $$
        \overline{\Pi}_{2t+s} \colon *_{2t+s} (\overline{E}^I_{\overline{B}}) \to \overline{E} \times_{\overline{B}} \overline{E}
    $$
    of the fibration $\overline{\Pi}$ admits a global section, see \cite[Theorem 3]{Sva}.  
    % Note that $F$ is $(d-2)$-connected (see discussion after the statement of Theorem 4.1 in \cite{PTCcolfree}).
    % Hence, the loop space $\Omega F$, which is the fibre of $\Pi$, is path-connected because $d \geq 3$.
    Note that the loop space $\Omega F$, which is the fibre of $\Pi$, is path-connected because $\Omega F$ is $(d-3)$-connected and $d \geq 3$.
    Therefore, the fibre $*_{2t+s}\Omega F$ of $\overline{\Pi}_{2t+s}$ is simply connected, and hence $k$-simple for all $k$.
    Thus, the obstructions to a global section of $\overline{\Pi}_{2t+s}$ lie in the cohomology groups
    $$
        H^{k+1}(\overline{E} \times_{\overline{B}} \overline{E}; \pi_{k}(*_{2t+s}\Omega F)),
    $$
    see \cite[Theorem 34.2]{Steenrod}.
    Since $*_{2t+s}\Omega F$ is $((2t+s)(d-1)-2)$-connected (see \cite[Theorem 5]{Sva} or \cite[Theorem 3.5]{grant2019symmetrized} for the homotopy connectivity of joins), it follows that $H^{k+1}(\overline{E} \times_{\overline{B}} \overline{E}; \pi_{k}(*_{2t+s}\Omega F)) = 0$ for $k \leq (2t+s)(d-1)-2$.
    
    Suppose $\mathcal{N}$ is any local coefficient system on $\overline{B}$. 
    Let $\mathcal{M}$ be the local coefficient system on $\overline{E} \times_{\overline{B}} \overline{E}$ obtained as the pullback of $\mathcal{N}$ under the fibration 
    \begin{equation}
    \label{eq: fib over bar(B)}
         F \times F \hookrightarrow \overline{E} \times_{\overline{B}} \overline{E} \to \overline{B}.
    \end{equation}
    Consider the Serre spectral sequence $E^{p,q}_r$ with local coefficients for the fibration \eqref{eq: fib over bar(B)}, see \cite[Theorem 2.9]{Siegel-Spectral-sequence}.
    Then $E_2^{p,q} = H^p(\overline{B}, H^{q}(F \times F, \mathcal{M}))$, and the spectral sequence converges to $H^{p+q}(\overline{E} \times_{\overline{B}} \overline{E},\mathcal{M})$.
    If $p > \mathrm{hdim}(\overline{B})$ or $q > \mathrm{hdim}(F \times F)$, then $H^p(\overline{B}, H^{q}(F \times F, \mathcal{M})) = 0$. 
    Hence, 
    \begin{align*}
        H^{p+q}(\overline{E} \times_{\overline{B}} \overline{E},\mathcal{M}) = 0 
        \text{ if } p+q & > \mathrm{hdim}(\overline{B}) + \mathrm{hdim}(F \times F)\\
                        & = (d-1)(s-1)+2(d-1)t\\
                        & = (2t+s-1)(d-1),
    \end{align*}
    see \cite[Equation 4.1]{PTCcolfree} for the homotopy dimension of $F$. 
    Since $d \geq 3$, the space $F \times F$ is simply connected.
    Consequently, the induced map $\pi_1(\overline{E} \times_{\overline{B}} \overline{E}) \to \pi_1(\overline{B})$ is an isomorphism.
    This can be deduced from the long exact sequence of homotopy groups corresponding to fibration \eqref{eq: fib over bar(B)}.
    This implies that every local coefficient on $\overline{E} \times_{\overline{B}} \overline{E}$ is a pullback of a local coefficient system on $\overline{B}$.
    As a result, we have
    $
        H^{k+1}(\overline{E} \times_{\overline{B}} \overline{E}; \pi_{k}(*_{2t+s}\Omega F)) = 0
    $
    for $k > (2t+s-1)(d-1)-1 = (2t+s)(d-1)-d$.
    Thus, the obstruction classes vanishes for all $k$.
\end{proof}

% We would like to thank Professor Mark Grant for suggesting the use of obstruction theory in the above proof to improve our earlier bound, and for explicitly explaining how to apply the cohomological dimension of the unordered configuration space in the Serre spectral sequence to show that the obstructions vanish.

\begin{comment}
\begin{remark}
    For $d=2$, the induced map $\pi_1(\overline{E} \times_{\overline{B}} \overline{E}) \to \pi_1(\overline{B})$ will be surjective but not necessarily an isomorphism, as the fibre $F$ is not simply connected.
    However, the space $*_{2t+s}\Omega F$ remains simply connected, since it is $(2t+s-2)$-connected.
    Thus, it is enough to show that the following cohomology groups with local coefficients
    \begin{equation}
    \label{eq: local cohomology groups}
    H^{k+1}(\overline{E} \times_{\overline{B}} \overline{E}; \pi_{k}(*_{2t+s}\Omega F)) = 0
    \end{equation}
for all 
\begin{align*}
    2t+s - 2 < k 
                    & \leq \mathrm{hdim}(\overline{E} \times_{\overline{B}} \overline{E}) - 1\\
                    & \leq (d-1)(s-1)+2d t - 1 & \text{by \Cref{prop: Fadell-Neuwirth-dim} (3)}\\
                    & = 4t+s-2.
\end{align*}
\end{remark}
\end{comment} 

We now turn our attention to the case $d=2$.
For $d=2$, we note that the spaces $E$ and $B$ are aspherical, see \cite[Lemma 3.4]{grant2018artin-braid}.
Since the maps $E \to \overline{E}$ and $B \to \overline{B}$ are covering maps, it follows that $\overline{E}$ and $\overline{B}$ are also aspherical, as $E \to \overline{E}$ and $B \to \overline{B}$ are covering maps.
Therefore, we can apply the techniques developed by Grant in \cite{Grant-ParaTC-group} to compute $\TC[\overline{p} \colon \overline{E} \to \overline{B}]$, as the map $\overline{p} \colon \overline{E} \to \overline{B}$ is a fibration of aspherical spaces with path-connected fibre $F$.
More precisely, this is equivalent to computing the topological complexity of the group epimorphism $\overline{p}_* \colon \pi_1(\overline{E}) \to \pi_1(\overline{B})$.
We begin by recalling some definitions that will be useful in the discussion.
% [{\cite[Definition 3.1]{Grant-ParaTC-group}}]
Recall that a pointed map $f \colon X \to Y$ is said to realize a group homomorphism $\rho \colon G \to H$ if $X$ is a $K(G, 1)$ space, $Y$ is a $K(H, 1)$ space, and $f$ induces the homomorphism $\rho$ on fundamental groups.

\begin{definition}[{\cite{davis2000poincare}}]
    Suppose $G$ is a group and $K(G,1)$ is the corresponding Eilenberg-MacLane space.
    \begin{itemize}
        \item Then $G$ is said to be of type $F$ if $K(G,1)$ is homotopy equivalent to a finite $\mathrm{CW}$-complex.
        \item Then $G$ is said to be of type $FP$ if there exists a finite length resolution of $\Z$ by finitely generated projective $\Z G$-modules.
    \end{itemize}
\end{definition}

We note that a group of type $F$ is of type $FP$, see \cite[Page 171]{davis2000poincare}.

\begin{definition}[{\cite[Theorem 3.6]{davis2000poincare}}]
    A group $G$ is said to be a duality group of cohomological dimension $n$ if $G$ is of type $FP$, and there exists a $\Z G$-module $D$ such that 
    $$
    H^{i}(G,\Z G) \cong
    \begin{cases}
        0 & \text{for }i \neq n,\\
        D & \text{for }i = n.
    \end{cases}
    $$
    If $D$ can be chosen to have underlying abelian group $\Z$, then $G$ is said to be a Poincar\'{e} duality group of cohomological dimension $n$.
\end{definition}

    We now use arguments presented in \cite[Theorem 3]{wall-poincare-duality} to prove the following lemma.
    This lemma, in a certain sense, generalizes part of \cite[Theorem 4.3]{davis2000poincare} from Poincaré duality groups to duality groups.

\begin{lemma}
\label{lemma: extension of duality groups converse}
    Suppose $1 \to H \to G \to K \to 1$ is a short exact sequence of groups, where $H$ is a Poincar\'{e} duality group of cohomological dimension $h$.
    Then $H^{p}(K,\Z K) = 0$ if and only if $H^{p+h}(G,\Z G) = 0$.
    In particular, if $K$ is of type $FP$ and $G$ is a duality group of cohomological dimension $h+k$, then $K$ is a duality group of cohomological dimension $k$.
\end{lemma}

\begin{proof}
    Consider the Lyndon-Hochschild-Serre spectral sequence corresponding to the extension $1 \to H \to G \to K \to 1$ with coefficients in $\Z G$. 
    Then $E_2^{p,q} = H^p(K;H^q(H;\Z G))$, and the spectral sequence converges to $H^{p+q}(G;\Z G)$.
    As $\Z G$ is a free $\Z H$-module and $H$ is a Poincar\'{e} duality group of dimension $h$, it follows that
    $$
        H^{q}(H;\Z G) 
            \cong H^{q}(H;\Z H) \otimes_{\Z H} \Z G
            \cong
            \begin{cases}
                0 & \text{if } q \neq h,\\
                \Z \otimes_{\Z H} \Z G \cong \Z K & \text{if } q = h,
            \end{cases}
    $$
    as $K$-modules. 
    This implies
    $$
        E_2^{p,q}
            = H^{p}(K;H^{q}(H;\Z G)) 
            \cong 
            \begin{cases}
                0 & \text{if } q \neq h,\\
                H^{p}(K; \Z K) & \text{if } q = h.
            \end{cases}
    $$
% Observe that all lattice points on the page $E_2$ that do not lie on the vertical line $q = h$ are zero. 
Since the differential $d_r$ maps $E_r^{p,q}$ to $E_r^{p+r,q-r+1}$, it follows that $H^{p}(K;H^{q}(H;\Z G)) = E_2^{p,q} = E_{\infty}^{p,q} = H^{p+q}(G;\Z G)$.
In particular, $H^p(K;\Z K)=0$ if and only if $H^{p+h}(G, \Z G) = 0$.
\end{proof}

\begin{theorem}
\label{thm:estimates-for-FN-fibrations for d=2}
Suppose $s\geq 2$ and $t\geq 2$. 
Then
\begin{equation}
\label{eq:invptc-FN-estimates for d=2}
    \TC^{\Sigma_s}[p \colon F(\R^2,s+t)\to F(\R^2, s)] 
        =  2t+s-1.
\end{equation}  
\end{theorem}

\begin{proof}
    Suppose $B_{s+t}$ is the braid group on $(s+t)$-strands.
    Then, by \cite[Lemma 3.4]{grant2018artin-braid}, the fundamental group of $F(\mathbb{R}^2,s+t)/\Sigma_s$ is 
    $$
        B_{s+t}^{\Sigma_s} := \pi^{-1}(\Sigma_s \times \{1\}^t),
    $$
    where $\pi \colon B_{s+t} \to \Sigma_{s+t}$ is the canonical map.
    Moreover, by \cite[Lemma 3.6]{grant2018artin-braid}, we have $B_{s+t}^{\Sigma_s}$ is a duality group of cohomological dimension $s+t-1$. 

    We note that the fibration $\overline{p} \colon \overline{F(\R^2,s+t)} \to \overline{F(\R^2, s)}$ realizes the group epimorphism $\rho \colon B_{s+t}^{\Sigma_s} \twoheadrightarrow B_s$ which forgets the last $t$ strands.
    Hence, it follows that
    \begin{align*}
    \TC^{\Sigma_s}[p \colon F(\R^2,s+t)\to F(\R^2, s)] 
        & =  \TC[\overline{p} \colon \overline{F(\R^2,s+t)} \to \overline{F(\R^2, s)}] \\
        & = \TC[\rho \colon B_{s+t}^{\Sigma_s} \twoheadrightarrow B_{s}],
    \end{align*}
    by \Cref{thm: inv-para-tc under free action} and \cite[Proposition 3.5]{Grant-ParaTC-group}.
    
    If $Z$ is the centre of $B_{s+t}$, then $Z$ is infinite cyclic and the centre of $B_{s+t}^{\Sigma_s}$ is $Z$ as well, see \cite[Lemma 3.7]{grant2018artin-braid}.
    Hence, by \cite[Theorem 3.5]{Grant-ParaTC-group}, it follows that
    $$
    \TC[\rho \colon B_{s+t}^{\Sigma_s} \twoheadrightarrow B_{s}] 
        \leq \mathrm{cd}\left(\frac{B_{s+t}^{\Sigma_s} \times_{B_s} B_{s+t}^{\Sigma_s}}{\Delta(Z)}\right)+1,
    $$
    where $\mathrm{cd}$ denotes the cohomological dimension of a group. 

Suppose $\overline{P}_{t,s} = \pi_1(F)$. 
Then we have an extension
$$
    1 \to \overline{P}_{t,s} \to B_{s+t}^{\Sigma_s} \xrightarrow{\rho} B_{s} \to 1
$$
corresponding to the fibration $F \hookrightarrow \overline{E} \xrightarrow{\overline{p}} \overline{B}$.
Pulling back this extension by $\rho \colon B_{s+t}^{\Sigma_s} \to B_s$ we get an extension
$$
    1 \to \overline{P}_{t,s} \to B_{s+t}^{\Sigma_s} \times_{B_s} B_{s+t}^{\Sigma_s} \to B_{s+t}^{\Sigma_s} \to 1.
$$
Taking the quotient of $B_{s+t}^{\Sigma_s}$ by $Z$ gives an extension
\begin{equation}
\label{eq: SES final}
    1 \to \overline{P}_{t,s} \to \frac{B_{s+t}^{\Sigma_s} \times_{B_s} B_{s+t}^{\Sigma_s}}{\Delta(Z)} \to \frac{B_{s+t}^{\Sigma_s}}{Z} \to 1.
\end{equation}
We note that $\overline{P}_{t,s}$ and $P_{s+t}$ are duality groups with $\mathrm{cd}(\overline{P}_{t,s}) = t$ and $\mathrm{cd}(P_{s+t}) = s+t-1$, see \cite[Lemma 3.6]{grant2018artin-braid}, where $P_{s+t}$ is the pure braid group on $(s+t)$-strands.

We will now show that $B_{s+t}^{\Sigma_s}/Z$ is a duality group with $\mathrm{cd}(B_{s+t}^{\Sigma_s}/Z) = s+t-2$.
As $t \geq 2$, the inclusion $Z \hookrightarrow B_{s+t}^{\Sigma_s}$ splits.
This can be seen geometrically by the projection $B_{s+t}^{\Sigma_s} \to P_2$ obtained by forgetting the first $s+t-2$ strands, where $P_2$ is the pure braid group on $2$ strands.
Hence,
$$
    B_{s+t}^{\Sigma_s} \simeq Z \times \left(B_{s+t}^{\Sigma_s}/Z\right).
$$
Note that the space $F(\mathbb{R}^d,s+t)$ contains a finite $\mathrm{CW}$-complex $C$
% of dimension $(d-1)(s+t-1)$
which is a $\Sigma_{s+t}$-equivariant strong deformation retract, see \cite[Theorem 3.13]{B-Z}.
Hence, $C$ is also a $\Sigma_{s}$-equivariant strong deformation retract of $F(\mathbb{R}^d,s+t)$.
Hence, the homotopy equivalence $K(B_{s+t}^{\Sigma_s},1) \simeq_h F(\mathbb{R}^2,s+t)/\Sigma_s \simeq_h C/\Sigma_s$ implies $K(B_{s+t}^{\Sigma_s},1)$ is homotopy equivalent to a finite $\mathrm{CW}$-complex. 
Therefore, the homotopy equivalence
$$
    K(B_{s+t}^{\Sigma_s},1) \simeq_h K(Z,1) \times K(B_{s+t}^{\Sigma_s}/Z,1)
$$
implies $K(B_{s+t}^{\Sigma_s}/Z,1)$ is homotopy equivalent to a finite $\mathrm{CW}$-complex, i.e., $B_{s+t}^{\Sigma_s}/Z$ is a group of type $F$. 
Then \Cref{lemma: extension of duality groups converse} applied to the extension $1 \to Z \hookrightarrow B_{s+t}^{\Sigma_s} \to B_{s+t}^{\Sigma_s}/Z \to 1$ implies that $B_{s+t}^{\Sigma_s}/Z$ is a duality group of cohomological dimension $s+t-2$.
Hence, by \cite[Theorem 3.5]{bieri-duality-groups}, it follows that the middle group in \eqref{eq: SES final} has cohomological dimension $s+2t-2$.
\end{proof}

\begin{remark*}
In this remark, we present an argument showing that $K(B_{s+t}^{\Sigma_s}/Z,1)$ is homotopy equivalent to a finite $\mathrm{CW}$-complex, which was inadvertently left out of the proof above and the published version of this paper.
Since the extension $1 \to Z \hookrightarrow B_{s+t}^{\Sigma_s} \to B_{s+t}^{\Sigma_s}/Z \to 1$ splits, it follows that $K(B_{s+t}^{\Sigma_s}/Z,1)$ is dominated by the finite CW complex $K(B_{s+t}^{\Sigma_s},1)$, that is, there exist maps 
$$
    K(B_{s+t}^{\Sigma_s}/Z,1) 
        \xrightarrow{\sigma} K(B_{s+t}^{\Sigma_s},1) 
        \xrightarrow{p} K(B_{s+t}^{\Sigma_s}/Z,1)
$$
such that $p \circ \sigma \simeq \mathrm{id}_{K(B_{s+t}^{\Sigma_s}/Z,1)}$.
% Hence, by Proposition A.11 in the book Algebraic Topology by Allen Hatcher, it follows that it is homotopy equivalent to a CW-complex.
Note that $\mathrm{hdim}(K(B_{s+t}^{\Sigma_s},1)) = \mathrm{hdim}(\overline{F(\R^2,s+t)}) \leq s+t-1 \leq 3$, see \cite[Theorem 3.13]{B-Z}.
Hence, it follows that $H^{s+t}(K(B_{s+t}^{\Sigma_s}/Z,1)) = 0$, since $\sigma^* \colon H^*(K(B_{s+t}^{\Sigma_s},1)) \to H^*(K(B_{s+t}^{\Sigma_s}/Z,1))$ is surjective.
We note that the universal cover of $K(B_{s+t}^{\Sigma_s},1)$ is a CW complex of dimension $s+t-1$, see Exercise 11-8 (Page 303) in Introduction to Topological Manifolds (Second Edition) by John M. Lee.
Hence, by the K\"unneth formula, it follows that 
$$
    H_i(\widetilde{K(B_{s+t}^{\Sigma_s}/Z,1)}) = 0
        \text{ for all } i > s+t-1,
$$ 
where $\widetilde{K(B_{s+t}^{\Sigma_s}/Z,1)}$ is the universal cover of $K(B_{s+t}^{\Sigma_s}/Z,1)$.
Hence, by Theorem E in Finiteness Conditions for CW-Complexes by C.T.C. Wall, it follows that $K(B_{s+t}^{\Sigma_s}/Z,1)$ is homotopy equivalent to a finite CW-complex.
This argument, which shows that $K(B_{s+t}^{\Sigma_s}/Z,1)$ is homotopy equivalent to a finite $\mathrm{CW}$-complex, is adapted from the arguments found in \href{https://mathoverflow.net/questions/267669}{https://mathoverflow.net/questions/267669}. 

\end{remark*}

\begin{remark}
    Note that when $t=1$, the group $B_{s+1}^{\Sigma_s}/Z$ is not torsion free, 
    as shown in \cite[Proposition 4.2]{grant2018artin-braid}.
    Consequently, every $\mathrm{CW}$-complex of type $K(B_{s+1}^{\Sigma_s}/Z,1)$ is infinite dimensional.
    Therefore, the argument used in the preceding theorem fails for the case $t=1$.
\end{remark}

\begin{conjecture}
\label{conj}
    $
        \TC^{\Sigma_s}[F(\R^2,s+1) \to F(\R^2, s)] \leq 3+s.
    $
\end{conjecture}

    We expect the above conjecture to be true, since the motion planning problem for two submarines constrained by the unknown positions of $s$ mines (where the order in which the mines are placed does not matter) should be more complex than that for a single submarine.
    Hence, 
    $$
        \TC[\overline{F(\R^2,s+1)} \to \overline{F(\R^2, s)}] 
            \leq \TC[\overline{F(\R^2,s+2)} \to \overline{F(\R^2, s)}]
            = 3+s,
    $$
    \Cref{conj} is true if $s=1$, see \Cref{rem: special cases of Fadell} and \cite[Theorem 1]{TC-of-configuration-spaces}.
    
    We note that we can get $\TC[\overline{F(\R^2,s+1)} \to \overline{F(\R^2, s)}] \leq 4+s$, using \cite[Proposition 7.2]{farber-para-tc} and \Cref{prop: Fadell-Neuwirth-dim}.
    Moreover, $\TC[\overline{F(\R^2,s+1)} \to \overline{F(\R^2, s)}] \geq 1+s$, by \Cref{thm: ptc-Fadell_Neuwirth}.

\begin{comment}
    Note that we have subgroup inclusions
    $$
        P_{s+t}/Z \subseteq B_{s+t}^{\Sigma_s}/Z \subseteq B_{s+t}/Z,
    $$
    such that the index $[B_{s+t}^{\Sigma_s}/Z \colon P_{s+t}/Z]$ is finite.
    We know that $P_{s+t}/Z$ is a duality group of cohomological dimension $s+t-2$, see \cite[Theorem 6.2]{Grant-ParaTC-group}.
    Thus, if we can show that $B_{s+t}^{\Sigma_s}/Z$ is torsion-free, then, by \cite[Theorem 3.3]{bieri-duality-groups}, it will follow that $B_{s+t}^{\Sigma_s}/Z$ is a duality group of cohomological dimension $s+t-2$.

    \textcolor{blue}{Using  you wanted to say? If we put $n=s+1$ and $k=1$, then $B_{s+1}^{\Sigma_s}/Z$ has torsion, because we have $(s+1,1)=(s,1)=1$ But $(s-1,0)=s-1$.}
\end{comment}

\section{Acknowledgements}
% The authors thank both the reviewers for their insightful suggestions and feedback, which have improved the article, particularly for recommending the correct formulation of  \Cref{prop: subgroup-ineq for equi-cat} and suggesting the inclusion of additional examples.
We thank Professor Mark Grant for his valuable feedback on an earlier version of this article, which significantly improved the paper in various aspects. 
% In particular, we are grateful to him for suggesting an approach to establish a cohomological lower bound via the cohomology of orbit spaces, and for helping us prove \Cref{thm:estimates-for-FN-fibrations} without any additional assumptions.
In particular, we are grateful to him for suggesting an approach to establish a cohomological lower bound via the cohomology of orbit spaces, for recommending the use of obstruction theory in \Cref{thm:estimates-for-FN-fibrations} to improve our earlier bound, and for clearly explaining how to apply the cohomological dimension of the unordered configuration space in the Serre spectral sequence to show that the obstructions vanish.
% We also thank Professor Jesus Gonzalez for pointing us to an appropriate reference concerning the homotopy dimension of the unordered configuration space.
We also thank Professor Jesus Gonzalez for pointing us to an appropriate reference on the equivariant $\mathrm{CW}$-complex structure on the ordered configuration space, from which we can determine the homotopy dimension of the unordered configuration space.
% This result is used in the previous proposition, the following theorem, and \Cref{thm:estimates-for-FN-fibrations for d=2}.
The first author would like to acknowledge IISER Pune - IDeaS Scholarship and Siemens-IISER Ph.D. fellowship for economical support. 
The second author acknowledges the support of NBHM through grant 0204/10/(16)/2023/R\&D-II/2789 and DST–INSPIRE Faculty Fellowship (Faculty Registration No. IFA24-MA218), Department of Science and Technology, Government of India, as well as support from the Industrial Consultancy and Sponsored Research (IC\&SR), Indian Institute of Technology Madras for the New Faculty Initiation Grant (RF25261395MANFIG009294).
% We would like to thank Professor Mark Grant for suggesting the use of obstruction theory in the above proof to improve our earlier bound, and for explicitly explaining how to apply the cohomological dimension of the unordered configuration space in the Serre spectral sequence to show that the obstructions vanish.
\bibliographystyle{plain} 
\bibliography{references}

\end{document}